\pdfoutput=1
\RequirePackage{ifpdf}
\ifpdf 
\documentclass[pdftex]{sigma}
\else
\documentclass{sigma}
\fi

\usepackage{mathrsfs,tikz, multirow,caption,subcaption, soul,multicol}

\graphicspath{{./Diagrams/} }

\usetikzlibrary{arrows.meta}
\input{xy}
\xyoption{all}

\theoremstyle{plain}
\newtheorem{Thm}{Theorem}[section]
\newtheorem{Prop}[Thm]{Proposition}
\newtheorem{Cor}[Thm]{Corollary}
\newtheorem{Lem}[Thm]{Lemma}
\newtheorem{Conj}[Thm]{Conjecture}

\theoremstyle{definition}
\newtheorem{Def}[Thm]{Definition}
\newtheorem{Exa}[Thm]{Example}
\newtheorem{Rmk}[Thm]{Remark}

\def\a{\alpha}
\def\b{\beta}

\def\g{\gamma}

\def\s{\sigma}

\def\r{\rho}
\def\D{\Delta}
\def\G{\Gamma}

\def\S{\Sigma}

\def\m{\mu}

\def\R{\mathbb R}

\def\Z{\mathbb Z}

\def\cA{\mathcal{A}}
\def\cD{\mathcal{D}}
\def\cE{\mathcal{E}}

\numberwithin{equation}{section}

\begin{document}

\newcommand{\arXivNumber}{2006.00627}

\renewcommand{\thefootnote}{}

\renewcommand{\PaperNumber}{010}

\FirstPageHeading

\ShortArticleName{$C$-Vectors and Non-Self-Crossing Curves for Acyclic Quivers of Finite Type}

\ArticleName{$\boldsymbol{C}$-Vectors and Non-Self-Crossing Curves\\ for Acyclic Quivers of Finite Type\footnote{This paper is a~contribution to the Special Issue on Cluster Algebras. The full collection is available at \href{https://www.emis.de/journals/SIGMA/cluster-algebras.html}{https://www.emis.de/journals/SIGMA/cluster-algebras.html}}}

\Author{Su Ji HONG}

\AuthorNameForHeading{S.J.~Hong}

\Address{Department of Mathematics, University of Nebraska-Lincoln, USA}
\Email{\href{mailto:sujihong@huskers.unl.edu}{sujihong@huskers.unl.edu}}

\ArticleDates{Received June 01, 2020, in final form January 17, 2021; Published online February 01, 2021}

\Abstract{Let $Q$ be an acyclic quiver and $k$ be an algebraically closed field. The indecomposable exceptional modules of the path algebra $kQ$ have been widely studied. The real Schur roots of the root system associated to $Q$ are the dimension vectors of the indecomposable exceptional modules. It has been shown in~[N\'ajera~Ch\'avez~A., \textit{Int.\ Math.\ Res.\ Not.} \textbf{2015} (2015), 1590--1600] that for acyclic quivers, the set of positive $c$-vectors and the set of real Schur roots coincide. To give a diagrammatic description of $c$-vectors, K-H.~Lee and K.~Lee conjectured that for acyclic quivers, the set of $c$-vectors and the set of roots corresponding to non-self-crossing admissible curves are equivalent as sets~[\textit{Exp.\ Math.}, {t}o appear,
arXiv:1703.09113]. In [\textit{Adv. Math.} \textbf{340} (2018), 855--882], A.~Felikson and P.~Tumarkin proved this conjecture for 2-complete quivers. In this paper, we prove a~revised version of Lee-Lee conjecture for acyclic quivers of type $A$, $D$, and $E_{6}$ and $E_7$.}

\Keywords{real Schur roots; $c$-vectors; acyclic quivers; non-self-crossing curves}

\Classification{13F60; 16G20}

\renewcommand{\thefootnote}{\arabic{footnote}}
\setcounter{footnote}{0}

\section{Introduction}\label{intro}

\looseness=-1 Let $Q$ be an acyclic quiver and $k$ be an algebraically closed field. The path algebra $kQ$ is an algebra generated by all oriented paths of $Q$. The category of $kQ$ modules is equivalent to the category of quiver representations. Kac has shown that the dimension vectors of indecomposable quiver representations are the positive roots of the root system corresponding to the quiver~$Q$~\cite{Kac_inflie}. One of the objects of interest in relation to~$kQ$ is the set of indecomposable exceptional representations of~$kQ$. The dimension vectors of these representations are called real Schur roots and the set of real Schur roots is a subset of positive roots. There are many different ways to describe real Schur roots. In~\cite{Schofield}, Schofield showed a way to classify real Schur roots using subrepresentations and in~\cite{IgusaSchf}, Igusa and Schiffler used the Coxeter element to classify real Schur roots. In \cite{HubKrause}, Hubery and Krause used non-crossing partitions to determine real Schur roots. Another way to describe a real Schur root comes from the cluster algebra associated to~$Q$. A~cluster algebra of~$Q$ is an algebra with generators, called cluster variables, obtained from the quiver and exchange relations. Fomin and Zelevinksy showed that any cluster variables can be expressed as Laurent polynomials in the initial cluster variables~\cite{FominZel1}. When a~cluster variable is at its most reduced form, the exponent vector of the denominator monomial gives the denominator vector, $d$-vector, of that cluster variable. Caldero and Keller proved that for an acyclic quiver, the set of real Schur roots and the set of $d$-vectors of non-initial cluster variables are equivalent~\cite{CalKel}.

Another description of a real Schur root comes from the framed quiver of $Q$. In the framed quiver of $Q$, $c$-vectors describe the relations between the mutable vertices and the frozen vertices (more precise definition is given in Section~\ref{background}). One of the big questions about $c$-vectors was the sign coherence of $c$-vectors. In~\cite{DerkWeyZel}, Derksen, Weyman, and Zelevinsky proved the sign coherence of $c$-vectors. Thus the set of $c$-vectors is the disjoint union of the set of positive $c$-vectors and the set of negative $c$-vectors. Chavez showed that for an acyclic quiver, the set of positive $c$-vectors and the set of real Schur roots coincide~\cite{Chavez15}.

Although there are many ways to describe a real Schur root, none of the descriptions mentioned so far are diagrammatic. In cluster algebras, the diagrammatic descriptions of clusters are useful. In~\cite{FominShapThurs}, Fomin, Shapiro, and Thurston described clusters using tagged triangulations and arcs. Also, Nakanishi and Stella gave diagrammatic descriptions of $c$-vectors and $d$-vectors of quivers of finite type~\cite{NakStel}. To give more general geometric description of $c$-vectors of acyclic quivers, hence real Schur roots, K.-H.~Lee and K.~Lee formulated the following conjecture about the set of real Schur roots and a set of certain paths on a Riemann surface called admissible curves (defined in Section~\ref{subsection_admcurv}).

\begin{Conj}[Lee--Lee \cite{LeeLee}]\label{conjll}
For an acyclic quiver, the set of roots associated to non-self-crossing admissible curves and the set of real Schur roots coincide.
\end{Conj}

In~\cite{FelTum}, Felikson and Tumarkin gave an alternative but equivalent definition of non-self-crossing admissible curves. Conjecture~\ref{conjll} follows from \cite[Proposition~33]{Bourbaki} for the acyclic quivers of finite type (see Corollary~\ref{cor_of_prop_bourbaki}). Given a real Schur root, there are multiple non-self-crossing admissible curves that correspond to the root. To categorify these curves, K. Lee defined positive, non-decreasing, and strictly increasing curves which will be defined in Section~\ref{background}. These descriptions of curves led to the following conjecture.

\begin{Conj}[Lee \cite{LeePersonalConversation}]\label{conjllnondec}
For an acyclic quiver, the set of associated roots of non-decreasing non-self-crossing admissible curves and the set of real Schur roots coincide.
\end{Conj}
More precise version of this conjecture is given in Conjecture \ref{conjmain}. We note that Conjecture~\ref{conjllnondec} implies Conjecture~\ref{conjll}. Felikson and Tumerkin's result~\cite{FelTum} implies Conjecture~\ref{conjllnondec} for acyclic 2-complete quivers, i.e., quivers with at least two arrows between every pair of vertices. In this paper we prove Conjecture \ref{conjllnondec} for the quivers type $A$, $D$, $E_6$ and $E_7$.

\begin{Thm}[revised Lee--Lee conjecture for type $A$, $D$, $E_6$ and $E_7$] \label{thm_in_intro} Given a quiver of type $A$, $D$, $E_6$ or $E_7$, the set of roots associated to non-decreasing non-self-crossing admissible curves is the same as the set of real Schur roots. In particular, if an acyclic quiver is of type $A$, then the set of roots associated to strictly increasing non-self-crossing admissible curves is the same as the set of real Schur roots.
\end{Thm}

We introduce the roots and admissible curves in Section~\ref{background} and introduce Coxeter transformation and state Theorem~\ref{thm_in_intro} more precisely in Section~\ref{coxeter}. In Sections~\ref{type_a}--\ref{type_e}, we describe the non-decreasing non-self-crossing admissible curves for quivers of type $A$, $D$ and $E$ respectively using Coxeter transformation and other methods. Appendix \ref{affine_a} gives a supporting evidence to Conjecture~\ref{conjllnondec} by showing that it holds true for affine quiver of type $A$ with single source and single sink.

\section{Background}\label{background}
In this section, we define some of the objects of interest including the admissible curve.
\subsection[c-vectors and roots system]{$\boldsymbol{c}$-vectors and roots system}

A \emph{quiver} $Q$ is a finite oriented graph without any oriented 2-cycles and loops. We denote $Q = (Q_0,Q_1)$ where $Q_0 = [n]=\{1,\ldots, n\}$ is the set of vertices of $Q$ and $Q_1$ is the set of oriented edges, which we call arrows, of $Q$. Let $h,t\colon Q_1\to Q_0$ be maps given by $h(i\to j) = j$ and $t(i\to j) = i$. The \emph{exchange matrix} of $Q$ is $B =(b_{ij})_{1\leq i,j \leq n}$ where $b_{ij}$ is the number of arrows from the vertex $i$ to the vertex $j$ (if there are $r$ arrows from the vertex $j$ to the vertex $i$, then $b_{ij}=-r$). We \emph{mutate} a quiver $Q$ at a vertex $i$ to obtain $\mu_i(Q)$ by the steps below:
\begin{enumerate}\itemsep=0pt
\item[1)] for every path $j\to i\to k$, create an arrow $j\to k$,
\item[2)] reverse the orientations of all the arrows incident to $i$, 	
\item[3)] delete any 2-cycles.
\end{enumerate}

\begin{Def}Let $Q=(Q_0,Q_1)$ be a quiver and $\hat{Q}_0$ be a duplicate of $Q_0$. Then the \emph{framed quiver} of $Q$ is $\hat{Q}= \big(Q_0 \cup \hat{Q}_0,\, Q_1 \cup \{i\to i'\,|\,\text{for all } i \in [n]\}\big)$ where mutation at a vertex $i'$ is not allowed for all $i' \in \hat{Q}_0$.
\end{Def}

\begin{Def}Given a quiver $Q$ and a sequence of mutations $w = \m_{i_k} \cdots \m_{i_1}$, consider $w\big(\hat{Q}\big)$. A \emph{$c$-vector} of $Q$ is given by
\[\left[ \begin{matrix}
c_{i,1'}\\
c_{i,2'}\\
\vdots\\
c_{i,n'}
\end{matrix}\right], \qquad \text{where}\qquad
c_{i,j'} = \begin{cases}
\hphantom{-}r & \text{if } i \xrightarrow{r \text{ arrows}} j', \\
-r & \text{if } i \xleftarrow{r \text{ arrows}} j'.
\end{cases}\]
\end{Def}
\begin{Exa}
Consider a framed quiver $\hat{Q}$ and $\mu_1\mu_2\mu_3\big(\hat{Q}\big)$ below.
\begin{center}

\begin{tikzpicture}
\node[shape=circle,draw=black, minimum size = .4cm, inner sep = 4pt, outer sep = 0.5pt,](1) at (0,0) {1};
\node[shape=circle,draw=black, minimum size = .4cm, inner sep = 4pt, outer sep = 0.5pt,](2) at (2,0) {2};
\node[shape=circle,draw=black, minimum size = .4cm, inner sep = 4pt, outer sep = 0.5pt,](3) at (4,0) {3};
\node[shape=circle,draw=black, minimum size = .4cm, inner sep = 2.7pt, outer sep = 0.5pt,](4) at (0,-1) {$1'$};
\node[shape=circle,draw=black, minimum size = .4cm, inner sep = 2.7pt, outer sep = 0.5pt,](5) at (2,-1) {$2'$};
\node[shape=circle,draw=black, minimum size = .4cm, inner sep = 2.7pt, outer sep = 0.5pt,](6) at (4,-1) {$3'$};
\path[->](1) edge (2) ;
\path[->](2) edge (3);
\path[->](1) edge (4);
\path[->](2) edge (5);
\path[->](3) edge (6);

\draw[-Implies,line width=1pt,double distance=1pt] (4.8,-.5) -- (6.5,-.5);
\node at (5.7,-.7){$\mu_1\mu_2\mu_3$};
\node[shape=circle,draw=black, minimum size = .4cm, inner sep = 4pt, outer sep = 0.5pt,](1) at (7,0) {1};
\node[shape=circle,draw=black, minimum size = .4cm, inner sep = 4pt, outer sep = 0.5pt,](2) at (9,0) {2};
\node[shape=circle,draw=black, minimum size = .4cm, inner sep = 4pt, outer sep = 0.5pt,](3) at (11,0) {3};
\node[shape=circle,draw=black, minimum size = .4cm, inner sep = 2.7pt, outer sep = 0.5pt,](4) at (7,-1) {$1'$};
\node[shape=circle,draw=black, minimum size = .4cm, inner sep = 2.7pt, outer sep = 0.5pt,](5) at (9,-1) {$2'$};
\node[shape=circle,draw=black, minimum size = .4cm, inner sep = 2.7pt, outer sep = 0.5pt,](6) at (11,-1) {$3'$};
\path[->](1) edge (2) ;
\path[<-](1) edge (4);
\path[<-](1) edge (5);
\path[<-](1) edge (6);
\path[->](2) edge (3);
\path[->](2) edge (4);
\path[->](3) edge (5);
\end{tikzpicture}
\end{center}

The $c$-vectors are
\[ \left[ \begin{matrix}
-1\\-1\\-1
\end{matrix}\right], \qquad \left[ \begin{matrix}
1\\0\\0
\end{matrix}\right],\qquad \left[ \begin{matrix}
0\\1\\0
\end{matrix}\right].\]
\end{Exa}

As mentioned before, these $c$-vectors are sign coherent, meaning that all the entries of a~$c$-vector are either non-negative or non-positive.

For the rest of this paper, let $Q$ be an acyclic quiver and $B=(b_{ij})_{1\leq i,j \leq n}$ be the exchange matrix of~$Q$. Then the \emph{generalized Cartan matrix} $C(B)$ is $(a_{ij})_{1\leq i,j \leq n}$ where $a_{ii} =2$, $a_{ij}= -|b_{ij}|$ for $i\neq j$.

Let $\D_Q$ be the root system of the Kac-Moody algebra associated to $C(B)$. Let $\a_1,\ldots, \a_n$ be the simple roots associated to the vertices $1,2,\ldots,n$ respectively. Any root can be viewed as $\Z$-linear combination of the simple roots, i.e., for any root $\a$, there exist $\b_1,\ldots \b_n \in \Z$ such that $\a=\sum_{i=1}^n \b_i\a_i$. Given a full subquiver~$Q'$, the root system of~$Q'$ is a subroot system of~$\D_Q$. All the roots are either positive or negative. Thus $\D_Q = \D_Q^+ \cup \D_Q^-$ where $\D_Q^+$ is the set of all positive roots and~$\D_Q^-$ is the set of all negative roots. Of these roots, a~real Schur root is a~positive root that corresponds to the dimension vector of an indecomposable exceptional module of~$kQ$. Chavez showed that the set of $c$-vectors coincides with the set of real Schur roots and their opposites in the root system in~\cite{Chavez15}. Since real Schur roots are positive, the set of positive $c$-vectors coincides with the set of real Schur roots.

Given a quiver $Q$ and the exchange matrix $B$, let $\a$ be a root in $\D_Q$. We can represent $\a$ with the underlying graph of the quiver by labeling the vertex~$i$ with $\b_i$ for $i\in [n]$. For example, consider a quiver whose underlying graph is
\begin{center}
\begin{tikzpicture}[scale =.6]		
\node[shape=circle,draw=black, minimum size = .4cm, inner sep = 4pt, outer sep = 0.5pt](2) at (-4,0) {$2$};
\node[shape=circle,draw=black, minimum size = .4cm, inner sep = 4pt, outer sep = 0.5pt](3) at (-2,0) {$3$};
\node[shape=circle,draw=black, minimum size = .4cm, inner sep = 4pt, outer sep = 0.5pt](4) at (0,0) {$4$};
\node[shape=circle,draw=black, minimum size = .4cm, inner sep = 4pt, outer sep = 0.5pt](5) at (2,0) {$5$};
\node[shape=circle,draw=black, minimum size = .4cm, inner sep = 4pt, outer sep = 0.5pt](8) at (4,0) {$8$};
\node[shape=circle,draw=black, minimum size = .4cm, inner sep = 4pt, outer sep = 0.5pt](6) at (6,0) {$6$};
\node[shape=circle,draw=black, minimum size = .4cm, inner sep = 4pt, outer sep = 0.5pt](1) at (8,0) {$1$};
\node[shape=circle,draw=black, minimum size = .4cm, inner sep = 4pt, outer sep = 0.5pt](7) at (4,-2) {$7$};
\path[-](1) edge (6);
\path[-](2) edge (3);
\path[-](3) edge (4);
\path[-](4) edge (5);
\path[-](5) edge (8);
\path[-](7) edge (8);
\path[-] (8) edge (6);
\end{tikzpicture}
\end{center}

Let $\a =\a_1 + \a_2+2\a_3+3\a_4+4\a_5+3\a_6+2\a_7+5\a_8$ be a root in $\D_Q$. Then we can also represent $\a$ as
\begin{center}
\begin{tikzpicture}
\node at (-1,0){1};
\node at (-.5,0){2};
\node(1) at (0,0) {3};
\node(2) at (.5,0) {$4$};
\node(3) at (1,0) {5};
\node(4) at (1.5,0) {$3$};
\node(5) at (2,0) {$1$};
\node(6) at (1,-.5) {$2$};
\end{tikzpicture}
\end{center}

\begin{Def}We define the standard partial order on $\Delta_Q^+$ as follows. Let $\a =\sum_{i=1}^n \b_i\a_i$ and $\a' =\sum_{i=1}^n \b'_i\a_i$ be roots in $\D_Q^+$. Then \emph{$\a \leq_D \a'$} if and only if $\b_i \leq \b'_i$ for all $i \in [n]$ and \emph{$\a<_D\a'$} if $\a\leq_D\a'$ and $\b_i<\b_i'$ for some $i\in[n]$.
\end{Def}

We can define reflections for any root in $\D_Q$; in particular, a simple reflection $s_i$ is a reflection in a simple root $\a_i$ and is defined by
\[s_i (\a_j) = \a_j - a_{ij}\a_i, \qquad \text{for all} \quad i,j \in [n].\]

\subsection{Admissible curves}\label{subsection_admcurv}
K.-H.~Lee and K.~Lee defined the admissible curves on some space in~\cite{LeeLee}. In~\cite{FelTum}, A.~Felikson and P.~Tumarkin gave an equivalent definition of the admissible curves on a disc with $n$ marked interior points and a boundary point called a basepoint. We use this definition of admissible curves in this paper.

\begin{Def}Let $H = \R \times \R_{\geq 0}$ and for all $i\in [n]$, $p_i = (i,1) $ be a marked point and $\r_i = \{(i,y)\,|\,y\geq 1\}$ be a~ray in~$H$. Let $b = (0,0)$ be the base point and $P = \{p_i\,|\,i\in [n]\}$ be the set of all marked points. An \emph{admissible curve} is a smooth continuous function $\g\colon [0,1]\to H$ such that
\begin{enumerate}\itemsep=0pt
\item[(a)] $\g(0) \in P$ and $\g(1) = b$,
\item[(b)] if $\g(x) \in \{b\}\cup P$, then $x\in \{0,1\}$, and
\item[(c)] if $\g((0,1]) \cap \r_i \neq \varnothing$, then $\g$ and $\r_i$ intersect transversely.
\end{enumerate}
An admissible curve is \emph{non-self-crossing} if for all $x, y \in [0,1]$ such that $\g(x) = \g(y)$, $x=y$. See Fig.~\ref{exam_adm_cur} for an example of a non-self-crossing admissible curve.
\end{Def}

\begin{figure}[h]\centering\vspace*{-10mm}
\begin{subfigure}[b]{.45\textwidth}\centering
	\includegraphics[trim = 30mm 230mm 60mm 10mm, clip, width = 0.8\linewidth ]{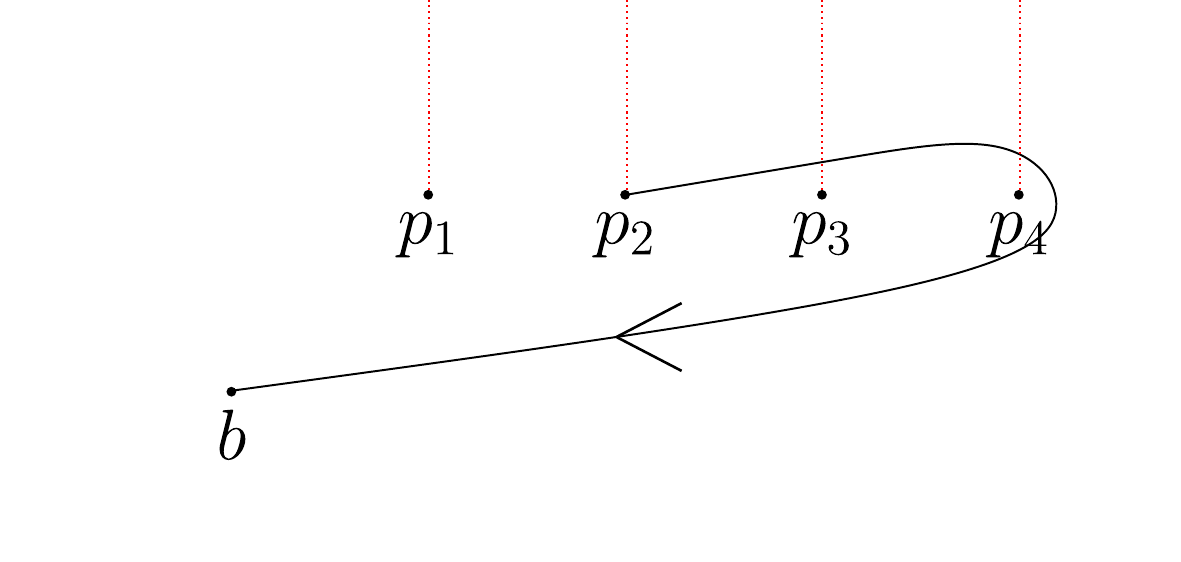}
	\caption{Example of an admissible curve}
 \label{exam_adm_cur}
\end{subfigure}%
\begin{subfigure}[b]{.45\textwidth}
 \centering
	\includegraphics[trim = 30mm 165mm 60mm 60mm, clip, width = 0.8\linewidth ]{ex_of_adm_curv.pdf}
\caption{$\g_3$ when $n=4$} \label{gamma_3}
\end{subfigure}
\caption{}\label{fig:test}
\end{figure}

Let $\g_i$ be a curve defined as $\g_i(x) = (i(1-x),1-x),$ for all $i\in[n]$. For example, the curve in Fig.~\ref{gamma_3} is $\g_3$ when $n=4$. We can easily see that $\g_i$ is an admissible curve, as{\samepage
\begin{enumerate}\itemsep=0pt
\item[(a)] $\g_i(0) = p_i\in P$ and $\g_i(1) = b$,
\item[(b)] if $\g_i(x) \in \{b\} \cup P$, then $\g_i(x) = (i(1-x),1-x) = (0,0)$ or $(k,1)$ for some $k\in [n]$, which means that $x = 0$ or $x=1$,
\item[(c)] $\g_i((0,1]) \cap \r_j = \varnothing$ for all $j$.
\end{enumerate}
Thus $\g_i$ is an admissible curve and it does not cross itself.}

In this paper, we express these curves using the description of the braid group as a mapping class group of $H$ as described in \cite[Chapter~9]{FarbMargalitMapping}. See Fig.~\ref{fig_braid_action} for an example of the braid group action on $H$.

\begin{figure}[h]\centering
\begin{subfigure}[b]{.45\textwidth}
 \centering
	\includegraphics[trim = 0mm 233mm 115mm 10mm, clip, width = 0.8\linewidth ]{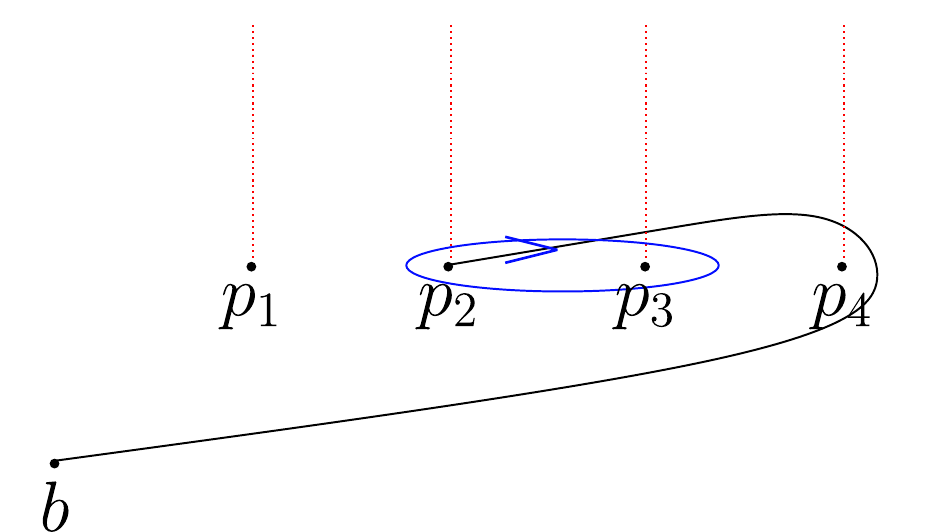}
	\caption{$\g$ and $\s_2$} \label{exam_adm_cur_2_1}
\end{subfigure}%
\begin{subfigure}[b]{.45\textwidth}
 \centering
 	\includegraphics[trim = 0mm 165mm 115mm 70mm, clip, width = 0.8\linewidth ]{ex_of_adm_curv_2.pdf}
\caption{$\s_2(\g)$} \label{exam_adm_cur_2_2-1}
\end{subfigure}
\begin{subfigure}[b]{.45\textwidth}
 \centering
	\includegraphics[trim = 100mm 233mm 20mm 10mm, clip, width = 0.8\linewidth ]{ex_of_adm_curv_2.pdf}
	\caption{$\g$ and $\s_2^{-1}$}
 \label{exam_adm_cur_2_3}
\end{subfigure}%
\begin{subfigure}[b]{.45\textwidth}
 \centering
	\includegraphics[trim = 100mm 165mm 20mm 70mm, clip, width = 0.8\linewidth ]{ex_of_adm_curv_2.pdf}
\caption{$\s_2^{-1}(\g)$}
 \label{exam_adm_cur_2_2-2}
\end{subfigure}
\caption{}\label{fig_braid_action}
\end{figure}

For $i,j \in [n]$ such that $i<j$, define $\s_{[i,j]} = \s_{j-1}\cdots\s_{i+1}\s_{i}$ and $\s_{[j,i]} = \s_i\s_{i+1}\cdots \s_{j-1}$. Then $\s_{[i,j]}^{-1} = \s_{i}^{-1}\s_{i+1}^{-1}\cdots \s_{j-1}^{-1}$ and $\s_{[j,i]}^{-1} =\s_{j-1}^{-1}\cdots \s_{i+1}^{-1}\s_i^{-1}$. Consider $\g_i$ and $\s_{[i,j]}\g_i$ in Fig.~\ref{seq_sig}. Then $\s_{[i,j]}\g_i(0) = p_j$ and as $x$ increases, $\s_{[i,j]}\g_i(x)$ crosses $\r_{j-1}$, $\r_{j-2}$, \dots, $\r_i$ in that order.

\begin{figure}[t]\centering\vspace*{-5mm}
\begin{subfigure}[b]{.3\textwidth}
 \centering
	\includegraphics[trim = 15mm 250mm 110mm 10mm, clip, width = \linewidth ]{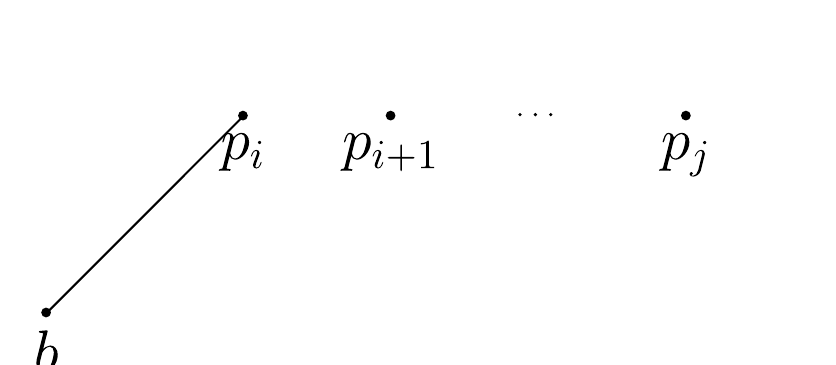}
	\caption{$\g_i$}
 \label{S_2}
\end{subfigure}%
\begin{subfigure}[b]{.3\textwidth}
 \centering
\includegraphics[trim =100mm 250mm 20mm 10mm, clip, width = \linewidth ]{ex_S_i.pdf}
\caption{$\s_{[i,j]}(\g_i)$}
 \label{S_2g_5-1}
\end{subfigure}
\begin{subfigure}[b]{.3\textwidth}
 \centering
\includegraphics[trim =100mm 200mm 20mm 50mm, clip, width = \linewidth ]{ex_S_i.pdf}
\caption{$\s_{[j,i]}\s_{[i,j]}(\g_i)$}
 \label{S_2g_5-2}
\end{subfigure}
\caption{}
\label{seq_sig}
\end{figure}

\begin{Def}
Let $\pi$ be a permutation in $S_n$ and $\g$ be a non-self-crossing admissible curve. Assume $\g(x)$ crosses the rays $\r_{k_1},\r_{k_2}, \ldots, \r_{k_m}$ as $x$ increases. Let $I_\g$ be the word $\r_{k_1}\r_{k_2}\cdots\r_{k_m}$ over $\{\r_1,\ldots,\r_n\}$. Then the \emph{associated root} of $\g$ with respect to $\pi$ is
\begin{gather*}\a_\pi(\g)=\begin{cases} s_{\pi(k_{m})}s_{\pi(k_{m-1})}\ldots s_{\pi(k_2)}s_{\pi(k_1)}\a_{\pi(k_0)}, \\
\qquad \text{if } s_{\pi(k_{m})}s_{\pi(k_{m-1})}\ldots s_{\pi(k_2)}s_{\pi(k_1)}\a_{\pi(k_0)} \text{ is positive},\\
-s_{\pi(k_{m})}s_{\pi(k_{m-1})}\ldots s_{\pi(k_2)}s_{\pi(k_1)}\a_{\pi(k_0)}, \\
\qquad \text{if } s_{\pi(k_{m})}s_{\pi(k_{m-1})}\ldots s_{\pi(k_2)}s_{\pi(k_1)}\a_{\pi(k_0)} \text{ is negative}, \end{cases}
\end{gather*}
where $\a_i$ is a simple root associated to the vertex $i$ and $s_i$ is a simple reflection for the simple root $\a_i$. Note that $\a_\pi(\g)$ is always positive.

As $s_is_i = id$, if two curves are isotopic, then they have the same associated root.
\end{Def}

\begin{Def}
Given a permutation $\pi$ in $S_n$ and a non-self-crossing admissible curve $\g$, assume $I_\g= k_0k_1\cdots k_m$. Then let $\pi(I_\g)$ be the word $\r_{\pi(k_0)}\r_{\pi(k_1)}\cdots \r_{\pi(k_m)}$ over $\{\r_1,\ldots, \r_n\}$. Let $a, j_1,\ldots, j_m \in \{1,\ldots, n\}$ so that $\pi(k_0)=a$ and $\pi(k_i) = j_i$ for $1\leq i\leq m$. Then $\g$ is said to be
\begin{itemize}\itemsep=0pt
\item \emph{positive} with respect to $\pi$ if $s_{j_i}\cdots s_{j_1}\a_a$ are positive for all $i\in[m]$,
\item \emph{non-decreasing} with respect to $\pi$ if $s_{j_i}\cdots s_{j_1}\a_a \geq_D s_{j_{i-1}}\cdots s_{j_1}\a_a$, for all $i\in [m]$, and
\item \emph{strictly increasing} with respect to $\pi$ if $s_{j_i}\cdots s_{j_1}\a_a>_D s_{j_{i-1}}\cdots s_{j_1}\a_a$, for all $i\in [m]$.
\end{itemize}
If $\pi$ is clear from the context, we just say positive, non-decreasing, and strictly increasing curves. Let $\G_{\pi,p}$ be the set of positive non-self-crossing curves with respect to $\pi$, $\G_{\pi,nd}$ be the set of non-decreasing non-self-crossing curves with respect to $\pi$, and $\G_{\pi,s}$ be the set of strictly increasing curves with respect to $\pi$.
\end{Def}

\begin{Rmk}\label{rmk_diff_curves}Recall that $\G$ is the set of isotopy classes of non-self-crossing admissible curves. Let $\pi$ be a permutation in $S_n$ and~$\g$ be a non-self-crossing admissible curve that is non-decreasing with respect to $\pi$. Let $\pi(I_\g) = \r_a\r_{j_1}\cdots \r_{j_m}$. If $\g$ is strictly increasing, then $s_{j_i}\cdots s_{j_1}\a_a \geq_D s_{j_{i-1}}\cdots s_{j_1}\a_a$ holds true for all $i$ by definition of strictly increasing non-self-crossing curve. Also, $s_{j_i}\cdots s_{j_1}\a_a \geq_D \a_a$ for all $i$. Since $\a_a$ is positive, $s_{j_i}\cdots s_{j_1}\a_a$ is positive for all $i$. Thus $\g$ is positive and we can see that $\G_{\pi,s} \subset \G_{\pi, nd} \subset \G_{\pi,p} \subset \G$.
\end{Rmk}

\begin{Conj}[Lee \cite{LeePersonalConversation}]\label{conjmain}
Let $Q$ be an acyclic quiver with $n$ vertices and $\G_{\pi,nd}$ be the set of non-decreasing non-self-crossing admissible curve. Let $P_Q$ be the set of permutations $\pi$ such that if $a >b$, then there is no arrow from $\pi(a)$ to $\pi(b)$. Then \[\bigcup_{\pi \in P_Q}\{ \a_\pi(\g)|\g\in \G_{\pi,nd}\} = \{\text{real Schur roots of } Q \}. \]
\end{Conj}

\begin{Rmk} \label{rmk_associatedd_roots_subset_of_real_schur_roots}
For the acyclic quivers of type $ADE$, the set of real Schur roots coincide with the set of positive roots. By definition of the associated root of a non-self-crossing admissible curve, it is clear that $\a_\pi(\g)$ is a positive root, hence real Schur root, for all $\g \in \G$. Thus, for an acyclic quiver $Q$ of type ADE, \[\bigcup_{\pi \in P_Q}\{ \a_\pi(\g)|\g\in \G\} \subseteq \{\text{real Schur roots of } Q \}. \]
\end{Rmk}

\section{Coxeter transformation}\label{coxeter}
From now on we consider finite types of ADE. In this section we consider a Coxeter transformation of the Weyl group $W$ of a root system and develop some useful tools to prove Conjecture~\ref{conjmain}. The Weyl group $W$ of a root system is a Coxeter group with generators, $s_1,\ldots, s_n$, and relations on them. Let $c_\pi= s_{\pi(1)}s_{\pi(2)}\cdots s_{\pi(n)}$ where $\pi$ is a permutation of $[n]$. The element $c_\pi \in W$ is called a Coxeter transformation. As $W$ is finite, the order of~$c_\pi$ is finite. Let $h$ be the order of $c_\pi$. This $h$ does not depend on $\pi$, i.e., for any $\tau$ in $S_n$, $h$ is the order of~$c_\tau$. Let $\theta_i = s_{\pi(n)}s_{\pi(n-1)}\cdots s_{\pi(i+1)}\a_{\pi(i)}$. We can recognize $\theta_i$ as the associated root of a non-self-crossing admissible curve. Let $\g$ be a non-self-crossing curve as below:
\begin{center}
	\includegraphics[trim = 30mm 220mm 30mm 25mm, clip, width = .5 \linewidth ]{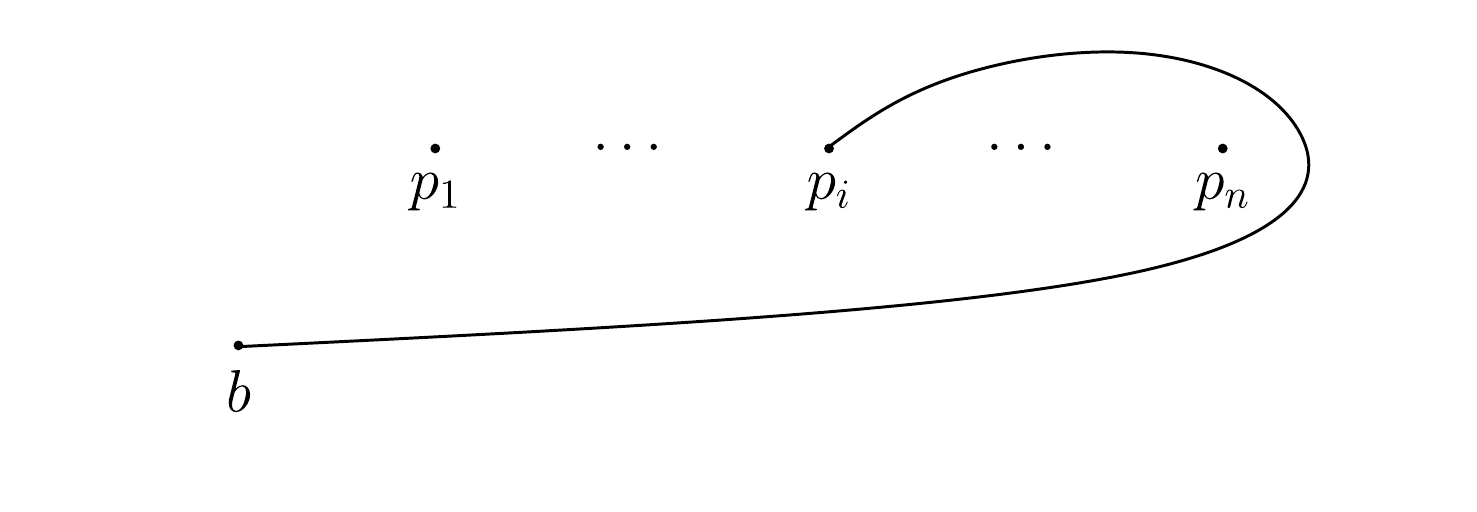}
\end{center}

As $I_\g=r_i\cdots r_n$, $\a_\pi(\g) = \theta_i$. Let $c\colon\Gamma \to \Gamma$ be defined by sending $\g \in \G$ to a non-self-crossing admissible curve obtained from $\g$ by wrapping around all marked points counter-clockwise at the end. For example, $\g_i$ and $c(\g_i)$ are given below.
\begin{center}\vspace*{-7mm}
	\includegraphics[trim = 30mm 160mm 30mm 70mm, clip, width = .5 \linewidth ]{theta_i_and_cox_trans.pdf}
\end{center}

Note that $I(c(\g))$ is given by concatenating $I(\g)$ and $\r_n\r_{n-1}\cdots \r_1$; thus $\a_\pi(c(\g)) = s_{\pi(1)}s_{\pi(2)}\allowbreak \cdots s_{\pi(n)} \a_\pi(\g)$ up to a sign where $\pi$ is any permutation in $S_n$. Let $\big\{c_\pi^k \theta_i\,|\,0\leq k <h\big\}_+$ be the intersection of $\big\{c_\pi^k \theta_i\,|\,0\leq k <h\big\}$ and the set of positive roots of~$Q$. Note that $\big\{c_\pi^k \theta_i\,|\,0\leq k <h\big\}_+ = \big\{\a_\pi(c^k\g)\,|\,0\leq k <h\big\}$. Son Nguyen informed the author of the following observation.

\begin{Prop}[Bourbaki \cite{Bourbaki}] \label{prop_Bourbaki}
Let $\Delta$ be a root system and $W$ be its Weyl group. Let $\pi$ be a~permutation of $[n]$, $\theta_i = s_{\pi(n)}s_{\pi(n-1)}\cdots s_{\pi(i+1)}\a_i$, and $\Omega_i = \big\{c_\pi^k\theta_i \,|\,k=0,\ldots, h-1\big\}$. Then
\begin{enumerate}\itemsep=0pt
 \item[$1)$] $\Omega_i \cap \Omega_j = \varnothing$ for all $i \neq j$, and
 \item[$2)$] $\Delta = \bigcup_{i=1}^n \Omega_i$.
\end{enumerate}
\end{Prop}

Recall that $P_Q = \{\pi \in S_n\,|\, \text{if } \pi(a)\to \pi(b) \in Q_1, \text{ then } a<b \}$.
\begin{Cor}[Lee--Lee conjecture for finite acyclic case]\label{cor_of_prop_bourbaki}
For an acyclic quiver $Q$ of type $ADE$,
\[\bigcup_{\pi\in P_Q}\{\a_\pi(\g)|\g \in \G\} = \{\text{real Schur roots of } Q\} = \D_Q^+.\]
\end{Cor}
\begin{proof}
By Remark \ref{rmk_associatedd_roots_subset_of_real_schur_roots}, we know that $\bigcup_{\pi\in P_Q}\{\a_\pi(\g)\,|\,\g \in \G\} \subseteq \D_Q^+$. Note that by Proposition~\ref{prop_Bourbaki}, $\big\{c_\pi^k \theta_i\,|\,0\leq k <h\big\}$ is the set of all roots. Thus the subset of positive roots, denoted by $\big\{c_\pi^k \theta_i\,|\,0\leq k <h\big\}_+$, is the same as~$\D_Q^+$.
\[\D_Q^+=\big\{c_\pi^k \theta_i\,|\,0\leq k <h\big\}_+=\big\{\a_\pi\big(c^k\g\big)\,|\,0\leq k <h\big\} \subseteq \bigcup_{\pi\in P_Q}\{\a_\pi(\g)\,|\,\g \in \G\} \subseteq \D_Q^+.\]
Therefore $ \bigcup_{\pi\in P_Q}\{\a_\pi(\g)\,|\,\g \in \G\} =\D_Q^+$.
\end{proof}

\begin{Exa}Consider the quiver below. Note that $\a_2\in \D^+_Q$.{\samepage
\begin{center}
\begin{tikzpicture}
\node[shape=circle,draw=black, minimum size = .4cm, inner sep = 4pt, outer sep = 0.5pt,](1) at (0,0) {$1$};
\node[shape=circle,draw=black, minimum size = .4cm, inner sep = 4pt, outer sep = 0.5pt,](2) at (2,0) {2};
\node[shape=circle,draw=black, minimum size = .4cm, inner sep = 4pt, outer sep = 0.5pt,](3) at (4,0) {$3$};
\path[->](1) edge (2);
\path[->](2) edge (3);
\end{tikzpicture}
\end{center}}

\noindent
As there is an arrow from vertex 1 to 2 and an arrow from 2 to 3, the only permutation in $P_Q$ is the identity. Then $\pi \in P_Q$ is the identity. Consider the following non-self-crossing admissible curves below.
\begin{center}
	\includegraphics[trim = 40mm 238mm 90mm 20mm, clip, width = .3\linewidth ]{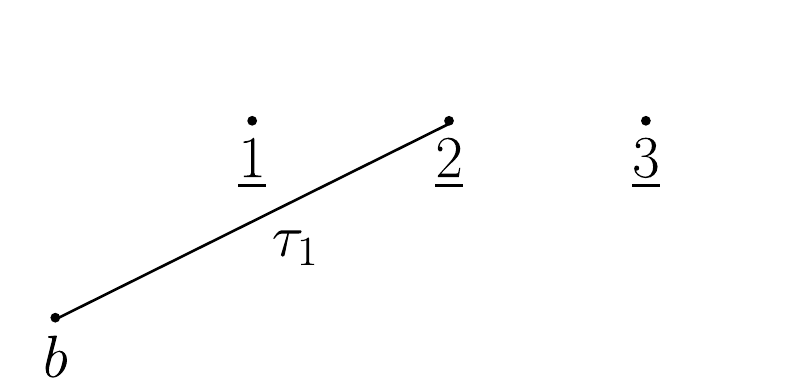}
	\includegraphics[trim = 40mm 196mm 90mm 65mm, clip, width = .3\linewidth ]{ex_many_curves.pdf}
	\includegraphics[trim = 40mm 153mm 90mm 105mm, clip, width = .3\linewidth ]{ex_many_curves.pdf}
\end{center}
Note that $\a_\pi(\tau_1) = \a_2$, $\a_\pi(\tau_2) = s_3s_2\a_3 = \a_2$, and $\a_\pi(\tau_3) = s_3s_1s_3s_1\a_2 = \a_2$. The associated roots for all three curves are $\a_2$. Note that $\tau_1$ is strictly increasing and thus non-decreasing whereas $\tau_2$ and $\tau_3$ are positive but are not non-decreasing. By Corollary~\ref{cor_of_prop_bourbaki}, given a quiver $Q$ and a root $\a \in \D_Q^+$, there exists a non-self-crossing admissible curve whose associated root is $\a$. However, we do not know that whether that curve is non-decreasing.
\end{Exa}

Given a positive root, there could be many non-self-crossing admissible curves in $\G$ that correspond to that root. It is a natural question to see if there are non-decreasing non-self-crossing admissible curves among those non-self-crossing admissible curves. The following is our main theorem.
\begin{Thm}\label{thm_main}
Let $Q$ be an acyclic quiver of type $A$, $D$, $E_6$ and $E_7$. Then
\[\bigcup_{\pi\in P_Q}\{\a_\pi(\g)\,|\,\g \in \G_{\pi,nd}\} = \D^+_Q.\]
Furthermore, if $Q$ is an acyclic quiver of type $A$, then
$\bigcup_{\pi\in P_Q}\{\a_\pi(\g)\,|\,\g \in \G_{\pi,s}\} = \D^+_Q$.
\end{Thm}
\begin{Rmk}
We have a supporting evidence that Theorem~\ref{thm_main} holds for quivers of type $E_8$ as well as affine quivers of type~$A$. We will discuss those two cases in Appendix~\ref{appendix}.
\end{Rmk}

Note that Proposition~\ref{prop_Bourbaki} does not imply this theorem. To prove this theorem, we first observe a relationship between the root system of a quiver and the root system of a full subquiver. Let $Q=
(Q_0,Q_1)$ be a quiver and $Q'$ be a full subquiver induced by $V \subset Q_0$. Recall that $P_Q$ is the set of all permutations in $S_{|Q_0|}$ such that if $i>j$, then there is no arrow $\pi(i)\to \pi(j)$ in $Q_1$. As~$\pi(i)$ represents a vertex of $Q$, we can think of $\pi$ as a map $\pi\colon \{1,\ldots, |Q_0|\}\to Q_0$. Thus $P_{Q'}$ can be thought of as the set of bijective maps $\pi'\colon \{1,\ldots, |V|\} \to V$ such that if $i >j$, then there is no arrow $\pi'(i) \to \pi'(j)$ in $Q'$. Let $\pi \in P_Q$ and $\{a_1,\ldots, a_k\}$ be a subset of $\{1,\ldots,|Q_0|\}$ such that $a_i <a_j$ for $i<j$ and $\{a_1,\ldots,a_k \} = \pi^{-1}(V)$. Define $\varphi(\pi)$ to be the map $\varphi(\pi)\colon \{1,\ldots,k\} \to V$ given by $i \mapsto \pi(a_i)$. Note that if $i >j$, then $a_i > a_j$. Thus there is no arrow from $\pi(a_i) \to \pi(a_j)$. Thus we may think of $\varphi(\pi)$ as of an element of $P_{Q'}$. Abusing notation, we will write $\varphi(\pi)\in P_{Q'}$.

\begin{Lem} \label{lem_subquiver}Let $Q=(Q_0,Q_1)$ be an acyclic quiver and $Q'$ be a connected full subquiver induced by $V \subset Q_0$. For $\pi' \in P_{Q'}$ such that $\varphi^{-1}(\pi')$ is not empty,{\samepage
\begin{itemize}\itemsep=0pt
\item $\{\a_{\pi'}(\g)\,|\,\g \in \G_{\pi',s}\} \subseteq \{\a_{\pi}(\g)\,|\,\g \in \G_{\pi,s}\}$,
\item $\{\a_{\pi'}(\g)\,|\,\g \in \G_{\pi',nd}\} \subseteq \{\a_{\pi}(\g)\,|\,\g \in \G_{\pi,nd}\}$,
\item $\{\a_{\pi'}(\g)\,|\,\g \in \G_{\pi',p}\} \subseteq \{\a_{\pi}(\g)\,|\,\g \in \G_{\pi,p}\}$,
\end{itemize}
where $\pi$ is any permutation in $\varphi^{-1}(\pi')$, where $\varphi$ is given above.}
\end{Lem}

\begin{proof}
Let $\a = \a_{\pi'}(\g')$ for $\pi' \in P_{Q'}$ and some non-self-crossing curve $\g'$. Let $a, j_1,\ldots, j_m$ be vertices in $V$ such that $\pi'(I(\g')) = \r_a\r_{j_1}\cdots \r_{j_m}$. Let $\pi \in \varphi^{-1}(\pi')$. Then $\pi^{-1}(j_i) - \pi^{-1}(j_{i+1})$ and $\pi'^{-1}(j_i) - \pi'^{-1}(j_{i+1})$ are both positive or both negative. Let $\g$ be a non-self-crossing curve such that $\g(0) = p_{\pi^{-1}(a)}$, crosses $\r_{\pi^{-1}(j_i)}$ in the order that~$\g'$ crosses, and goes under the marked points $p_{\pi^{-1}(k)}$ for all $k \not\in V$. Then $\pi(I(\g)) = \r_a\r_{j_1}\cdots \r_{j_m}$.

If $\g' \in \G_{\pi', s}$, then $s_{j_i}\cdots s_{j_1}\a_a >_D s_{j_{i-1}}\cdots s_{j_1}\a_a $ for all $i\in [m]$. Thus $\g \in \G_{\pi,s}$. Similarly, if $\g' \in \G_{\pi,nd}$, then $\g \in \G_{\pi,nd}$ and if $\g' \in \G_{\pi,p}$, then $\g \in \G_{\pi,p}$.
\end{proof}

\begin{Rmk}\label{rmk_subquiver}
If $Q=(Q_0,Q_1)$ is an acyclic quiver of finite type and $Q'$ is a connected full subquiver of $Q=(V,E)$ such that $Q_0\setminus V=\{v\}$, then $v$ is a sink or source of $Q$. Without loss of generality, assume that $v$ is a sink. Then given any $\pi' \in P_{Q'}$, define $\pi$ be given by $\pi(i) = \pi'(i)$ for $i < |Q_0|$ and $\pi(|Q_0|) = v$. Note that $\pi \in P_Q$ and $\varphi(\pi) = \pi'$. We can iterate this process to see that given any quiver $Q$ of finite type and a connected full subquiver $Q'$, if $\pi' \in P_{Q'}$, then~$\varphi^{-1}(\pi')$ is not empty.
\end{Rmk}

\begin{Rmk}In Lemma \ref{lem_subquiver}, a vertex $i$ is ``ignored'' if $\b_i=0$. This method of removing unused vertices is a common approach. For example, a maximal green sequence of a full subquiver of a quiver can be obtained from a maximal green sequence of the quiver by this approach; see in \cite[Theorem~3.3]{GarMcSer}.
\end{Rmk}
If $\g$ is a non-self-crossing admissible curve, then so is $c \g$. Furthermore, as $I(c\g) =I(\g) \r_n\r_{n-1}\allowbreak \cdots\r_1$ we can compare $\a$ and $c_\pi\a$ to say more about~$c \g$.

\begin{Lem}\label{lem_cox_transf}
Let $Q$ be an acyclic quiver, $\pi \in P_Q$, and $c_\pi$ be a Coxeter transformation of the Weyl group of $\D_Q$. If $\a \in \{\a_\pi(\g)\,|\,\g\in \G_{\pi,nd}\}$ and $\a <_D c_\pi \a$, then $c_\pi\a \in \{\a_\pi(\g)\,|\,\g\in \G_{\pi,nd}\}$. Similarly, if $\a \in \{\a_\pi(\g)\,|\,\g\in \G_{\pi,nd}\}$ and $\a <_D c_\pi^{-1} \a$, then $c_\pi^{-1}\a \in \{\a_\pi(\g)\,|\,\g\in \G_{\pi,nd}\}$.
\end{Lem}

\begin{proof}
Given $\pi \in P_Q$ and $\g \in \G_{\pi,nd}$, let $\a = \a_\pi(\g)$. Assume that $\a<_D c_\pi\a$. As $c\colon \G\to\G$, $c\g$ is a non-self-crossing admissible curve. To show that $c\g$ is non-decreasing, note that $I_{c\g} = I_{\g} \oplus (n,\ldots,2,1)$. As $\g \in \G_{\pi,nd}$, it suffices to show that $\a\leq_Ds_{\pi(n)}\a$ and $s_{\pi(k+1)}s_{\pi(k+2)}\cdots s_{\pi(n)}\a \leq_D s_{\pi(k)}s_{\pi(k+1)}\cdots s_{\pi(n)}\a$ for all $k \in [n-1]$. Let $a =\a$ or $a=s_{\pi(k+1)}s_{\pi(k+2)}\cdots s_{\pi(n)}\a$. Given any root, a simple reflection affects the coefficient of the corresponding simple root only. Thus~$a$ and~$s_{\pi(k)}a$ differ only by the coefficient of~$\a_{\pi(k)}$. As each simple reflections appears only once in~$c_\pi$, the coefficient of~$\a_{\pi(k)}$ in~$a$ matches the coefficient of~$\a_{\pi(k)}$ in~$\a$. Similarly, the coefficient of~$\a_{\pi(k)}$ in~$s_{\pi(k)}a$ matches the coefficient of $\a_{\pi(k)}$ in $c_\pi\a$. As $\a <_D c_\pi\a$, we know that $a \leq_D s_{\pi(k)}a$. Therefore $c\g$ is non-decreasing.

If $\a<_D c_\pi^{-1}\a$, then $c^{-1} \g$ is a non-self-crossing admissible curve and $c_\pi^{-1}\a= s_{\pi(n)}s_{\pi(n-1)}\cdots\allowbreak s_{\pi(1)}\a$. Using a similar argument as above, we can see that $c^{-1}\g \in \G_{\pi,nd}$ and $c_\pi^{-1}\a \in \{\a_\pi(\g)\,|\,\g\in \G_{\pi,nd}\}$.
\end{proof}

\section[Type A]{Type $\boldsymbol{A}$}\label{type_a}
In this section, we focus on quivers of type $A$, i.e., quivers whose underlying undirected graph is
\begin{center}
\begin{tikzpicture}
\node[shape=circle,draw=black, minimum size = .6cm, inner sep = 4pt, outer sep = 0.5pt,](1) at (0,0) {$n$};
\node[shape=circle,draw=black, minimum size = .9cm, inner sep = .5pt, outer sep = 0.5pt,](2) at (2,0) {{\scriptsize $n-1$}};
\node(3) at (4,0) {\ldots};
\node[shape=circle,draw=black, minimum size = .4cm, inner sep = 4pt, outer sep = 0.5pt,](4) at (6,0) {$2$};
\node[shape=circle,draw=black, minimum size = .4cm, inner sep = 4pt, outer sep = 0.5pt,](5) at (8,0) {$1$};
\path[-](1) edge node[above] {$e_{n-1}$} (2) ;
\path[-](2) edge node[above] {$e_{n-2}$} (3);
\path[-](3) edge node[above] {$e_{2}$} (4);
\path[-](4) edge node[above] {$e_{1}$} (5);
\end{tikzpicture}
\end{center}

Let $\cA_n$ be the set of such quivers. As mentioned before, all positive roots of type $A$ are real Schur roots. It is known that all the positive roots of type $A$ are of the form
\[\sum_{k=\ell}^{m} \a_k, \qquad \text{where}\qquad 1\leq \ell \leq m \leq n.\]

We will prove that for certain $\pi \in P_Q$, any positive root is the associated root of a strictly increasing non-self-crossing admissible curve. We show this by providing an algorithm to construct such curves. Furthermore, for any permutation $\pi\in P_Q$, any root is the associated root of a non-decreasing non-self-crossing admissible curve.

\subsection{Strictly increasing curves}
Let $Q \in \cA_n$ in this section. To show that for a certain permutation $\pi \in P_Q$, $\{\a_\pi(\g)\,|\,\g\in \G_{\pi,s}\} = \D_Q^+$, we first define a unimodal permutation.
\begin{Def}
A permutation $\pi\in S_n$ is said to be \emph{unimodal} if there exists $k\in [n]$ such that for all $i<j\leq k$, $\pi(i)<\pi(j)$ and for all $k\leq i<j$, $\pi(i)>\pi(j)$. In particular, $\pi(k) = n$. Let $U_n$ be the set of all unimodal permutations in $S_n$.
\end{Def}

For a permutation $\pi \in U_n$, we define a quiver $Q_\pi$ in $\cA_n$ by giving an orientation of the edges $e_i$. If $\pi^{-1}(i) < \pi^{-1}(n)$, then $t(e_i) = i$ and if $\pi^{-1}(i)>\pi^{-1}(n)$, then $h(e_i) = i$. Define $\omega\colon U_n \to \cA_n$ by $\omega(\pi) = Q_\pi$.

For example, $\omega\left(\left(\begin{smallmatrix}
1 & 2 & 3& 4 & 5\\
1 & 3 & 5 & 4 &2\\
\end{smallmatrix}\right)\right)$ is
\begin{center}
\begin{tikzpicture}
\node[shape=circle,draw=black, minimum size = .4cm, inner sep = 4pt, outer sep = 0.5pt,](1) at (0,0) {$5$};
\node[shape=circle,draw=black, minimum size = .4cm, inner sep = 4pt, outer sep = 0.5pt,](2) at (2,0) {$4$};
\node[shape=circle,draw=black, minimum size = .4cm, inner sep = 4pt, outer sep = 0.5pt,](3) at (4,0) {$3$};
\node[shape=circle,draw=black, minimum size = .4cm, inner sep = 4pt, outer sep = 0.5pt,](4) at (6,0) {$2$};
\node[shape=circle,draw=black, minimum size = .4cm, inner sep = 4pt, outer sep = 0.5pt,](5) at (8,0) {$1$};
\path[->](1) edge node[above] {$e_{4}$} (2) ;
\path[<-](2) edge node[above] {$e_{3}$} (3);
\path[->](3) edge node[above] {$e_{2}$} (4);
\path[<-](4) edge node[above] {$e_{1}$} (5);
\end{tikzpicture}
\end{center}

For any quiver $Q \in \cA_n$, let $E_u = \{i\in [n-1]\,|\, t(e_i) = i\}$ and $E_d = \{i\in [n-1]\,|\, h(e_i) = i\}$. We define a unimodal permutation using the elements of $E_u$ to form the increasing sequence and the elements of $E_d$ to form the decreasing sequence. Define $a_1 = \min E_u$ and $a_i = \min (E_u\setminus \{a_1,\ldots, a_{i-1}\})$ for $2\leq i \leq |E_u|$. Let $b_1 = \max E_d$, and $b_i= \max(E_d \setminus \{b_1,\ldots, b_{i-1}\})$ for $2\leq i \leq n-|E_u|-1$. Then define a unimodal permutation $\pi_Q$ as below
\[\pi_Q = \left(\begin{smallmatrix}
1 & 2 & \cdots & |E_u| & |E_u|+1 & |E_u|+2 & \cdots &n\\
a_1 & a_2& \cdots &a_{|E_u|} & n & b_1 & \cdots & b_{n-|E_u|-1} \\
\end{smallmatrix}\right).\]

Let $\psi\colon \cA_n \to U_n$ be a map defined by $\psi\colon Q\mapsto \pi_Q$. For example, given $Q$ below,
\begin{center}
\begin{tikzpicture}
\node[shape=circle,draw=black, minimum size = .4cm, inner sep = 4pt, outer sep = 0.5pt,](2) at (2,0) {$4$};
\node[shape=circle,draw=black, minimum size = .4cm, inner sep = 4pt, outer sep = 0.5pt,](3) at (4,0) {$3$};
\node[shape=circle,draw=black, minimum size = .4cm, inner sep = 4pt, outer sep = 0.5pt,](4) at (6,0) {$2$};
\node[shape=circle,draw=black, minimum size = .4cm, inner sep = 4pt, outer sep = 0.5pt,](5) at (8,0) {$1$};
\path[->](2) edge (3);
\path[<-](3) edge (4);
\path[<-](4) edge (5);
\end{tikzpicture}
\end{center}
the unimodal permutation given by $\psi$ is $\psi(Q)= \left(\begin{smallmatrix}
1 & 2 &3&4\\
1 & 2 & 4 & 3\\
\end{smallmatrix}\right)$.

\begin{Lem}
There exists a bijection between $\cA_n$ and $U_n$. \label{anun}
\end{Lem}
\begin{proof}
We claim that $\omega\colon U_n\to \cA_n$ given above is a bijection. It suffices to show that~$\omega$ and~$\psi$ are inverses. To show that $\omega(\psi(Q))=Q$, it suffices to show that the arrows, $e_i$, have the same orientations in both $Q$ and $\omega(\psi(Q))$. Let $\pi_Q = \psi(Q)$. Let $c= \pi_Q^{-1}(n)$ and $a=\pi_Q^{-1}(i)$ for some $i\in [n-1]$. If $t(e_{i}) = i$, then $ a< c$. Thus in $ \omega(\pi_Q),$ $t(e_{\pi_Q(a)}) = i$. Similarly, if $h(e_i) = i$, then $h(e_{\pi_Q(a)})= i$ in $\omega(\pi_Q)$. Thus $\omega(\pi_Q)$ and $Q$ have the same orientations for all edges, which means $\omega(\psi(Q))=Q$.

To see that $\psi(\omega(\pi)) = \pi$, let $Q_\pi = \omega(\pi)$ and $k=\pi^{-1}(n)$. Then
\[E_u = \{i\in [n-1]\,|\, t(e_i) = i\}=\big\{i\in [n-1]\,|\, \pi^{-1}(i) <\pi^{-1}(n)\big\} = \{\pi(i)\,|\, i <k\}.\]

 As $\pi$ is unimodal, $\pi(1) < \pi(2) < \cdots < \pi(k-1)$. Thus $a_1 = \min E_u = \pi(1)$, $a_2 =\pi(2), \ldots$, and $a_{k-1}= \pi(k-1)$. Similarly, we can see that $E_d = \{\pi(i)\,|\,i >k\}$ and $b_1 = \pi(k+1), \ldots, b_{n-k} = \pi(n)$. Thus $\psi(\omega(\pi)) = \pi$. Therefore $\psi$ and $\omega$ are inverses.
\end{proof}

\begin{Rmk}Note that given $Q\in \cA_n$ and $\pi_Q = \psi(Q)$, there are no arrows from the vertex~$\pi_Q(a)$ to the vertex $\pi_Q(b)$ in $Q$ if $a >b$. Thus $\pi_Q \in P_Q$. By Lemma~\ref{anun}, each quiver $Q$ has a~unimodal permutation in $P_Q$ as $P_Q \cap U_n = \{\pi_Q\}$.
\end{Rmk}

We will describe non-self-crossing admissible curves using the braid group action on~$\{\g_1,\allowbreak \ldots,\g_n \}$. Recall $\s_{[i,j]} = \s_{j-1}\cdots \s_{i+1}\s_i$ and $\s_{[j,i]} = \s_i\s_{i+1} \cdots \s_{j-1}$ for $i<j$. Given $\pi \in U_n$, let
\[\Sigma_i = \begin{cases}
\s_{[\pi^{-1}(i)+1,\pi^{-1}(i+1)]}^{-1}\s_{\pi^{-1}(i)}& \text{if } \pi^{-1}(i) <\pi^{-1}(i+1),\\
\s_{[\pi^{-1}(i)-1,\pi^{-1}(i+1)]}\s_{\pi^{-1}(i)-1}^{-1}& \text{if } \pi^{-1}(i) >\pi^{-1}(i+1).
\end{cases}\]

From here on, let $\underline{i} = p_{\pi(i)}$. To depict $\S_i$, consider $\pi =\left(\begin{smallmatrix}
1 & 2 & 3& 4 & 5&6\\
3 & 6 & 5 & 4 &2&1\\
\end{smallmatrix}\right).$ Then $\S_2\g_{\pi^{-1}(2)} = \s_{[4,1]}\s^{-1}_{4}\g_5$ is
\begin{center}
\vspace*{-6mm}

\includegraphics[trim = 20mm 171mm 60mm 85mm, clip, width =.47\linewidth ]{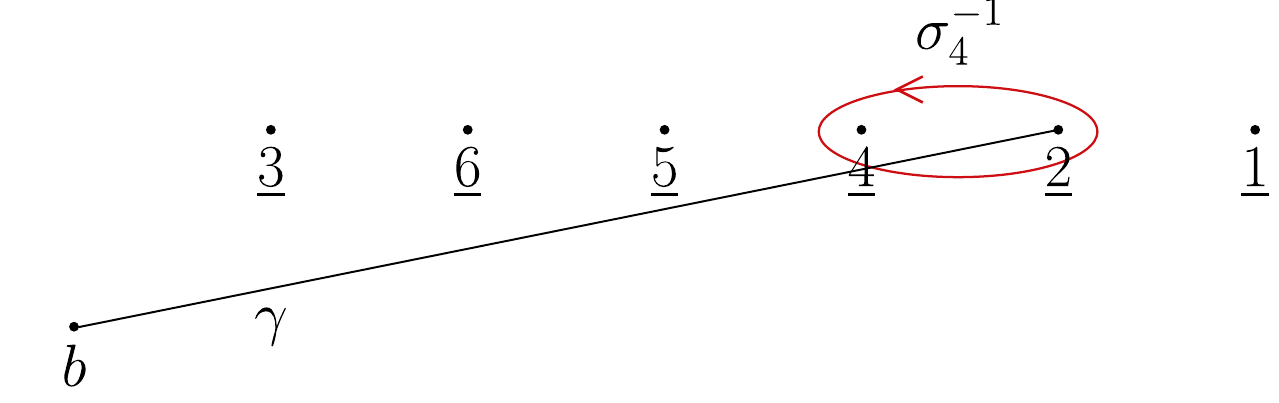}
\includegraphics[trim =20mm 110mm 60mm 130mm, clip, width = .47\linewidth ]{an_ex_1_rev.pdf}\\
\includegraphics[trim =20mm 68mm 60mm 195mm, clip, width = .47\linewidth ]{an_ex_1_rev.pdf}
\end{center}

Note that $\S_2\g_{\pi^{-1}(2)}$ loops around $\underline{2}$, go below all the points between $\underline{2}$ and $\underline{3}$, and end at $\underline{3}$. Then $\a_{\psi(Q)}(\S_2\g_{\pi^{-1}(2)})= \s_2(\a_3) = \a_2+\a_3$. Also there exists a sufficiently small neighborhood around $\underline{3}$ so that $\s_2\g_{\pi^{-1}(2)}$ is homotopy equivalent to $\g_{\pi^{-1}(3)}$ in that neighborhood also the both curves do not intersect $\r_{\pi^{-1}(3)}$; see Fig.~\ref{S_2g_5_2}. Thus $\S_{3}\S_2\g_{\pi^{-1}(2)}$ loops around $\underline{3}$ and going under all the other points and end at $\underline{4}$.

\begin{Prop}\label{prop_a_strict_inc} Let $Q$ be an acyclic quiver of type $A$ and $\pi \in P_Q \cap U_n$. Let $\a = \sum_{i=\ell}^m \a_i$ be a positive root of type $A$ and let $\g = \S_{m-1}\S_{m-2}\cdots \S_\ell \g_{\pi^{-1}(\ell)}$. Then $\g \in \G_{\pi,s}$ and $\a_\pi(\g) = \a$. Furthermore,
\[\{\a_\pi(\g)\,|\,\g\in\G_{\pi,s}\} = \D_Q^+.\]
\end{Prop}
See Example~\ref{ex_of_a_n_lemma} for an illustration of this Proposition.

\begin{figure}[h]\centering
\begin{subfigure}[b]{.4\textwidth}
 \centering
\includegraphics[trim =20mm 70mm 60mm 195mm, clip, width = \linewidth ]{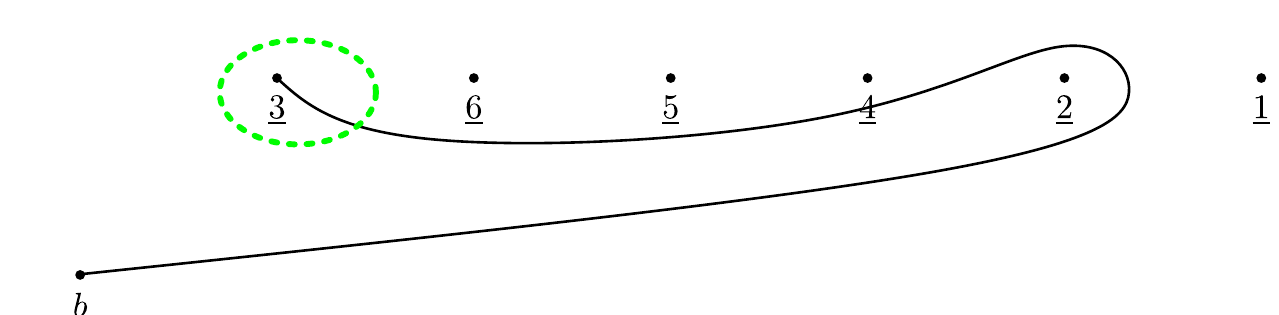}
\caption{$\S_2\g_5$}
\end{subfigure}
\begin{subfigure}[b]{.4\textwidth}
 \centering
\includegraphics[trim =20mm 20mm 20mm 240mm, clip, width = \linewidth ]{ex1_a_n_2.pdf}
\caption{$\g_3$}
\end{subfigure}
\caption{In the neighborhood around $\underline{3}$, $\S_2\g_5$ homotopy equivalent to $\g_3$.}\label{S_2g_5_2}
\end{figure}

\begin{proof}Let $\a = \a_{\ell} +\a_{\ell+1} + \cdots + \a_m$ and $\g = \S_{m-1}\cdots \S_\ell \g_{\pi^{-1}(\ell)}$. We claim that $\pi(I_\g) = \r_m\r_{m-1}\cdots \r_\ell$. To prove this, we induct on $m-\ell$. If $m-\ell = 0$, then $\g = \g_{\pi^{-1}(\ell)}$. Thus $\pi(I_\g) = \r_\ell$.

Let $\g' = \S_{m-2}\cdots \S_\ell\g_{\pi^{-1}(\ell)}$. By the induction hypothesis, $\pi(I_{\g'}) = \r_{m-1}\r_{m-2}\cdots\r_\ell$. There are two possible cases: $\pi^{-1}(m)<\pi^{-1}(m-1)$ and $\pi^{-1}(m)>\pi^{-1}(m-1)$. As they are symmetric, we can assume that $\pi^{-1}(m)>\pi^{-1}(m-1)$ without loss of generality.

Let $a = \pi^{-1}(m-1)$ and $b = \pi^{-1}(m)$. Then $\g = \S_{m-1}\g' =\s_{b-1}^{-1} \s_{b-2}^{-1}\cdots \s_{a+1}^{-1}\s_{a}\g'$. By the induction hypothesis, $\g'(0) = \underline{m-1}$, which means that $\s_a(\g')(0) = \underline{\pi(a+1)}$. There are two possibilities for $\pi(I_{\s_a(\g')})$: $\r_{\pi(a+1)}\r_{m-1}\r_{m-2}\ldots\r_\ell$ or $\r_{\pi(a+1)}\r_{m-2}\cdots\r_\ell$. Refer to Fig.~\ref{s_a_g'} for illustrations of these two cases. We claim that the second case is not possible. By way of contradiction, assume $\pi(I_{\s_a(\g')})=\r_{\pi(a+1)}\r_{m-2}\cdots \r_\ell$. Apply $\s_a^{-1}$ to $\s_a(\g')$ to obtain $\g'$. Note that in this case $\g'$ crosses $\r_{a+1}$. Since $\g'$ crosses $\r_{a+1}$, $\ell\leq \pi(a+1) \leq m-2$. However, $\pi$ is unimodal and $\pi(a) = m-1$ and $\pi(b) = m$. Then $a$ is part of an increasing sequence of this unimodal permutation $\pi$. Thus, $\pi(a+1) > \pi(a)=m-1$, which contradicts that $\pi(a+1) \leq m-2$. Thus the second case does not happen and $\s_a(\g')$ crosses $\r_{a}$ transversally.

\begin{figure}[h]\centering
\begin{subfigure}[b]{.45\textwidth}
 \centering
	\includegraphics[trim = 30mm 240mm 55mm 20mm, clip, width = \linewidth ]{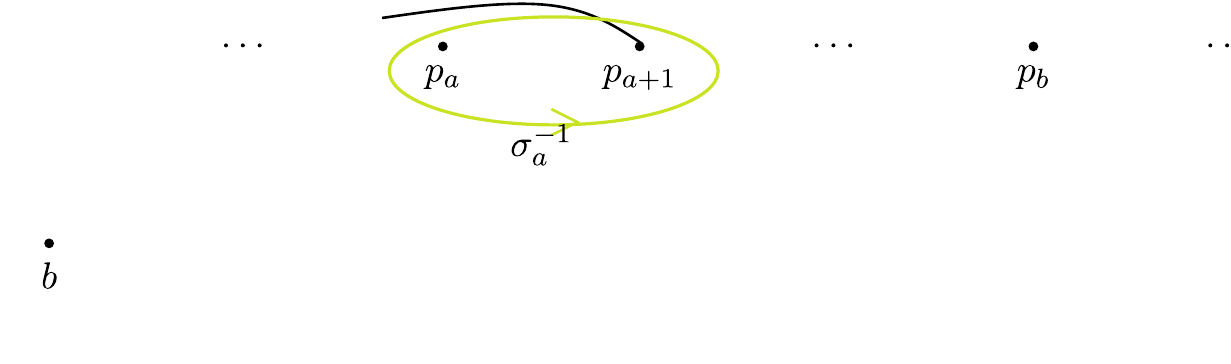}
	\caption{Case 1: $\s_a(\g')$ crosses $\r_{a}$}
 \label{S_a_y_m-1}
\end{subfigure}%
\begin{subfigure}[b]{.45\textwidth}
 \centering
\includegraphics[trim =30mm 187mm 55mm 73mm, clip, width = \linewidth ]{ex_a_n_g.pdf}
\caption{Case 2: $\s_a(\g')$ does not cross $\r_{a}$}
 \label{s_a_n_m-1}
\end{subfigure}
\begin{subfigure}[b]{.5\textwidth}
 \centering
	\includegraphics[trim = 30mm 139mm 55mm 110mm, clip, width = \linewidth ]{ex_a_n_g.pdf}
	\caption{Case 1: $\s_a^{-1}(\s_a(\g'))$ }
 \label{S_a_y_m-1_a}
\end{subfigure}%
\begin{subfigure}[b]{.5\textwidth}
 \centering
\includegraphics[trim =30mm 90mm 55mm 160mm, clip, width = \linewidth ]{ex_a_n_g.pdf}
\caption{Case 2: $\s_a^{-1}(\s_a(\g'))$}
 \label{s_a_n_m-1_a}
\end{subfigure}

\caption{Two cases.}\label{s_a_g'}
\end{figure}

Now consider $\s_{a+1}^{-1}\s_{a}(\g')$. The curve goes below $p_{a+1}$ since $\s_{a}(\g')$ crosses $\r_{a}$. Thus the curve does not cross $\r_{a+1}$. Also $\s_{a+1}^{-1}\s_{a}(\g')$ ends at $p_{a+2}$. Similarly, $\s_{b-1}^{-1}\s_{b-2}^{-1}\cdots \s_{a+1}^{-1}\s_a(\g')$ does not cross $\r_i$ for all $a<i\leq b$. Thus $\g$ ends at $p_b = \underline{m}$ and crosses $\r_{\pi^{-1}(m-1)}$. Therefore $\pi(I_\g) = \r_m\r_{m-1} \cdots \r_\ell$.

For a quiver in $\cA_n$,
\[s_i(\a_j)=\begin{cases}
 \a_j+\a_i & \text{if } j = i-1, i+1,\\
 \a_j & \text{otherwise}.
\end{cases}\]
Thus, $s_{m-1}\a_m = \a_{m-1}+\a_m$ and $s_{m-2}(\a_{m-1}+\a_m)=s_{m-2}(\a_{m-1})+s_{m-2}(\a_m) = \a_{m-2}+\a_{m-1}+\a_m.$ By continuing this process, we can see that $s_j s_{j-1}\cdots s_{m-1}\a_m = \sum_{k=j}^m \a_k$. Note that $\sum_{k=j-1}^m \a_k <_D \sum_{k=j}^m \a_k$ for all $\ell \leq j < m$. Therefore $\g \in \G_{\pi,s}$ and $\a_\pi(\g) = \a$.

Furthermore, $\{\a_\pi(\g)\,|\,\g\in\G_{\pi,s}\} \subseteq \D_Q^+$ as mentioned before. Since $\a_\pi(\g) = \a$ for any positive root $\a$ of type $A$, $\D_Q^+ = \{\a_\pi(\g)\,|\,\g\in\G_{\pi,s}\}$.
\end{proof}

\begin{Exa}
\label{ex_of_a_n_lemma}

Consider the following quiver $Q$:
\begin{center}
\begin{tikzpicture}
\node[shape=circle,draw=black, minimum size = .4cm, inner sep = 4pt, outer sep = 0.5pt,](6) at (0,0) {$6$};
\node[shape=circle,draw=black, minimum size = .4cm, inner sep = 4pt, outer sep = 0.5pt,](5) at (2,0) {$5$};
\node[shape=circle,draw=black, minimum size = .4cm, inner sep = 4pt, outer sep = 0.5pt,](4) at (4,0) {$4$};
\node[shape=circle,draw=black, minimum size = .4cm, inner sep = 4pt, outer sep = 0.5pt,](3) at (6,0) {$3$};
\node[shape=circle,draw=black, minimum size = .4cm, inner sep = 4pt, outer sep = 0.5pt,](2) at (8,0) {$2$};
\node[shape=circle,draw=black, minimum size = .4cm, inner sep = 4pt, outer sep = 0.5pt,](1) at (10,0){$1$};
\path[->](1) edge (2) ;
\path[->](2) edge (3);
\path[<-](3) edge (4);
\path[->](4) edge (5);
\path[->](5) edge (6);
\end{tikzpicture}
\end{center}
Then $\pi_Q = \left(\begin{smallmatrix}
1 & 2 & 3 & 4 &5 &6\\
1&2&4&5&6&3 \\
\end{smallmatrix}\right) $ and $\g = \S_{5}\S_4\S_3\S_{2} \S_1\g_{1}$ is shown below.
\begin{center}
	\includegraphics[trim = 30mm 40mm 45mm 223mm, clip, width = .8 \linewidth ]{ex_a_n_g.pdf}
\end{center}

Thus $\pi_Q(I_\g) = \r_6\r_5\r_4\r_3\r_2\r_1$ and $\a_{\pi_Q}(\g) = s_1s_2s_3s_4s_5\a_6 = \sum_{i=1}^6 \a_i.$ Note that this curve is not given by $c_\pi^k\theta_i$ for any $i$ and $k$. Thus Proposition \ref{prop_Bourbaki} does not cover this curve.
\end{Exa}

\subsection{Non-decreasing curve}
Not every permutation in $P_Q$ has a strictly increasing non-self-crossing admissible curve for every positive root. However, for any permutation in $P_Q$, there is a non-decreasing non-self-crossing admissible curve for every positive root.

\begin{Prop} \label{prop_a_non_dec}
Let $Q \in \cA_n$. For any $\pi \in P_Q$, $\{\a_\pi(\g)\,|\,\g \in \G_{\pi,nd}\} = \D^+_Q$.
\end{Prop}
\begin{proof}
As $\{\a_\pi(\g)\,|\,\g \in \G_{\pi,nd}\} \subseteq \D^+_Q$ for any $Q$ and $\pi \in P_Q$, we just need to show the other inclusion. The positive roots of type $A$ are of the form $\sum_{i=k}^\ell \a_i$, which is the highest root of the full subquiver induced by the vertices $k,\ldots, \ell.$ By Lemma \ref{lem_subquiver} and Remark~\ref{rmk_subquiver}, if $Q'$ is a full subquiver of $Q$, then $\{\a_{\pi'}(\g)\,|\,\g \in \G_{\pi',nd}\} \subseteq \{\a_\pi(\g)\,|\,\g \in \G_{\pi,nd}\}$ for $\pi' \in P_{Q'}$ and $\pi \in \varphi^{-1}(\pi')$. Thus it suffices to show that the highest root is in $\{\a_\pi(\g)\,|\,\g \in \G_{\pi,nd}\}$ for any $Q \in \cA_n$ and $\pi \in P_Q$.

Let us proceed by induction on $|Q_0|$. If $|Q_0|=1$, it is trivial. Assume that for a quiver in~$\cA_n$ with less than $n$ vertices, the highest root is in $\{\a_\pi(\g)\,|\,\g \in \G_{\pi,nd}\}$ for any $\pi$. Let $Q \in \cA_n$ be a quiver with $n$ vertices and $\a$ be the highest root. For any $\pi \in P_Q$, $c_\pi\a <_D \a$. If $c_\pi\a$ is positive, then it is of the form $\sum_{i=k}^\ell \a_k$ where $\ell-k < n-1$. So $c_\pi\a$ is the highest root of a full subquiver of~$Q$ and by the induction hypothesis, $c_\pi\a$ is in $\{\a_{\pi'}(\g)\,|\,\g \in \G_{\pi',nd}\}$ where $\pi' = \varphi(\pi)$. Thus $c_\pi\a \in \{\a_{\pi}(\g)\,|\,\g \in \G_{\pi,nd}\}$ by Lemma~\ref{lem_subquiver} and Remark~\ref{rmk_subquiver} and $\a \in \{\a_\pi(\g)\,|\,\g \in \G_{\pi,nd}\}$ by Lemma~\ref{lem_cox_transf}.

If $c_\pi\a$ is negative, then $c_\pi\a = -\a_{\pi(1)}$ as $\a$ is the highest root of $Q$. Thus $\a \!=\! s_{\pi(n)}\!\cdots s_{\pi(2)}\a_{\pi(1)}$. In this case, let $\g$ be the curve shown below
\begin{center}
\includegraphics[trim = 15mm 253mm 100mm 15mm, clip, width = .5 \linewidth ]{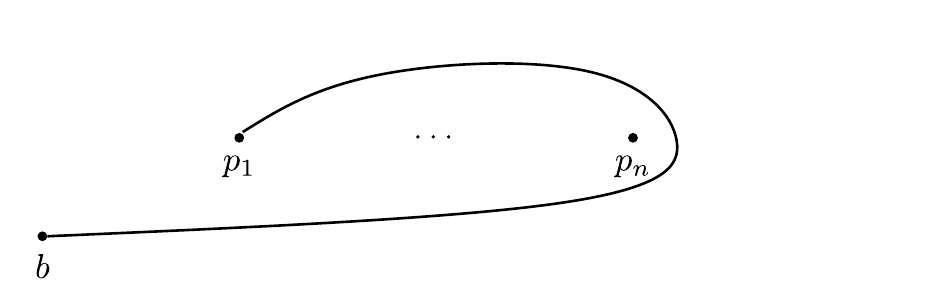}
\end{center}
Then $\a_\pi(\g) = s_{\pi(n)}\cdots s_{\pi(2)}\a_2 = \a$ and $\g \in \G_{\pi,s} \subset \G_{\pi,nd}$.
\end{proof}

\section[Type D]{Type $\boldsymbol{D}$}\label{type_d}
 In this section, we prove that for an acyclic quiver of type $D$ and any permutation respecting the orientations, there exists a non-decreasing non-self-crossing admissible curve for every root. Let~$\mathcal{D}_n$ be the set of type~$D$ quivers with~$n$ vertices, i.e., quivers whose underlying graph is
\begin{center}
\begin{tikzpicture}
\node(1)[shape=circle,draw=black, minimum size = .4cm, inner sep = 4pt, outer sep = 0.5pt,] at (0,0) {$n$};
\node(2)[shape=circle,draw=black, minimum size = .4cm, inner sep = .5pt, outer sep = 0.5pt,] at (2,0) {{\scriptsize $n-1$}};
\node(3)[shape=circle,draw=black, minimum size = .4cm, inner sep = .5pt, outer sep = 0.5pt,] at (0,-1) {{\scriptsize $n-2$}};
\node(4)[shape=circle,draw=black, minimum size = .4cm, inner sep = .5pt, outer sep = 0.5pt,] at (-2,0) {{\scriptsize $n-3$}};
\node(5)[shape=circle,draw=black, minimum size = .4cm, inner sep = .5pt, outer sep = 0.5pt,] at (-4.2,0) {{\scriptsize $n-4$}};
\node(6) at (-6.2,0) {$\ldots$};
\node(7)[shape=circle,draw=black, minimum size = .4cm, inner sep = 4pt, outer sep = 0.5pt,] at (-7.5,0) {$2$};
\node(8)[shape=circle,draw=black, minimum size = .4cm, inner sep = 4pt, outer sep = 0.5pt,] at (-9,0) {$1$};
\path[-](1) edge (2);
\path[-](1) edge (3);
\path[-](1) edge (4);
\path[-](4) edge (5);
\path[-](5) edge (6);
\path[-](6) edge (7);
\path[-](7) edge (8);
\end{tikzpicture}
\end{center}

It is known that the positive roots of a quiver of type $D$ are either type $A$ or one of the following:
\[\begin{smallmatrix}
0 &\cdots &0 & 1&\cdots & 1&1 \\
&& & &&1&
\end{smallmatrix},\,\, \begin{smallmatrix}
0 &\cdots &0 & 1&\cdots & 1 &2&\cdots & 2&1 \\
&&&&& & &&1&
\end{smallmatrix}\]

Unlike the type $A$ case, if $Q$ is an acyclic quiver of type $D$, not all positive roots are the associated roots of strictly increasing non-self-crossing curves.
\begin{Exa}
Consider the quiver $Q$ below. Then $P_Q = \left\{\left(\begin{smallmatrix}
1&2&3&4&5\\
4&1&2&5&3
\end{smallmatrix}\right),\left(\begin{smallmatrix}
1&2&3&4&5\\
1&4&2&5&3
\end{smallmatrix}\right),\left(\begin{smallmatrix}
1&2&3&4&5\\
1&2&4&5&3
\end{smallmatrix}\right)\right\}$.
\begin{center}
\begin{tikzpicture}[scale=.8]
\node(1)[shape=circle,draw=black, minimum size = .4cm, inner sep = 4pt, outer sep = 0.5pt,]at (0,0) {$5$};
\node(2)[shape=circle,draw=black, minimum size = .4cm, inner sep = 4pt, outer sep = 0.5pt,] at (2,0) {4};
\node(3)[shape=circle,draw=black, minimum size = .4cm, inner sep = 4pt, outer sep = 0.5pt,] at (0,-1.5) {3};
\node(4)[shape=circle,draw=black, minimum size = .4cm, inner sep = 4pt, outer sep = 0.5pt,] at (-2,0) {2};
\node(5)[shape=circle,draw=black, minimum size = .4cm, inner sep = 4pt, outer sep = 0.5pt,] at (-4.2,0) {1};
\path[<-](1) edge (2);
\path[->](1) edge (3);
\path[<-](1) edge (4);
\path[<-](4) edge (5);
\end{tikzpicture}
\end{center}
Let $\a = \begin{smallmatrix}
1&2&2&1\\
&&1&
\end{smallmatrix}$. If $\a = \a_\pi(\g)$ for some $\pi \in P_Q$ and $\g \in \G_{\pi,s}$, then $\pi(I_\g) = \r_{k_1}\r_{k_2}\r_{k_3}\r_{k_4}\r_{k_5}\r_5\r_2$ where $k_i$'s are distinct and the full subquiver induced by $\{k_1,\ldots, k_i\}$ is connected for all $i \leq 5$. However, this is not possible for any $\pi \in P_Q$.
\end{Exa}

If $\a = \sum_{i=1}^n \b_i\a_i$ and $\b_i = 0$ for some $i$, then it suffices to view $\a$ as a root of a full subquiver by Lemma~\ref{lem_subquiver} and Remark~\ref{rmk_subquiver}. Thus we only consider the roots of the form:
\[\begin{smallmatrix}
1&\cdots & 1&1 \\
 &&1&
\end{smallmatrix},\,\, \begin{smallmatrix}
1&\cdots & 1 &2&\cdots & 2&1 \\
&& & &&1&
\end{smallmatrix}\]

\begin{Lem}\label{lem_reducing_root}
Let $Q$ be an acyclic quiver of type $ADE$ with $n$ vertices. Let $\a$ be a positive root of $Q$ with a unique index $k\in [n]$ such that $s_k(\a) >_D\a$. If there is a vertex $j$ that is adjacent to the vertex $k$ such that $s_j(\a) <_D \a$, then $c_\pi\a<_D\a$ or $c_\pi^{-1}\a<_D\a$ for any $\pi \in P_Q$.
\end{Lem}

\begin{proof}
Let $Q$ be an acyclic quiver of type $ADE$. Let $\a = \sum_{i=1}^n \b_i\a_i$ be a positive root of $Q$ such that there exist two adjacent vertices $j$, $k$ such that $s_k\a >_D\a$, $s_j\a<_D$, and $s_i\a \leq_D\a$ for all other $i$. Note that as $\a$ is a positive root of finite type, $s_i\a = \a \pm \a_i$ or $\a$ for any $i$. As~$j$ and~$k$ are adjacent, $s_ks_j(\a) = s_k(\a - \a_j) = s_k(\a) - s_k(\a_j) = \a+\a_k -(\a_j+\a_k) = \a - \a_j$.

If $i \neq k,j$, then $s_i(\a-\a_j) =s_i\a-s_i\a_j \leq_D \a - \a_j$ as $s_i\a \leq_D\a$ and $s_i\a_j \geq_D \a_j$. Thus $s_i(\a-\a_j) \leq_D \a-\a_j$. If $\pi(k) < \pi(j)$, then $c_\pi \a = s_{\pi(1)} \cdots s_{\pi(n)} \a \leq_D s_ks_j\a <_D \a$. If $\pi(k) > \pi(j)$, then $c^{-1}_\pi \a = s_{\pi(n)} \cdots s_{\pi(1)} \a \leq_D s_ks_j\a <_D \a$.
\end{proof}

\begin{Prop}\label{prop_non_dec_curve_d}
Let $Q \in \cD_n$. Then $\{\a_\pi(\g)\,|\,\g\in \G_{\pi,nd}\} = \D_Q^+$ for any $\pi \in P_Q$.
\end{Prop}

\begin{proof}Let $Q \in \cD_n$ and $\pi$ be any permutation in $P_Q$. It is known that $\{\a_\pi(\g)\,|\,\g\in \G_{\pi,nd}\} \subseteq \D_Q^+$. The positive roots of~$Q$ are type~$A$, $\begin{smallmatrix}
0&\cdots&0&1&\cdots & 1&1 \\
&&& &&1&
\end{smallmatrix}$, or $\begin{smallmatrix}
0&\cdots&0& 1&\cdots & 1 &2&\cdots & 2&1 \\
&&& && & &&1&
\end{smallmatrix}$. By Lemma~\ref{lem_subquiver}, Remark~\ref{rmk_subquiver}, and Proposition~\ref{prop_a_non_dec} all the type $A$ positive roots are in $\{\a_\pi(\g)\,|\,\g\in \G_{\pi,nd}\} $.

If $\a= \begin{smallmatrix}
0&\cdots&0&1&\cdots & 1&1 \\
&&& &&1&
\end{smallmatrix},$ then it suffices to view $\a$ as a root of full subquiver induced by vertices whose corresponding simple roots have non-zero coefficients by Lemma~\ref{lem_subquiver} and Remark~\ref{rmk_subquiver}. So $\a= \begin{smallmatrix}
1&\cdots & 1&1 \\
&&1&
\end{smallmatrix}$ and $s_n\a >_D\a$, $s_{n-1}\a<_D\a$, and $s_i \leq_D\a$ for all other~$i$. Thus by Lemma~\ref{lem_reducing_root}, $c_\pi\a$~or~$c_\pi^{-1}\a$ is smaller than $\a$. Without loss of generality, assume that $c_\pi\a <_D\a$. Moreover $c_\pi\a <_D\a-\a_{n-1}$. If $c_\pi\a$ is positive, $c_\pi\a$~is of type~$A$ as $\a-\a_{n-1}$ is type~$A$. Thus by Proposition~\ref{prop_a_non_dec}, Lemma~\ref{lem_subquiver}, and Remark~\ref{rmk_subquiver}, $c_\pi\a \in \{\a_\pi(\g)\,|\,\g\in\G_{\pi, nd}\}.$ Then $\a \in \{\a_\pi(\g)\,|\,\g\in\G_{\pi, nd}\}$ by Lemma~\ref{lem_cox_transf}. If $c_\pi\a$ is negative, then $c_\pi\a = -\a_{\pi(1)}$ as $h(\a) = n$. Then $\a = s_{\pi(n)}\cdots s_{\pi(2)}\a_2$. Just as we have seen in the proof of Proposition~\ref{prop_a_non_dec}, there is $\g$ such that $\a_\pi(\g) = s_{\pi(n)}\cdots s_{\pi(2)}\a_2 = \a$ and $\g \in \G_{\pi,s} \subset \G_{\pi,nd}$.

Let $\a$ be a root of the form $\begin{smallmatrix}
0&\cdots&0& 1&\cdots & 1 &2&\cdots & 2&1 \\
&&&&& & &&1&
\end{smallmatrix}$. Let us proceed by induction on the height of~$\a$, $h(\a) = \sum_{i=1}^n \b_i$, to show that $\a \in \{\a_\pi(\g)\,|\,\g\in \G_{\pi,nd}\}$. The smallest $h(\a)$ is~5, i.e., $\a = \begin{smallmatrix} 0&\cdots&0&1&2&1\\&&& &1&\\ \end{smallmatrix}$. By Lemma~\ref{lem_subquiver} and Remark~\ref{rmk_subquiver}, it suffices to view $\a$ as the highest root of a quiver in $D_n$ with four vertices. Then $c_\pi\a$ is type $A$ or $\begin{smallmatrix}
1&1&1\\
&1&\\
\end{smallmatrix}$. By Proposition \ref{prop_a_non_dec} and the above argument, $c_\pi\a \in \{\a_\pi(\g)\,|\,\g\in \G_{\pi,nd}\}$ for any $\pi \in P_Q$. Then $\a \in \{\a_\pi(\g)\,|\,\g\in \G_{\pi,nd}\}$ for any $\pi \in P_Q$ by Lemma~\ref{lem_cox_transf}.

If $h(\a) >5,$ assume that for all $\a'$ such that $h(\a') < h(\a)$, $\a' \in\{\a_\pi(\g)\,|\,\g\in\G_{\pi, nd}\}$. Let $\a=\sum_{i=1}^n \b_i\a_i$. By Lemma~\ref{lem_subquiver} and Remark~\ref{rmk_subquiver}, we can assume that $\b_i >0$ for all $i$. Let $j$ be the smallest index so that $\b_j=2$. Note that $s_{j+1} \a = \a + \a_{j+1}$, $s_j \a = \a - \a_j$, and $s_i \a<_D\a$ for all other $i$. Thus by Lemma~\ref{lem_reducing_root}, $c_\pi \a <_D \a$ or $c_\pi^{-1} \a <_D\a$. As $h(\a)>n$ and each simple reflections decrease the coefficient by at most 1, $c_\pi\a$ and $c_\pi^{-1}\a$ cannot be negative. By Lemma~\ref{lem_cox_transf} and the induction hypothesis, $\a \in \{\a_\pi(\g)\,|\,\g \in \G_{\pi,nd}\}.$
\end{proof}

\begin{Exa}
Let $Q$ be the quiver below, $\pi = \left(\begin{smallmatrix} 1&2&3&4&5&6\\ 1&2&3&6&5&4 \end{smallmatrix}\right) \in P_Q$ and $\a = \begin{smallmatrix} 1&1&2&2&1\\ &&&1&\end{smallmatrix}$.
\begin{center}
\begin{tikzpicture}
\node(1)[shape=circle,draw=black, minimum size = .4cm, inner sep = 4pt, outer sep = 0.5pt,]at (0,0) {$6$};
\node(2)[shape=circle,draw=black, minimum size = .4cm, inner sep = 4pt, outer sep = 0.5pt,] at (2,0) {5};
\node(3)[shape=circle,draw=black, minimum size = .4cm, inner sep = 4pt, outer sep = 0.5pt,] at (0,-1.5) {4};
\node(4)[shape=circle,draw=black, minimum size = .4cm, inner sep = 4pt, outer sep = 0.5pt,] at (-2,0) {3};
\node(5)[shape=circle,draw=black, minimum size = .4cm, inner sep = 4pt, outer sep = 0.5pt,] at (-4,0) {2};
\node(6)[shape=circle,draw=black, minimum size = .4cm, inner sep = 4pt, outer sep = 0.5pt,] at (-6,0) {1};
\path[<-](1) edge (2);
\path[->](1) edge (3);
\path[<-](1) edge (4);
\path[<-](4) edge (5);
\path[<-](5) edge (6);
\end{tikzpicture}
\end{center}

Note that $c_\pi\a = \begin{smallmatrix} 0&1&1&2&1\\ &&&1&\end{smallmatrix}$. The set of indices with nonzero coefficient is $\{2,3,4,5,6\}$. Note that $s_2s_3s_6s_5s_4c_\pi\a = \begin{smallmatrix} 0&0&1&1&1\\ &&&1&\end{smallmatrix}$ and $s_3s_6s_5s_4s_2s_3s_6s_5s_4c_\pi\a = \a_3$. Thus let $\g$ be the curve below.
\begin{center}
\includegraphics[trim = 25mm 220mm 25mm 15mm, clip, width = 0.7\linewidth ]{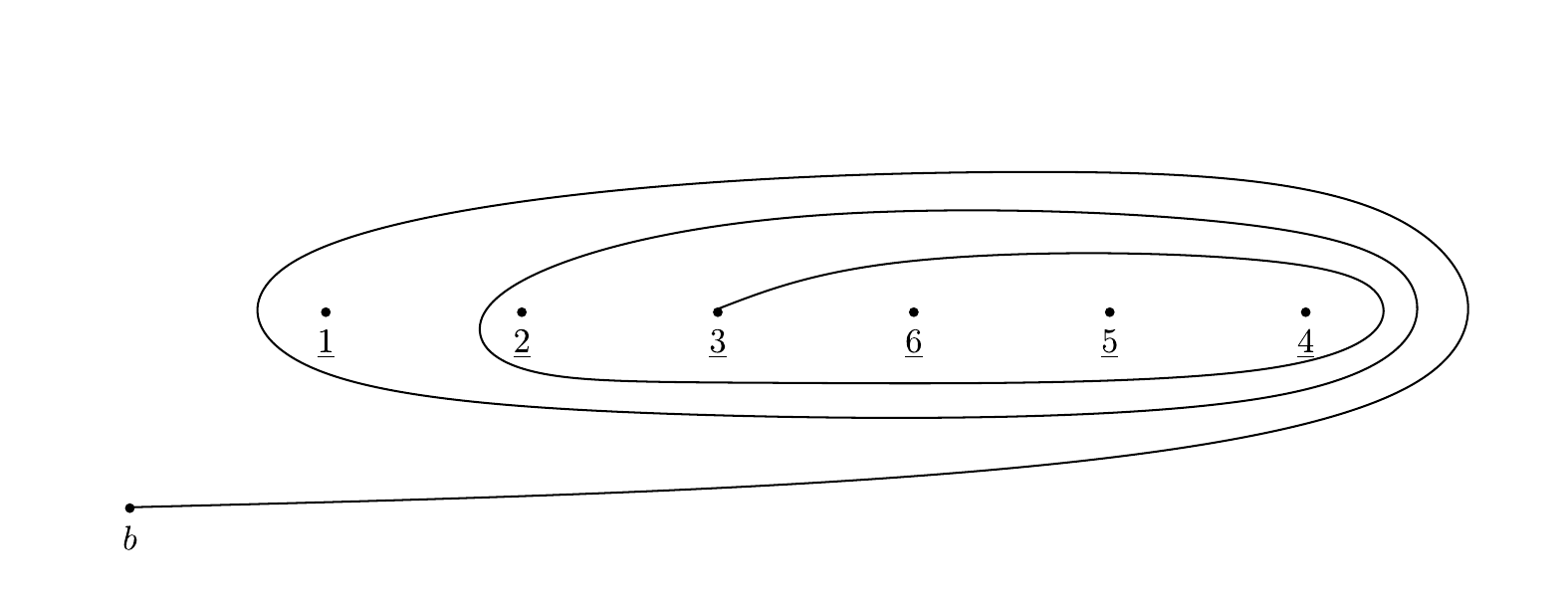}
\end{center}
Then $\pi(I_\g) = \r_3\r_6\r_5\r_4\r_2\r_3\r_6\r_5\r_4\r_1\r_2\r_3\r_6\r_5\r_4$, $\a_\pi(\g) = \a$, and $\g \in \G_{\pi,nd}$.
\end{Exa}

\section[Type E]{Type $\boldsymbol{E}$}\label{type_e}
In this section, we consider the quivers of type $E_6$ and $E_7$. For $n\in \{6,7\}$, let $\cE_n$ be a set of quivers whose underlying graph is
\begin{center}
\begin{tikzpicture}[scale =.8]		
\node[shape=circle,draw=black, minimum size = .4cm, inner sep = 4pt, outer sep = 0.5pt](7) at (-2,0) {$2$};
\node(1) at (0,0) {$\cdots$};
\node[shape=circle,draw=black, minimum size = .4cm, inner sep = 0pt, outer sep = 0.5pt](2) at (2,0) {$n-3$};
\node[shape=circle,draw=black, minimum size = .4cm, inner sep = 4pt, outer sep = 0.5pt](3) at (4,0) {$n$};
\node[shape=circle,draw=black, minimum size = .4cm, inner sep = 0pt, outer sep = 0.5pt](4) at (6,0) {$n-2$};
\node[shape=circle,draw=black, minimum size = .4cm, inner sep = 4pt, outer sep = 0.5pt](5) at (8,0) {$1$};
\node[shape=circle,draw=black, minimum size = .4cm, inner sep = 0pt, outer sep = 0.5pt](6) at (4,-2) {$n-1$};
\path[-](1) edge (2);
\path[-](2) edge (3);
\path[-](3) edge (4);
\path[-](4) edge (5);
\path[-](3) edge (6);
\path[-](7) edge (1);
\end{tikzpicture}
\end{center}

The roots of type $E$ are well-known; we refer to~\cite{Schif} for the list of positive roots of type $E$. Some of the positive roots of type $E$ are of the form $\a + \a_1$ or $\a+ \a_2$ where $\a$ is a positive root of a full subquiver of $E$. If there is a curve $\g$ such that the associated root is $\a$, then $\g$ can be extended a curve whose associated root is $\a+\a_1$ or $\a+\a_2$.

\begin{Lem}\label{lem_1_at_leaves}
Let $Q$ be a quiver in $\cE_n$ and $\a$ be a positive root of $Q$. Let $Q'$ be the full subquiver induced by the vertices $2, \ldots, n$. If $\a-\a_1 \in \{\a_{\pi'}(\g)\,|\,\g \in \G_{\pi',nd}\}$ where $\pi' \in P_{Q'}$, then $\a \in \{\a_{\pi}(\g)\,|\,\g \in \G_{\pi,nd}\}$ for some $\pi \in P_Q$. Similarly, if $Q'$ is a full subquiver induced by the vertices $1,3,\ldots, n$ and $\a-\a_2 \in \{\a_{\pi'}(\g)\,|\,\g \in \G_{\pi',nd}\}$ where $\pi' \in P_{Q'}$, then $\a$ is in $\{\a_\pi(\g)\,|\,\g \in \G_{\pi,nd}\}$ for some $\pi \in P_Q$.
\end{Lem}
\begin{proof}
First note that the vertices 1 and 2 are always either source or sink. If $\a-\a_1 \in \{\a_{\pi'}(\g)\,|\,\g \in \G_{\pi',nd}\}$ where $\pi' \in P_{Q'}$, then $\a-\a_1 \in \{\a_{\pi}(\g)\,|\,\g \in \G_{\pi,nd}\}$ for any $\pi \in \varphi^{-1}(\pi')$ by Lemma~\ref{lem_subquiver} and Remark~\ref{rmk_subquiver}. As the vertex 1 is either source or sink, there are permutations in $\varphi^{-1}(\pi')$ so that $\pi^{-1}(1)$ is either 1 or $n$. Let $\g' \in \G_{\pi,nd}$ such that $\a_{\pi}(\g) = \a-\a_1$ and~$\g'$ does not cross $\r_1$. Consider $\g$ obtained from $\g'$ by looping around~$\underline{1}$ at the end. Then $\a_\pi(\g) = s_1(\a_\pi(\g'))= s_1(\a-\a_1) = \a$ and as $\g'$ is non-decreasing, $\g$ is also non-decreasing.

Similarly, there exists a permutation $\pi \in P_Q$ such that $\pi^{-1}(2)$ is either $1$ or $n$ and $\g \in \G_{\pi, nd}$ such that $\a_\pi(\g) = \a - \a_2$. We can extend $\g$ to $\g'$ by looping around $\underline{2}$ to obtain $\g'$. Then $\a_\pi(\g') = s_2(\a-\a_2) = \a$.
\end{proof}

\subsection[Type E6]{Type $\boldsymbol{E_6}$}
All roots of type $E$ are known. Let $Q \in \cE_6$ in this subsection. Any full subquiver of $Q$ is either type $A$ or type $D$.
By Lemma~\ref{lem_subquiver} and Remark~\ref{rmk_subquiver}, we only need to consider the following roots:
\[\begin{smallmatrix}
1&2 & 2&2 & 1 \\
&&1&&
\end{smallmatrix},\,\, \begin{smallmatrix}
1&2 & 3&2 & 1 \\
&&1&&
\end{smallmatrix},\,\,\begin{smallmatrix}
1&2 & 3&2 & 1 \\
&&2&&
\end{smallmatrix}\,\, \]
\begin{Lem}\label{lem_e_6_roots}
Let $Q \in \cE_6$. If $\a$ is one of the roots above, then for any $\pi \in P_Q$, $c_\pi(\a)$ and $c_\pi(\a)$ are positive. Furthermore, either $c_\pi(\a)<_D \a$ or $c^{-1}_\pi(\a) <_D \a$.
\end{Lem}
\begin{proof}The height of all three roots are greater than~6. As $c_\pi$ is a composition of 6 simple reflections, for any root $\a$ above, $h(c_\pi\a) >0$ and $h(c_\pi^{-1}\a)>0$ for all $\pi$.
If $\a = \begin{smallmatrix}
1&2 & 2&2 & 1 \\
&&1&&
\end{smallmatrix}$, then $s_6\a >_D \a$, $s_4\a<_D\a$, and $s_i \a \leq_D \a$ for all other $i$. If $\a = \begin{smallmatrix}
1&2 & 3&2 & 1 \\
&&1&&
\end{smallmatrix}$, then $s_5\a >_D \a$, $s_6\a<_D\a$, and $s_i \a\leq_D\a$ for all other $i$. Thus in either case, $c_\pi \a <_D \a$ or $c_\pi^{-1} \a <_D\a$ by Lemma~\ref{lem_reducing_root}.

If $\a = \begin{smallmatrix}
1&2 & 3&2 & 1 \\
&&2&&
\end{smallmatrix}$, then $\a$ is the largest positive root of type $E_6$. As $c_\pi \a$ is a positive root that is different from $\a$, $c_\pi \a <_D \a$ and $c_\pi^{-1}\a <_D \a$ for any $\pi\in P_Q$.
\end{proof}

\begin{Prop} \label{prop_e_6}
Let $Q \in \cE_6$. Then $\bigcup_{\pi\in P_Q}\{\a_\pi(\g)\,|\,\g\in \G_{\pi,nd}\} = \D_Q^+$.
\end{Prop}
\begin{proof}
The full subquivers of $Q$ are type $A$ or $D$. By Lemma~\ref{lem_subquiver}, and Remark~\ref{rmk_subquiver}, and Propositions~\ref{prop_a_non_dec} and~\ref{prop_non_dec_curve_d}, the positive roots of these full subquivers are in $\{\a_\pi(\g)\,|\,\g\in \G_{\pi,nd}\}$ for any $\pi \in P_Q$.
Any root that is smaller than $\begin{smallmatrix}
1&2 & 2&2 & 1 \\
&&1&&
\end{smallmatrix}$ is type $A$, type $D$, or satisfies the assumptions of Lemma~\ref{lem_1_at_leaves}. Thus by Lemmas~\ref{lem_subquiver}, \ref{lem_1_at_leaves}, Remark~\ref{rmk_subquiver}, Propositions~\ref{prop_a_non_dec} and~\ref{prop_non_dec_curve_d}, such root is in $\{\a_\pi(\g)\,|\,\g\in \G_{\pi,nd}\}$ for some $\pi \in P_Q$. We just need to show that $\begin{smallmatrix}
1&2 & 2&2 & 1 \\
&&1&&
\end{smallmatrix}$, $\begin{smallmatrix}
1&2 & 3&2 & 1 \\
&&1&&
\end{smallmatrix}$, and $\begin{smallmatrix}
1&2 & 3&2 & 1 \\
&&2&&
\end{smallmatrix}$ are in $\bigcup_{\pi\in P_Q}\{\a_\pi(\g)\,|\,\g\in \G_{\pi,nd}\}$.

Let $\pi \in P_Q$ such that $\big\{\pi^{-1}(1), \pi^{-1}(2)\big\} = \{1,2\}, \{5,6\}$, or $\{1,6\}$. Such permutation exists as each of the vertices 1 and 2 is always either source or sink. By Lemma~\ref{lem_e_6_roots}, $c_\pi \a$ or $c_\pi^{-1}\a$ is smaller than $\a$ where $\a = \begin{smallmatrix}
1&2 & 2&2 & 1 \\
&&1&&
\end{smallmatrix}$. If $\big\{\pi^{-1}(1), \pi^{-1}(2)\big\} = \{1,6\}$, then by the proof of Lemma~\ref{lem_1_at_leaves}, every root that is smaller than $\a$ is in $\{\a_\pi(\g)\,|\,\g\in \G_{\pi,nd}\}$. As $c_\pi\a$ or $c_\pi^{-1} \a <_D \a$, the root $\a$ is in $\{\a_\pi(\g)\,|\,\g\in \G_{\pi,nd}\}$ by Lemmas~\ref{lem_cox_transf} and~\ref{lem_1_at_leaves}. Similarly, the roots $\begin{smallmatrix}
1&2 & 3&2 & 1 \\
&&1&&
\end{smallmatrix}, \begin{smallmatrix}
1&2 & 3&2 & 1 \\
&&2&&
\end{smallmatrix}$ are in $\{\a_\pi(\g)|\g\in \G_{\pi,nd}\}$.

If $\big\{\pi^{-1}(1), \pi^{-1}(2)\big\} = \{1,2\}$ or $\{5,6\}$, then $c_\pi\a = c_{(1\,\, 2) \pi}\a$ as $s_1$ and $s_2$ commute. Note that every root that is smaller than $\a$ is in $\{\a_\pi(\g)\,|\,\g\in \G_{\pi,nd}\}$ or $ \{\a_{(1 \,\, 2)\pi}(\g)\,|\,\g\in \G_{(1\,\,2)\pi,nd}\}$. Thus $c_\pi\a = c_{(1\,\,2)\pi}\a$ is in either $\{\a_\pi(\g)\,|\,\g\in \G_{\pi,nd}\}$ or $ \{\a_{(1 \,\, 2)\pi}(\g)\,|\,\g\in \G_{(1\,\,2)\pi,nd}\}$. Therefore $\a$ is in $\{\a_\pi(\g)\,|\,\g\in \G_{\pi,nd}\} \cup \{\a_{(1 \,\, 2)\pi}(\g)\,|\,\g\in \G_{(1\,\,2)\pi,nd}\}$ and so are the roots $\begin{smallmatrix}
1&2 & 3&2 & 1 \\
&&1&&
\end{smallmatrix}$ and $\begin{smallmatrix}
1&2 & 3&2 & 1 \\
&&2&&
\end{smallmatrix}$.
\end{proof}

\subsection[Type E7]{Type $\boldsymbol{E_7}$}\label{section6.2}
Of the positive roots of $E_7$, many of them are either type $A$, $D$, or $E_6$. Also by Lemma~\ref{lem_1_at_leaves}, many of the roots of type $E_7$ are known to be in $\bigcup_{\pi \in P_Q} \{\a_\pi(\g)\,|\,\g\in \G_{\pi,nd}\}$ if the roots below are in $\bigcup_{\pi \in P_Q} \{\a_\pi(\g)\,|\,\g\in \G_{\pi,nd}\}$.

\begin{tikzpicture}
\node at (-.5,0){1};
\node(1) at (0,0) {$2$};
\node(2) at (.5,0) {$2$};
\node(3) at (1,0) {$2$};
\node(4) at (1.5,0) {$2$};
\node(5) at (2,0) {$1$};
\node(6) at (1,-.5) {$1$};

\node at (3,0){1};
\node(1) at (3.5,0) {$2$};
\node(2) at (4,0) {$2$};
\node(3) at (4.5,0) {$3$};
\node(4) at (5,0) {$2$};
\node(5) at (5.5,0) {$1$};
\node(6) at (4.5,-.5) {$1$};

\node at (6.5,0){1};
\node(1) at (7,0) {$2$};
\node(2) at (7.5,0) {$3$};
\node(3) at (8,0) {$3$};
\node(4) at (8.5,0) {$2$};
\node(5) at (9,0) {$1$};
\node(6) at (8,-.5) {$1$};

\node at (10,0){1};
\node(1) at (10.5,0) {$2$};
\node(2) at (11,0) {$2$};
\node(3) at (11.5,0) {$3$};
\node(4) at (12,0) {$2$};
\node(5) at (12.5,0) {$1$};
\node(6) at (11.5,-.5) {$2$};
\end{tikzpicture}

\begin{tikzpicture}
\node at (-.5,0){1};
\node(1) at (0,0) {$2$};
\node(2) at (.5,0) {$3$};
\node(3) at (1,0) {$3$};
\node(4) at (1.5,0) {$2$};
\node(5) at (2,0) {$1$};
\node(6) at (1,-.5) {$2$};

\node at (3,0){1};
\node(1) at (3.5,0) {$2$};
\node(2) at (4,0) {$3$};
\node(3) at (4.5,0) {$4$};
\node(4) at (5,0) {$2$};
\node(5) at (5.5,0) {$1$};
\node(6) at (4.5,-.5) {$2$};

\node at (6.5,0){1};
\node(1) at (7,0) {$2$};
\node(2) at (7.5,0) {$3$};
\node(3) at (8,0) {$4$};
\node(4) at (8.5,0) {$3$};
\node(5) at (9,0) {$1$};
\node(6) at (8,-.5) {$2$};

\node at (10,0){1};
\node(1) at (10.5,0) {$2$};
\node(2) at (11,0) {$3$};
\node(3) at (11.5,0) {$4$};
\node(4) at (12,0) {$3$};
\node(5) at (12.5,0) {$2$};
\node(6) at (11.5,-.5) {$2$};
\end{tikzpicture}

\begin{Lem}\label{lem_E_7_6_roots}
Let $Q \in \cE_7$ and $\pi \in P_Q$. If \[\a \in \left\{ \begin{smallmatrix}
1&2 &2 &2&2 & 1 \\
&&&1&&
\end{smallmatrix}, \begin{smallmatrix}
1&2 &2 &3&2 & 1 \\
&&&2&&
\end{smallmatrix}, \begin{smallmatrix}
1&2 &3 &3&2 & 1 \\
&&&2&&
\end{smallmatrix},\begin{smallmatrix}
1&2 &3 &4&2 & 1 \\
&&&2&&
\end{smallmatrix},\begin{smallmatrix}
1&2 &3 &4&3 & 1 \\
&&&2&&
\end{smallmatrix},\begin{smallmatrix}
1&2 &3 &4&3 & 2 \\
&&&2&&
\end{smallmatrix} \right\},\] then $c_\pi\a$ and $c_\pi^{-1}\a$ are positive and $c_\pi \a<_D\a$ or $c_\pi^{-1}\a <_D\a$.
\end{Lem}
\begin{proof}
As the height of each roots is greater than $7$, $c_\pi\a$ and $c_\pi^{-1}\a$ are positive.
The root $ \a = \begin{smallmatrix}
1&2 &3 &4&3 & 2 \\
&&&2&&
\end{smallmatrix}$ is the largest root of type $E_7$. Thus $c_\pi\a <_D \a$ for any $\pi \in P_Q$. For other roots, it suffices to show that for each root there is a unique $k \in [7]$ such that $s_k \a >_D \a$ and an adjacent vertex $\ell$ such that $s_\ell \a<_D\a$ due to Lemma~\ref{lem_reducing_root}.

For instance, let $\a= \begin{smallmatrix}
1&2 &3 &3&2 & 1 \\
&&&2&&
\end{smallmatrix}.$ Note that $s_7\a=\a+\a_7$, $s_6\a = \a -\a_6$, $s_4\a = \a -\a_4$, and $s_i\a = \a$ for other $i$. Similarly, all the other roots have a unique index such that $s_k\a = \a+\a_k$ and an adjacent vertex $j$ such that $s_j\a =\a-\a_j.$
\end{proof}

Note that the roots $\!\begin{smallmatrix}
1&2 &2 &2&2 & 1 \\
&&&1&&
\end{smallmatrix}\!$ and $\!\begin{smallmatrix}
1&2 &2 &3&2 & 1 \\
&&&2&&
\end{smallmatrix}\!$ reduce to a root that we know to be in $\bigcup_{\pi\in P_Q}\!\{\a_\pi(\g)\allowbreak |\, \g\in \G_{\pi,nd}\}$. To show the others, we first need to look at the roots $\begin{smallmatrix}
1&2 &3 &3&2 & 1 \\
&&&1&&
\end{smallmatrix}$ and $\begin{smallmatrix}
1&2 &2 &3&2 & 1 \\
&&&1&&
\end{smallmatrix}$.

\begin{Lem} \label{lem_e_7_root_1233211}
Let $Q \in \cE_7$ and $\a =\begin{smallmatrix}
1&2 &3 &3&2 & 1 \\
&&&1&&
\end{smallmatrix}.$ Then $\a \in \bigcup_{\pi\in P_Q} \{\a_\pi(\g)\,|\, \g\in \G_{\pi,nd}\}$.
\end{Lem}
\begin{proof}
Note that $s_4\a= \a-\a_4$, $s_6\a=\a+\a_6$, and $s_i\a=\a$ for $i \neq 4,6$. As the vertex $4$ is not adjacent to the vertex $6$, we cannot apply Lemma~\ref{lem_reducing_root}.
To show that $\a$ is the associated root of a non-decreasing non-self-crossing admissible curve for some $\pi\in P_Q$, we divide the quivers into different cases. First, consider any quiver $Q$ of the form below:
\begin{center}
\begin{tikzpicture}[scale =.7]		
\node[shape=circle,draw=black, minimum size = .4cm, inner sep = 4pt, outer sep = 0.5pt](7) at (-2,0) {$2$};
\node[shape=circle,draw=black, minimum size = .4cm, inner sep = 4pt, outer sep = 0.5pt](1) at (0,0) {$3$};
\node[shape=circle,draw=black, minimum size = .4cm, inner sep = 4pt, outer sep = 0.5pt](2) at (2,0) {$4$};
\node[shape=circle,draw=black, minimum size = .4cm, inner sep = 4pt, outer sep = 0.5pt](3) at (4,0) {$7$};
\node[shape=circle,draw=black, minimum size = .4cm, inner sep = 4pt, outer sep = 0.5pt](4) at (6,0) {$5$};
\node[shape=circle,draw=black, minimum size = .4cm, inner sep = 4pt, outer sep = 0.5pt](5) at (8,0) {$1$};
\node[shape=circle,draw=black, minimum size = .4cm, inner sep = 4pt, outer sep = 0.5pt](6) at (4,-2) {$6$};
\path[-](1) edge (2);
\path[->](2) edge (3);
\path[-](3) edge (4);
\path[-](4) edge (5);
\path[->](3) edge (6);
\path[-](7) edge (1);
\end{tikzpicture}
\end{center}

Due to the orientation of $Q$, for every $\pi \in P_Q$, $\pi^{-1}(4) < \pi^{-1}(7) < \pi^{-1}(6)$. Thus $c_\pi\a \leq_D s_4s_7s_6\a =\begin{smallmatrix}
1&2 &2 &2&2 & 1 \\
&&&1&&
\end{smallmatrix}$. This root and all the smaller roots are in $\{\a_\pi(\g)\,|\, \g\in \G_{\pi,nd}\}$ or $\{\a_{(1\,\,2)\pi}(\g)\,|\allowbreak \g\in \G_{(1\,\,2)\pi,nd}\}$ for $\pi \in P_Q$ such that $\big\{\pi^{-1}(1), \pi^{-1}(2)\big\} = \{1,2\},\,\{1,7\},$ or $\{6,7\},$ just like $\begin{smallmatrix}
1&2 &2&2 & 1 \\
&&1&&
\end{smallmatrix}$ in type $E_6$ case. Thus by Lemma~\ref{lem_cox_transf}, $\a \in \bigcup_{\pi\in P_Q} \{\a_\pi(\g)\,|\, \g\in \G_{\pi,nd}\}$.

Similarly, if $Q$ is a quiver with the following orientations
\begin{center}
\begin{tikzpicture}[scale =.7]		
\node[shape=circle,draw=black, minimum size = .4cm, inner sep = 4pt, outer sep = 0.5pt](7) at (-2,0) {$2$};
\node[shape=circle,draw=black, minimum size = .4cm, inner sep = 4pt, outer sep = 0.5pt](1) at (0,0) {$3$};
\node[shape=circle,draw=black, minimum size = .4cm, inner sep = 4pt, outer sep = 0.5pt](2) at (2,0) {$4$};
\node[shape=circle,draw=black, minimum size = .4cm, inner sep = 4pt, outer sep = 0.5pt](3) at (4,0) {$7$};
\node[shape=circle,draw=black, minimum size = .4cm, inner sep = 4pt, outer sep = 0.5pt](4) at (6,0) {$5$};
\node[shape=circle,draw=black, minimum size = .4cm, inner sep = 4pt, outer sep = 0.5pt](5) at (8,0) {$1$};
\node[shape=circle,draw=black, minimum size = .4cm, inner sep = 4pt, outer sep = 0.5pt](6) at (4,-2) {$6$};
\path[-](1) edge (2);
\path[<-](2) edge (3);
\path[-](3) edge (4);
\path[-](4) edge (5);
\path[<-](3) edge (6);
\path[-](7) edge (1);
\end{tikzpicture}
\end{center}
 then $\pi^{-1}(4) > \pi^{-1}(7) > \pi^{-1}(6)$ and $c_\pi^{-1}\a <_D \a$.

Now consider a quiver of the orientation
\begin{center}
\begin{tikzpicture}[scale =.7]		
\node[shape=circle,draw=black, minimum size = .4cm, inner sep = 4pt, outer sep = 0.5pt](7) at (-2,0) {$2$};
\node[shape=circle,draw=black, minimum size = .4cm, inner sep = 4pt, outer sep = 0.5pt](1) at (0,0) {$3$};
\node[shape=circle,draw=black, minimum size = .4cm, inner sep = 4pt, outer sep = 0.5pt](2) at (2,0) {$4$};
\node[shape=circle,draw=black, minimum size = .4cm, inner sep = 4pt, outer sep = 0.5pt](3) at (4,0) {$7$};
\node[shape=circle,draw=black, minimum size = .4cm, inner sep = 4pt, outer sep = 0.5pt](4) at (6,0) {$5$};
\node[shape=circle,draw=black, minimum size = .4cm, inner sep = 4pt, outer sep = 0.5pt](5) at (8,0) {$1$};
\node[shape=circle,draw=black, minimum size = .4cm, inner sep = 4pt, outer sep = 0.5pt](6) at (4,-2) {$6$};
\path[-](1) edge (2);
\path[<-](2) edge (3);
\path[-](3) edge (4);
\path[-](4) edge (5);
\path[->](3) edge (6);
\path[-](7) edge (1);
\end{tikzpicture}
\end{center}

Here $c_\pi \a $ and $c_\pi^{-1}\a $ are not comparable to $\a$. Thus to find $\pi \in P_Q$ and $\g \in \G_{\pi,nd}$ such that $\a_\pi(\g) = \a$, we use Sage to find a permutation $\pi$ and non-decreasing non-self-crossing admissible curve $\g$ such that $\a_\pi(\g) = \a$. Refer to Tables \ref{E_7_table_1233211_1} and \ref{E_7_table_1233211_2}. The quivers of following orientation
\begin{center}
\begin{tikzpicture}[scale =.7]		
\node[shape=circle,draw=black, minimum size = .4cm, inner sep = 4pt, outer sep = 0.5pt](7) at (-2,0) {$2$};
\node[shape=circle,draw=black, minimum size = .4cm, inner sep = 4pt, outer sep = 0.5pt](1) at (0,0) {$3$};
\node[shape=circle,draw=black, minimum size = .4cm, inner sep = 4pt, outer sep = 0.5pt](2) at (2,0) {$4$};
\node[shape=circle,draw=black, minimum size = .4cm, inner sep = 4pt, outer sep = 0.5pt](3) at (4,0) {$7$};
\node[shape=circle,draw=black, minimum size = .4cm, inner sep = 4pt, outer sep = 0.5pt](4) at (6,0) {$5$};
\node[shape=circle,draw=black, minimum size = .4cm, inner sep = 4pt, outer sep = 0.5pt](5) at (8,0) {$1$};
\node[shape=circle,draw=black, minimum size = .4cm, inner sep = 4pt, outer sep = 0.5pt](6) at (4,-2) {$6$};
\path[-](1) edge (2);
\path[->](2) edge (3);
\path[-](3) edge (4);
\path[-](4) edge (5);
\path[<-](3) edge (6);
\path[-](7) edge (1);
\end{tikzpicture}
\end{center}
are obtained from the quivers of orientation above by reversing all arrows. Then the associated root of the mirror images of the curves in the tables would be~$\a$.
\end{proof}

\begin{table}[t]\centering
\begin{tabular}{c|c|c}
\hline
Quiver, $Q$ & $\pi \in P_Q$ & $\gamma \in \Gamma_{\pi, nd}$\\
\hline
\begin{tikzpicture}[scale =.35]		
\node(7) at (-2,0) {$\cdot$};
\node(1) at (0,0) {$\cdot$};
\node(2) at (2,0) {$\cdot$};
\node(3) at (4,0) {$\cdot$};
\node(4) at (6,0) {$\cdot$};
\node(5) at (8,0) {$\cdot$};
\node(6) at (4,-2) {$\cdot$};
\path[<-](2) edge (3);
\path[->](3) edge (6);

\path[->](1) edge (2);
\path[->](3) edge (4);
\path[->](4) edge (5);
\path[->](7) edge (1);
\end{tikzpicture}&
$\left(\begin{smallmatrix}
1&2&3&4&5&6&7 \\
2& 3& 7& 6& 5& 4 &1\\
\end{smallmatrix}\right)$
&
\begin{tikzpicture}[rounded corners,scale=.1]
\node at (51,0){$\cdot$};
\node at (5,0){$\cdot$};
\node at (11,0){$\cdot$};
\node at (41,0){$\cdot$};
\node at (31,0){$\cdot$};
\node at (29,0){$\cdot$};
\node at (27,0){$\cdot$};
\draw (9,0) to (9,1) to (13,1) to (13,0);
\draw (35,0) to (35,3) to (47,3) to (47,0);
\draw (7,0) to (7,2) to (15,2) to (15,0);
\draw (23,0) to (23,6) to (49,6) to (49,0);
\draw (37,0) to (37,2) to (45,2) to (45,0);
\draw (1,-9) to (1,5) to (19,5) to (19,0);
\draw (25,0) to (25,2) to (33,2) to (33,0);
\draw (39,0) to (39,1) to (43,1) to (43,0);
\draw (3,0) to (3,4) to (17,4) to (17,0);
\draw (21,0) to (21,7) to (53,7) to (53,0);
\draw (27,0) to (27,-2) to (33,-2) to (33,0);
\draw (25,0) to (25,-3) to (35,-3) to (35,0);
\draw (9,0) to (9,-10) to (53,-10) to (53,0);
\draw (19,0) to (19,-6) to (43,-6) to (43,0);
\draw (13,0) to (13,-9) to (49,-9) to (49,0);
\draw (23,0) to (23,-4) to (37,-4) to (37,0);
\draw (3,0) to (3,-1) to (7,-1) to (7,0);
\draw (17,0) to (17,-7) to (45,-7) to (45,0);
\draw (21,0) to (21,-5) to (39,-5) to (39,0);
\draw (15,0) to (15,-8) to (47,-8) to (47,0);
\end{tikzpicture}\\
\hline

\begin{tikzpicture}[scale =.35]		
\node(7) at (-2,0) {$\cdot$};
\node(1) at (0,0) {$\cdot$};
\node(2) at (2,0) {$\cdot$};
\node(3) at (4,0) {$\cdot$};
\node(4) at (6,0) {$\cdot$};
\node(5) at (8,0) {$\cdot$};
\node(6) at (4,-2) {$\cdot$};
\path[<-](2) edge (3);
\path[->](3) edge (6);

\path[<-](7) edge (1);
\path[->](1) edge (2);
\path[->](3) edge (4);
\path[->](4) edge (5);
\end{tikzpicture}&
$\left(\begin{smallmatrix}
1&2&3&4&5&6&7 \\
3& 7& 6& 5& 4 &2&1\\
\end{smallmatrix}\right)$ &\begin{tikzpicture}[rounded corners,scale=.1]
\node at (49,0){$\cdot$};
\node at (47,0){$\cdot$};
\node at (9,0){$\cdot$};
\node at (35,0){$\cdot$};
\node at (27,0){$\cdot$};
\node at (25,0){$\cdot$};
\node at (23,0){$\cdot$};
\draw (3,0) to (3,4) to (15,4) to (15,0);
\draw (31,0) to (31,3) to (39,3) to (39,0);
\draw (29,0) to (29,4) to (41,4) to (41,0);
\draw (21,0) to (21,6) to (43,6) to (43,0);
\draw (1,-9) to (1,5) to (17,5) to (17,0);
\draw (5,0) to (5,3) to (13,3) to (13,0);
\draw (33,0) to (33,1) to (37,1) to (37,0);
\draw (45,0) to (45,2) to (51,2) to (51,0);
\draw (7,0) to (7,1) to (11,1) to (11,0);
\draw (19,0) to (19,7) to (53,7) to (53,0);
\draw (17,0) to (17,-4) to (37,-4) to (37,0);
\draw (7,0) to (7,-8) to (45,-8) to (45,0);
\draw (11,0) to (11,-7) to (43,-7) to (43,0);
\draw (23,0) to (23,-1) to (29,-1) to (29,0);
\draw (5,0) to (5,-9) to (51,-9) to (51,0);
\draw (3,0) to (3,-10) to (53,-10) to (53,0);
\draw (21,0) to (21,-2) to (31,-2) to (31,0);
\draw (15,0) to (15,-5) to (39,-5) to (39,0);
\draw (13,0) to (13,-6) to (41,-6) to (41,0);
\draw (19,0) to (19,-3) to (33,-3) to (33,0);
\end{tikzpicture}\\
\hline
\begin{tikzpicture}[scale =.35]		
\node(7) at (-2,0) {$\cdot$};
\node(1) at (0,0) {$\cdot$};
\node(2) at (2,0) {$\cdot$};
\node(3) at (4,0) {$\cdot$};
\node(4) at (6,0) {$\cdot$};
\node(5) at (8,0) {$\cdot$};
\node(6) at (4,-2) {$\cdot$};
\path[<-](2) edge (3);
\path[->](3) edge (6);

\path[->](7) edge (1);
\path[<-](1) edge (2);
\path[->](3) edge (4);
\path[->](4) edge (5);
\end{tikzpicture}&
$\left(\begin{smallmatrix}
1&2&3&4&5&6&7 \\
2& 7& 6& 5& 4 &3&1\\
\end{smallmatrix}\right)$
&
\begin{tikzpicture}[rounded corners,scale=.1]
\node at (37,0){$\cdot$};
\node at (5,0){$\cdot$};
\node at (31,0){$\cdot$};
\node at (25,0){$\cdot$};
\node at (19,0){$\cdot$};
\node at (17,0){$\cdot$};
\node at (15,0){$\cdot$};
\draw (11,0) to (11,5) to (39,5) to (39,0);
\draw (13,0) to (13,4) to (35,4) to (35,0);
\draw (21,0) to (21,2) to (29,2) to (29,0);
\draw (15,0) to (15,3) to (33,3) to (33,0);
\draw (1,-9) to (1,2) to (9,2) to (9,0);
\draw (23,0) to (23,1) to (27,1) to (27,0);
\draw (3,0) to (3,1) to (7,1) to (7,0);
\draw (7,0) to (7,-5) to (35,-5) to (35,0);
\draw (9,0) to (9,-4) to (27,-4) to (27,0);
\draw (29,0) to (29,-1) to (33,-1) to (33,0);
\draw (11,0) to (11,-3) to (23,-3) to (23,0);
\draw (3,0) to (3,-6) to (39,-6) to (39,0);
\draw (13,0) to (13,-2) to (21,-2) to (21,0);
\end{tikzpicture}\\
\hline

\begin{tikzpicture}[scale =.35]		
\node(7) at (-2,0) {$\cdot$};
\node(1) at (0,0) {$\cdot$};
\node(2) at (2,0) {$\cdot$};
\node(3) at (4,0) {$\cdot$};
\node(4) at (6,0) {$\cdot$};
\node(5) at (8,0) {$\cdot$};
\node(6) at (4,-2) {$\cdot$};
\path[<-](2) edge (3);
\path[->](3) edge (6);

\path[->](7) edge (1);
\path[->](1) edge (2);
\path[<-](3) edge (4);
\path[->](4) edge (5);
\end{tikzpicture}&
$\left(\begin{smallmatrix}
1&2&3&4&5&6&7 \\
2& 3& 5&7& 6& 4 &1\\
\end{smallmatrix}\right)$
&\begin{tikzpicture}[rounded corners,scale=.1]
\node at (47,0){$\cdot$};
\node at (7,0){$\cdot$};
\node at (11,0){$\cdot$};
\node at (39,0){$\cdot$};
\node at (27,0){$\cdot$};
\node at (31,0){$\cdot$};
\node at (29,0){$\cdot$};
\draw (9,0) to (9,1) to (13,1) to (13,0);
\draw (1,-9) to (1,6) to (19,6) to (19,0);
\draw (23,0) to (23,7) to (49,7) to (49,0);
\draw (33,0) to (33,4) to (45,4) to (45,0);
\draw (35,0) to (35,2) to (43,2) to (43,0);
\draw (5,0) to (5,3) to (15,3) to (15,0);
\draw (37,0) to (37,1) to (41,1) to (41,0);
\draw (3,0) to (3,5) to (17,5) to (17,0);
\draw (25,0) to (25,1) to (29,1) to (29,0);
\draw (21,0) to (21,8) to (51,8) to (51,0);
\draw (5,0) to (5,-1) to (9,-1) to (9,0);
\draw (19,0) to (19,-5) to (41,-5) to (41,0);
\draw (13,0) to (13,-8) to (49,-8) to (49,0);
\draw (17,0) to (17,-6) to (43,-6) to (43,0);
\draw (25,0) to (25,-2) to (33,-2) to (33,0);
\draw (23,0) to (23,-3) to (35,-3) to (35,0);
\draw (15,0) to (15,-7) to (45,-7) to (45,0);
\draw (3,0) to (3,-9) to (51,-9) to (51,0);
\draw (21,0) to (21,-4) to (37,-4) to (37,0);
\end{tikzpicture}\\
\hline

\begin{tikzpicture}[scale =.35]		
\node(7) at (-2,0) {$\cdot$};
\node(1) at (0,0) {$\cdot$};
\node(2) at (2,0) {$\cdot$};
\node(3) at (4,0) {$\cdot$};
\node(4) at (6,0) {$\cdot$};
\node(5) at (8,0) {$\cdot$};
\node(6) at (4,-2) {$\cdot$};
\path[<-](2) edge (3);
\path[->](3) edge (6);

\path[<-](7) edge (1);
\path[<-](1) edge (2);
\path[->](3) edge (4);
\path[->](4) edge (5);
\end{tikzpicture}&
$\left(\begin{smallmatrix}
1&2&3&4&5&6&7 \\
7& 6&5& 4 &3&2&1\\
\end{smallmatrix}\right)$
&\begin{tikzpicture}[rounded corners,scale=.1]
\node at (35,0){$\cdot$};
\node at (33,0){$\cdot$};
\node at (27,0){$\cdot$};
\node at (21,0){$\cdot$};
\node at (11,0){$\cdot$};
\node at (9,0){$\cdot$};
\node at (7,0){$\cdot$};
\draw (3,0) to (3,4) to (31,4) to (31,0);
\draw (5,0) to (5,2) to (13,2) to (13,0);
\draw (1,0) to (1,5) to (37,5) to (37,0);
\draw (15,0) to (15,3) to (29,3) to (29,0);
\draw (17,0) to (17,2) to (25,2) to (25,0);
\draw (19,0) to (19,1) to (23,1) to (23,-9);
\draw (25,0) to (25,-1) to (29,-1) to (29,0);
\draw (7,0) to (7,-2) to (13,-2) to (13,0);
\draw (31,0) to (31,-2) to (37,-2) to (37,0);
\draw (5,0) to (5,-3) to (15,-3) to (15,0);
\draw (1,0) to (1,-5) to (19,-5) to (19,0);
\draw (3,0) to (3,-4) to (17,-4) to (17,0);
\end{tikzpicture}
\\

\hline
\begin{tikzpicture}[scale =.35]		
\node(7) at (-2,0) {$\cdot$};
\node(1) at (0,0) {$\cdot$};
\node(2) at (2,0) {$\cdot$};
\node(3) at (4,0) {$\cdot$};
\node(4) at (6,0) {$\cdot$};
\node(5) at (8,0) {$\cdot$};
\node(6) at (4,-2) {$\cdot$};
\path[<-](2) edge (3);
\path[->](3) edge (6);

\path[<-](7) edge (1);
\path[->](1) edge (2);
\path[<-](3) edge (4);
\path[->](4) edge (5);
\end{tikzpicture}&
$\left(\begin{smallmatrix}
1&2&3&4&5&6&7 \\
3& 7& 6&5& 4 &2&1\\
\end{smallmatrix}\right)$
&\begin{tikzpicture}[rounded corners,scale=.1]
\node at (49,0){$\cdot$};
\node at (47,0){$\cdot$};
\node at (9,0){$\cdot$};
\node at (33,0){$\cdot$};
\node at (11,0){$\cdot$};
\node at (27,0){$\cdot$};
\node at (25,0){$\cdot$};
\draw (3,0) to (3,5) to (17,5) to (17,0);
\draw (45,0) to (45,2) to (51,2) to (51,0);
\draw (7,0) to (7,2) to (13,2) to (13,0);
\draw (5,0) to (5,4) to (15,4) to (15,0);
\draw (23,0) to (23,6) to (41,6) to (41,0);
\draw (1,-9) to (1,6) to (19,6) to (19,0);
\draw (21,0) to (21,7) to (43,7) to (43,0);
\draw (31,0) to (31,1) to (35,1) to (35,0);
\draw (25,0) to (25,5) to (39,5) to (39,0);
\draw (29,0) to (29,3) to (37,3) to (37,0);
\draw (17,0) to (17,-4) to (37,-4) to (37,0);
\draw (19,0) to (19,-3) to (35,-3) to (35,0);
\draw (23,0) to (23,-1) to (29,-1) to (29,0);
\draw (5,0) to (5,-8) to (45,-8) to (45,0);
\draw (21,0) to (21,-2) to (31,-2) to (31,0);
\draw (3,0) to (3,-9) to (51,-9) to (51,0);
\draw (7,0) to (7,-7) to (43,-7) to (43,0);
\draw (13,0) to (13,-6) to (41,-6) to (41,0);
\draw (15,0) to (15,-5) to (39,-5) to (39,0);
\end{tikzpicture}
\\
\hline
\begin{tikzpicture}[scale =.35]		
\node(7) at (-2,0) {$\cdot$};
\node(1) at (0,0) {$\cdot$};
\node(2) at (2,0) {$\cdot$};
\node(3) at (4,0) {$\cdot$};
\node(4) at (6,0) {$\cdot$};
\node(5) at (8,0) {$\cdot$};
\node(6) at (4,-2) {$\cdot$};
\path[<-](2) edge (3);
\path[->](3) edge (6);

\path[->](7) edge (1);
\path[<-](1) edge (2);
\path[<-](3) edge (4);
\path[->](4) edge (5);
\end{tikzpicture}&
$\left(\begin{smallmatrix}
1&2&3&4&5&6&7 \\
2& 5& 7& 6& 4 &3&1\\
\end{smallmatrix}\right)$
&\begin{tikzpicture}[rounded corners,scale=.1]
\node at (39,0){$\cdot$};
\node at (3,0){$\cdot$};
\node at (33,0){$\cdot$};
\node at (27,0){$\cdot$};
\node at (7,0){$\cdot$};
\node at (17,0){$\cdot$};
\node at (11,0){$\cdot$};
\draw (9,0) to (9,1) to (13,1) to (13,0);
\draw (25,0) to (25,6) to (49,6) to (49,0);
\draw (1,-9) to (1,5) to (19,5) to (19,0);
\draw (23,0) to (23,7) to (51,7) to (51,0);
\draw (35,0) to (35,2) to (43,2) to (43,0);
\draw (5,0) to (5,3) to (15,3) to (15,0);
\draw (31,0) to (31,4) to (45,4) to (45,0);
\draw (37,0) to (37,1) to (41,1) to (41,0);
\draw (29,0) to (29,5) to (47,5) to (47,0);
\draw (21,0) to (21,8) to (53,8) to (53,-9);
\draw (25,0) to (25,-1) to (29,-1) to (29,0);
\draw (23,0) to (23,-2) to (31,-2) to (31,0);
\draw (9,0) to (9,-8) to (47,-8) to (47,0);
\draw (5,0) to (5,-9) to (49,-9) to (49,0);
\draw (21,0) to (21,-3) to (35,-3) to (35,0);
\draw (1,0) to (1,-10) to (51,-10) to (51,0);
\draw (15,0) to (15,-5) to (41,-5) to (41,0);
\draw (19,0) to (19,-4) to (37,-4) to (37,0);
\draw (11,0) to (11,-7) to (45,-7) to (45,0);
\draw (13,0) to (13,-6) to (43,-6) to (43,0);
\end{tikzpicture}\\
\hline
\begin{tikzpicture}[scale =.35]		
\node(7) at (-2,0) {$\cdot$};
\node(1) at (0,0) {$\cdot$};
\node(2) at (2,0) {$\cdot$};
\node(3) at (4,0) {$\cdot$};
\node(4) at (6,0) {$\cdot$};
\node(5) at (8,0) {$\cdot$};
\node(6) at (4,-2) {$\cdot$};
\path[<-](2) edge (3);
\path[->](3) edge (6);
\path[<-](7) edge (1);
\path[<-](1) edge (2);
\path[<-](3) edge (4);
\path[->](4) edge (5);
\end{tikzpicture}&
$\left(\begin{smallmatrix}
1&2&3&4&5&6&7 \\
 5& 7& 6& 4 &3&2&1\\
\end{smallmatrix}\right)$ &\begin{tikzpicture}[rounded corners,scale=.1]
\node at (49,0){$\cdot$};
\node at (47,0){$\cdot$};
\node at (45,0){$\cdot$};
\node at (33,0){$\cdot$};
\node at (9,0){$\cdot$};
\node at (25,0){$\cdot$};
\node at (23,0){$\cdot$};
\draw (3,0) to (3,3) to (15,3) to (15,0);
\draw (21,0) to (21,5) to (41,5) to (41,0);
\draw (1,-9) to (1,4) to (17,4) to (17,0);
\draw (29,0) to (29,2) to (37,2) to (37,0);
\draw (5,0) to (5,2) to (13,2) to (13,0);
\draw (27,0) to (27,3) to (39,3) to (39,0);
\draw (7,0) to (7,1) to (11,1) to (11,0);
\draw (31,0) to (31,1) to (35,1) to (35,0);
\draw (43,0) to (43,2) to (51,2) to (51,0);
\draw (19,0) to (19,6) to (53,6) to (53,0);
\draw (19,0) to (19,-3) to (31,-3) to (31,0);
\draw (5,0) to (5,-9) to (51,-9) to (51,0);
\draw (11,0) to (11,-7) to (41,-7) to (41,0);
\draw (21,0) to (21,-2) to (29,-2) to (29,0);
\draw (3,0) to (3,-10) to (53,-10) to (53,0);
\draw (17,0) to (17,-4) to (35,-4) to (35,0);
\draw (7,0) to (7,-8) to (43,-8) to (43,0);
\draw (23,0) to (23,-1) to (27,-1) to (27,0);
\draw (15,0) to (15,-5) to (37,-5) to (37,0);
\draw (13,0) to (13,-6) to (39,-6) to (39,0);
\end{tikzpicture}\\
\hline
\end{tabular}

\caption{Given a quiver and a permutation in $P_Q$, the given non-self-crossing admissible curve~$\g$ is non-decreasing and $\a_\pi(\g) ={1\,\,2\,\,3\,\,3\,\,2\,\,1}\atop {\hspace{1.3cm}1}$.}
\label{E_7_table_1233211_1}
\end{table}

\begin{table}[t]\centering
\begin{tabular}{c|c|c}
\hline
Quiver, $Q$ & $\pi \in P_Q$ & $\gamma \in \Gamma_{\pi, nd}$\\
\hline

\hline
\begin{tikzpicture}[scale =.35]		
\node(7) at (-2,0) {$\cdot$};
\node(1) at (0,0) {$\cdot$};
\node(2) at (2,0) {$\cdot$};
\node(3) at (4,0) {$\cdot$};
\node(4) at (6,0) {$\cdot$};
\node(5) at (8,0) {$\cdot$};
\node(6) at (4,-2) {$\cdot$};
\path[<-](2) edge (3);
\path[->](3) edge (6);

\path[->](7) edge (1);
\path[->](1) edge (2);
\path[->](3) edge (4);
\path[<-](4) edge (5);
\end{tikzpicture}&
$\left(\begin{smallmatrix}
1&2&3&4&5&6&7 \\
1& 2& 3&7&6&5&4\\
\end{smallmatrix}\right)$ &
\begin{tikzpicture}[rounded corners,scale=.1]
\node at (-13,0){$\cdot$};
\node at (-7,0){$\cdot$};
\node at (-3,0){$\cdot$};
\node at (23,0){$\cdot$};
\node at (15,0){$\cdot$};
\node at (13,0){$\cdot$};
\node at (11,0){$\cdot$};
\draw (25,0) to (25,1) to (21,1) to (21,0);
\draw (31,0) to (31,6) to (9,6) to (9,0);
\draw (27,0) to (27,2) to (19,2) to (19,0);
\draw (3,0) to (3,5) to (-17,5) to (-17,0);
\draw (33,0) to (33,8) to (7,8) to (7,0);
\draw (5,0) to (5,7) to (-19,7) to (-19,-9);
\draw (-11,0) to (-11,1) to (-15,1) to (-15,0);
\draw (29,0) to (29,4) to (17,4) to (17,0);
\draw (-1,0) to (-1,1) to (-5,1) to (-5,0);
\draw (1,0) to (1,3) to (-9,3) to (-9,0);
\draw (31,0) to (31,-8) to (-15,-8) to (-15,0);
\draw (-1,0) to (-1,-3) to (-11,-3) to (-11,0);
\draw (19,0) to (19,-3) to (9,-3) to (9,0);
\draw (21,0) to (21,-4) to (7,-4) to (7,0);
\draw (25,0) to (25,-5) to (5,-5) to (5,0);
\draw (33,0) to (33,-9) to (-17,-9) to (-17,0);
\draw (27,0) to (27,-6) to (3,-6) to (3,0);
\draw (17,0) to (17,-2) to (11,-2) to (11,0);
\draw (-5,0) to (-5,-1) to (-9,-1) to (-9,0);
\draw (29,0) to (29,-7) to (1,-7) to (1,0);
\end{tikzpicture}
\\

\hline
\begin{tikzpicture}[scale =.35]		
\node(7) at (-2,0) {$\cdot$};
\node(1) at (0,0) {$\cdot$};
\node(2) at (2,0) {$\cdot$};
\node(3) at (4,0) {$\cdot$};
\node(4) at (6,0) {$\cdot$};
\node(5) at (8,0) {$\cdot$};
\node(6) at (4,-2) {$\cdot$};
\path[<-](2) edge (3);
\path[->](3) edge (6);

\path[<-](7) edge (1);
\path[->](1) edge (2);
\path[->](3) edge (4);
\path[<-](4) edge (5);
\end{tikzpicture}&
$\left(\begin{smallmatrix}
1&2&3&4&5&6&7 \\
1& 3&7&6&5&4&2\\
\end{smallmatrix}\right)$ &
\begin{tikzpicture}[rounded corners,scale=.1]
\node at (-11,0){$\cdot$};
\node at (29,0){$\cdot$};
\node at (-9,0){$\cdot$};
\node at (17,0){$\cdot$};
\node at (9,0){$\cdot$};
\node at (7,0){$\cdot$};
\node at (5,0){$\cdot$};
\draw (-7,0) to (-7,2) to (-13,2) to (-13,0);
\draw (21,0) to (21,3) to (13,3) to (13,0);
\draw (25,0) to (25,8) to (3,8) to (3,0);
\draw (-3,0) to (-3,6) to (-17,6) to (-17,0);
\draw (23,0) to (23,5) to (11,5) to (11,0);
\draw (-5,0) to (-5,4) to (-15,4) to (-15,0);
\draw (31,0) to (31,1) to (27,1) to (27,0);
\draw (-1,0) to (-1,7) to (-19,7) to (-19,-9);
\draw (19,0) to (19,1) to (15,1) to (15,0);
\draw (33,0) to (33,9) to (1,9) to (1,0);
\draw (23,0) to (23,-6) to (-5,-6) to (-5,0);
\draw (31,0) to (31,-9) to (-15,-9) to (-15,0);
\draw (27,0) to (27,-8) to (-13,-8) to (-13,0);
\draw (19,0) to (19,-4) to (-1,-4) to (-1,0);
\draw (11,0) to (11,-1) to (5,-1) to (5,0);
\draw (21,0) to (21,-5) to (-3,-5) to (-3,0);
\draw (25,0) to (25,-7) to (-7,-7) to (-7,0);
\draw (13,0) to (13,-2) to (3,-2) to (3,0);
\draw (33,0) to (33,-10) to (-17,-10) to (-17,0);
\draw (15,0) to (15,-3) to (1,-3) to (1,0);
\end{tikzpicture}
\\

\hline
\begin{tikzpicture}[scale =.35]		
\node(7) at (-2,0) {$\cdot$};
\node(1) at (0,0) {$\cdot$};
\node(2) at (2,0) {$\cdot$};
\node(3) at (4,0) {$\cdot$};
\node(4) at (6,0) {$\cdot$};
\node(5) at (8,0) {$\cdot$};
\node(6) at (4,-2) {$\cdot$};
\path[<-](2) edge (3);
\path[->](3) edge (6);

\path[->](7) edge (1);
\path[<-](1) edge (2);
\path[->](3) edge (4);
\path[<-](4) edge (5);
\end{tikzpicture}&
$\left(\begin{smallmatrix}
1&2&3&4&5&6&7 \\
1& 2&7&6&5&4&3\\
\end{smallmatrix}\right)$ &
\begin{tikzpicture}[rounded corners,scale=.1]
\node at (5,0){$\cdot$};
\node at (7,0){$\cdot$};
\node at (29,0){$\cdot$};
\node at (21,0){$\cdot$};
\node at (13,0){$\cdot$};
\node at (9,0){$\cdot$};
\node at (11,0){$\cdot$};
\draw (33,0) to (33,6) to (-1,6) to (-1,0);
\draw (9,0) to (9,2) to (3,2) to (3,0);
\draw (31,0) to (31,5) to (1,5) to (1,0);
\draw (27,0) to (27,4) to (15,4) to (15,0);
\draw (25,0) to (25,3) to (17,3) to (17,0);
\draw (23,-9) to (23,1) to (19,1) to (19,0);
\draw (19,0) to (19,-5) to (-1,-5) to (-1,0);
\draw (15,0) to (15,-3) to (3,-3) to (3,0);
\draw (17,0) to (17,-4) to (1,-4) to (1,0);
\draw (33,0) to (33,-2) to (25,-2) to (25,0);
\draw (31,0) to (31,-1) to (27,-1) to (27,0);
\end{tikzpicture}
\\

\hline
\begin{tikzpicture}[scale =.35]		
\node(7) at (-2,0) {$\cdot$};
\node(1) at (0,0) {$\cdot$};
\node(2) at (2,0) {$\cdot$};
\node(3) at (4,0) {$\cdot$};
\node(4) at (6,0) {$\cdot$};
\node(5) at (8,0) {$\cdot$};
\node(6) at (4,-2) {$\cdot$};
\path[<-](2) edge (3);
\path[->](3) edge (6);

\path[->](7) edge (1);
\path[->](1) edge (2);
\path[<-](3) edge (4);
\path[<-](4) edge (5);
\end{tikzpicture}&
$\left(\begin{smallmatrix}
1&2&3&4&5&6&7 \\
1&2&3&5&7&6&4
\end{smallmatrix}\right)$ &
\begin{tikzpicture}[rounded corners,scale=.1]
\node at (3,0){$\cdot$};
\node at (5,0){$\cdot$};
\node at (13,0){$\cdot$};
\node at (27,0){$\cdot$};
\node at (15,0){$\cdot$};
\node at (23,0){$\cdot$};
\node at (17,0){$\cdot$};
\draw (33,0) to (33,5) to (1,5) to (1,0);
\draw (29,0) to (29,1) to (25,1) to (25,0);
\draw (19,0) to (19,2) to (11,2) to (11,0);
\draw (31,0) to (31,4) to (7,4) to (7,0);
\draw (21,0) to (21,3) to (9,3) to (9,-9);
\draw (25,0) to (25,-1) to (21,-1) to (21,0);
\draw (29,0) to (29,-3) to (19,-3) to (19,0);
\draw (31,0) to (31,-4) to (17,-4) to (17,0);
\draw (7,0) to (7,-2) to (1,-2) to (1,0);
\draw (33,0) to (33,-5) to (11,-5) to (11,0);
\end{tikzpicture}
\\

\hline
\begin{tikzpicture}[scale =.35]		
\node(7) at (-2,0) {$\cdot$};
\node(1) at (0,0) {$\cdot$};
\node(2) at (2,0) {$\cdot$};
\node(3) at (4,0) {$\cdot$};
\node(4) at (6,0) {$\cdot$};
\node(5) at (8,0) {$\cdot$};
\node(6) at (4,-2) {$\cdot$};
\path[<-](2) edge (3);
\path[->](3) edge (6);

\path[<-](7) edge (1);
\path[<-](1) edge (2);
\path[->](3) edge (4);
\path[<-](4) edge (5);
\end{tikzpicture}&
$\left(\begin{smallmatrix}
1&2&3&4&5&6&7 \\
1&7&6&5&4&3&2
\end{smallmatrix}\right)$ &
\begin{tikzpicture}[rounded corners,scale=.1]
\node at (-1,0){$\cdot$};
\node at (31,0){$\cdot$};
\node at (25,0){$\cdot$};
\node at (19,0){$\cdot$};
\node at (13,0){$\cdot$};
\node at (11,0){$\cdot$};
\node at (9,0){$\cdot$};
\draw (3,0) to (3,2) to (-5,2) to (-5,-9);
\draw (21,0) to (21,1) to (17,1) to (17,0);
\draw (27,0) to (27,3) to (9,3) to (9,0);
\draw (1,0) to (1,1) to (-3,1) to (-3,0);
\draw (29,0) to (29,4) to (7,4) to (7,0);
\draw (33,0) to (33,5) to (5,5) to (5,0);
\draw (23,0) to (23,2) to (15,2) to (15,0);
\draw (17,0) to (17,-3) to (5,-3) to (5,0);
\draw (15,0) to (15,-2) to (7,-2) to (7,0);
\draw (33,0) to (33,-6) to (-3,-6) to (-3,0);
\draw (27,0) to (27,-1) to (23,-1) to (23,0);
\draw (29,0) to (29,-5) to (1,-5) to (1,0);
\draw (21,0) to (21,-4) to (3,-4) to (3,0);
\end{tikzpicture}
\\

\hline
\begin{tikzpicture}[scale =.35]		
\node(7) at (-2,0) {$\cdot$};
\node(1) at (0,0) {$\cdot$};
\node(2) at (2,0) {$\cdot$};
\node(3) at (4,0) {$\cdot$};
\node(4) at (6,0) {$\cdot$};
\node(5) at (8,0) {$\cdot$};
\node(6) at (4,-2) {$\cdot$};
\path[<-](2) edge (3);
\path[->](3) edge (6);

\path[<-](7) edge (1);
\path[->](1) edge (2);
\path[<-](3) edge (4);
\path[<-](4) edge (5);
\end{tikzpicture}&
$\left(\begin{smallmatrix}
1&2&3&4&5&6&7 \\
1&3&5&7&6&4&2
\end{smallmatrix}\right)$ &
\begin{tikzpicture}[rounded corners,scale=.1]
\node at (-1,0){$\cdot$};
\node at (29,0){$\cdot$};
\node at (7,0){$\cdot$};
\node at (15,0){$\cdot$};
\node at (9,0){$\cdot$};
\node at (19,0){$\cdot$};
\node at (11,0){$\cdot$};
\draw (21,0) to (21,4) to (1,4) to (1,0);
\draw (23,0) to (23,5) to (-3,5) to (-3,0);
\draw (13,0) to (13,2) to (5,2) to (5,0);
\draw (17,0) to (17,3) to (3,3) to (3,0);
\draw (33,-9) to (33,2) to (25,2) to (25,0);
\draw (31,0) to (31,1) to (27,1) to (27,0);
\draw (21,0) to (21,-2) to (11,-2) to (11,0);
\draw (27,0) to (27,-5) to (1,-5) to (1,0);
\draw (23,0) to (23,-3) to (5,-3) to (5,0);
\draw (17,0) to (17,-1) to (13,-1) to (13,0);
\draw (25,0) to (25,-4) to (3,-4) to (3,0);
\draw (31,0) to (31,-6) to (-3,-6) to (-3,0);
\end{tikzpicture}
\\

\hline
\begin{tikzpicture}[scale =.35]		
\node(7) at (-2,0) {$\cdot$};
\node(1) at (0,0) {$\cdot$};
\node(2) at (2,0) {$\cdot$};
\node(3) at (4,0) {$\cdot$};
\node(4) at (6,0) {$\cdot$};
\node(5) at (8,0) {$\cdot$};
\node(6) at (4,-2) {$\cdot$};
\path[<-](2) edge (3);
\path[->](3) edge (6);

\path[->](7) edge (1);
\path[<-](1) edge (2);
\path[<-](3) edge (4);
\path[<-](4) edge (5);
\end{tikzpicture}&
$\left(\begin{smallmatrix}
1&2&3&4&5&6&7 \\
1&2&5&7&6&4&3
\end{smallmatrix}\right)$ &
\begin{tikzpicture}[rounded corners,scale=.1]
\node at (-13,0){$\cdot$};
\node at (-9,0){$\cdot$};
\node at (29,0){$\cdot$};
\node at (21,0){$\cdot$};
\node at (-5,0){$\cdot$};
\node at (13,0){$\cdot$};
\node at (11,0){$\cdot$};
\draw (-3,0) to (-3,1) to (-7,1) to (-7,0);
\draw (31,0) to (31,7) to (9,7) to (9,0);
\draw (-1,0) to (-1,3) to (-11,3) to (-11,0);
\draw (3,0) to (3,6) to (-17,6) to (-17,0);
\draw (33,0) to (33,9) to (7,9) to (7,0);
\draw (1,0) to (1,5) to (-15,5) to (-15,0);
\draw (5,0) to (5,8) to (-19,8) to (-19,-9);
\draw (27,0) to (27,4) to (15,4) to (15,0);
\draw (25,0) to (25,2) to (17,2) to (17,0);
\draw (23,0) to (23,1) to (19,1) to (19,0);
\draw (19,0) to (19,-3) to (7,-3) to (7,0);
\draw (31,0) to (31,-7) to (-1,-7) to (-1,0);
\draw (27,0) to (27,-6) to (1,-6) to (1,0);
\draw (-3,0) to (-3,-3) to (-15,-3) to (-15,0);
\draw (23,0) to (23,-4) to (5,-4) to (5,0);
\draw (-7,0) to (-7,-1) to (-11,-1) to (-11,0);
\draw (17,0) to (17,-2) to (9,-2) to (9,0);
\draw (33,0) to (33,-8) to (-17,-8) to (-17,0);
\draw (25,0) to (25,-5) to (3,-5) to (3,0);
\draw (15,0) to (15,-1) to (11,-1) to (11,0);
\end{tikzpicture}

\\

\hline
\begin{tikzpicture}[scale =.35]		
\node(7) at (-2,0) {$\cdot$};
\node(1) at (0,0) {$\cdot$};
\node(2) at (2,0) {$\cdot$};
\node(3) at (4,0) {$\cdot$};
\node(4) at (6,0) {$\cdot$};
\node(5) at (8,0) {$\cdot$};
\node(6) at (4,-2) {$\cdot$};
\path[<-](2) edge (3);
\path[->](3) edge (6);

\path[<-](7) edge (1);
\path[<-](1) edge (2);
\path[<-](3) edge (4);
\path[<-](4) edge (5);
\end{tikzpicture}&
$\left(\begin{smallmatrix}
1&2&3&4&5&6&7 \\
1&5&7&4&6&3&2
\end{smallmatrix}\right)$ &
\begin{tikzpicture}[rounded corners,scale=.1]
\node at (1,0){$\cdot$};
\node at (29,0){$\cdot$};
\node at (25,0){$\cdot$};
\node at (13,0){$\cdot$};
\node at (9,0){$\cdot$};
\node at (21,0){$\cdot$};
\node at (11,0){$\cdot$};
\draw (33,0) to (33,5) to (-1,5) to (-1,0);
\draw (31,0) to (31,4) to (3,4) to (3,0);
\draw (19,0) to (19,3) to (5,3) to (5,-9);
\draw (15,0) to (15,1) to (11,1) to (11,0);
\draw (27,0) to (27,1) to (23,1) to (23,0);
\draw (17,0) to (17,2) to (7,2) to (7,0);
\draw (23,0) to (23,-1) to (19,-1) to (19,0);
\draw (3,0) to (3,-1) to (-1,-1) to (-1,0);
\draw (33,0) to (33,-4) to (7,-4) to (7,0);
\draw (31,0) to (31,-3) to (15,-3) to (15,0);
\draw (27,0) to (27,-2) to (17,-2) to (17,0);
\end{tikzpicture}
\end{tabular}

\caption{Given a quiver and a permutation in $P_Q$, the given non-self-crossing admissible curve~$\g$ is non-decreasing and $\a_\pi(\g) ={1\,\,2\,\,3\,\,3\,\,2\,\,1}\atop {\hspace{1.3cm}1}$.}\label{E_7_table_1233211_2}
\end{table}

\begin{Lem}\label{lem_e_7_root_1223211}
Let $Q \in \cE_7$ and $\a =\begin{smallmatrix}
1&2 &2 &3&2 & 1 \\
&&&1&&
\end{smallmatrix}.$ Then $\a \in \bigcup_{\pi\in P_Q} \{\a_\pi(\g)\,|\, \g\in \G_{\pi,nd}\}$.
\end{Lem}
\begin{proof}Similarly as the proof of Lemma \ref{lem_e_7_root_1233211}, we look at different cases of quivers. Note that $s_7\a = \a-\a_7$, $s_3\a=\a-\a_3$, $s_6\a=\a+\a_6$, $s_4\a = \a+\a_4$, and $s_i\a = \a$ for other $i$. As $s_6s_4s_7\a = s_4s_6s_7\a = \a - \a_7$, if $\pi^{-1}(7) < \pi^{-1}(4)$ and $\pi^{-1}(7) < \pi^{-1}(6)$, then $c_\pi\a \leq_D\begin{smallmatrix} 1&2&2&2&2&1\\ &&&1&&\end{smallmatrix}$. Similarly, if $\pi^{-1}(7) > \pi^{-1}(4)$ and $\pi^{-1}(7) > \pi^{-1}(6)$, then $c_\pi^{-1}\a \leq_D\begin{smallmatrix} 1&2&2&2&2&1\\ &&&1&&\end{smallmatrix}$.

Note that $s_4s_3\a = \a-\a_3$ and $s_6s_7\a=\a-\a_7$. Thus if $\pi^{-1}(3) < \pi^{-1}(4)$ and $\pi^{-1}(7) < \pi^{-1}(6)$, then $c_\pi\a \leq_D\begin{smallmatrix} 1&1&2&2&2&1\\ &&&1&&\end{smallmatrix}$. Similarly, if $\pi^{-1}(3) > \pi^{-1}(4)$ and $\pi^{-1}(7) > \pi^{-1}(6)$, $c_\pi^{-1}\a \leq_D\begin{smallmatrix} 1&1&2&2&2&1\\ &&&1&&\end{smallmatrix}$.

Now consider quivers of the form:
\begin{center}
\begin{tikzpicture}[scale =.7]		
\node[shape=circle,draw=black, minimum size = .4cm, inner sep = 4pt, outer sep = 0.5pt](7) at (-2,0) {$2$};
\node[shape=circle,draw=black, minimum size = .4cm, inner sep = 4pt, outer sep = 0.5pt](1) at (0,0) {$3$};
\node[shape=circle,draw=black, minimum size = .4cm, inner sep = 4pt, outer sep = 0.5pt](2) at (2,0) {$4$};
\node[shape=circle,draw=black, minimum size = .4cm, inner sep = 4pt, outer sep = 0.5pt](3) at (4,0) {$7$};
\node[shape=circle,draw=black, minimum size = .4cm, inner sep = 4pt, outer sep = 0.5pt](4) at (6,0) {$5$};
\node[shape=circle,draw=black, minimum size = .4cm, inner sep = 4pt, outer sep = 0.5pt](5) at (8,0) {$1$};
\node[shape=circle,draw=black, minimum size = .4cm, inner sep = 4pt, outer sep = 0.5pt](6) at (4,-2) {$6$};

\path[-](7) edge (1);
\path[<-](1) edge (2);
\path[->](2) edge (3);
\path[-](3) edge (4);
\path[-](4) edge (5);
\path[->](3) edge (6);
\end{tikzpicture}
\end{center}

Quivers of this form do not fit the descriptions above. Also $c_\pi\a \not<_D\a$ and $c_\pi^{-1}\a \not<_D\a$. We use Sage to find non-decreasing non-self-crossing admissible curves whose associated root is $\begin{smallmatrix}1&2&2&3&2&1\\ &&&1&& \end{smallmatrix} $. Refer to Table~\ref{E_7_table_12232111} for these curves.
\end{proof}

\begin{table}[t]\centering
\begin{tabular}{c|c|c}
\hline
Quiver, $Q$ & $\pi \in P_Q$ & $\gamma \in \Gamma_{\pi, nd}$\\
\hline
\begin{tikzpicture}[scale =.35]		
\node(7) at (-2,0) {$\cdot$};
\node(1) at (0,0) {$\cdot$};
\node(2) at (2,0) {$\cdot$};
\node(3) at (4,0) {$\cdot$};
\node(4) at (6,0) {$\cdot$};
\node(5) at (8,0) {$\cdot$};
\node(6) at (4,-2) {$\cdot$};

\path[<-](1) edge (2);
\path[->](2) edge (3);
\path[->](3) edge (6);

\path[->](7) edge (1);
\path[->](3) edge (4);
\path[->](4) edge (5);
\end{tikzpicture}&
$\left(\begin{smallmatrix}
1&2&3&4&5&6&7 \\
2& 4& 7& 6& 5& 3 &1\\
\end{smallmatrix}\right)$
&\begin{tikzpicture}[rounded corners,scale=.1]
\node at (31,0){$\cdot$};
\node at (5,0){$\cdot$};
\node at (29,0){$\cdot$};
\node at (11,0){$\cdot$};
\node at (19,0){$\cdot$};
\node at (17,0){$\cdot$};
\node at (13,0){$\cdot$};
\draw (9,0) to (9,1) to (13,1) to (13,0);
\draw (27,0) to (27,2) to (33,2) to (33,0);
\draw (7,0) to (7,3) to (15,3) to (15,0);
\draw (1,0) to (1,6) to (23,6) to (23,0);
\draw (3,0) to (3,5) to (21,5) to (21,0);
\draw (25,0) to (25,4) to (35,4) to (35,-9);
\draw (1,0) to (1,-5) to (33,-5) to (33,0);
\draw (7,0) to (7,-3) to (25,-3) to (25,0);
\draw (15,0) to (15,-1) to (21,-1) to (21,0);
\draw (3,0) to (3,-4) to (27,-4) to (27,0);
\draw (9,0) to (9,-2) to (23,-2) to (23,0);
\end{tikzpicture}
\\
\hline
\begin{tikzpicture}[scale =.35]		
\node(7) at (-2,0) {$\cdot$};
\node(1) at (0,0) {$\cdot$};
\node(2) at (2,0) {$\cdot$};
\node(3) at (4,0) {$\cdot$};
\node(4) at (6,0) {$\cdot$};
\node(5) at (8,0) {$\cdot$};
\node(6) at (4,-2) {$\cdot$};

\path[<-](1) edge (2);
\path[->](2) edge (3);
\path[->](3) edge (6);

\path[<-](7) edge (1);
\path[->](3) edge (4);
\path[->](4) edge (5);
\end{tikzpicture}&
$\left(\begin{smallmatrix}
1&2&3&4&5&6&7 \\
4& 7& 6& 5& 3 &2&1\\
\end{smallmatrix}\right)$
&\begin{tikzpicture}[rounded corners,scale=.1]
\node at (35,0){$\cdot$};
\node at (33,0){$\cdot$};
\node at (29,0){$\cdot$};
\node at (9,0){$\cdot$};
\node at (17,0){$\cdot$};
\node at (15,0){$\cdot$};
\node at (11,0){$\cdot$};
\draw (1,0) to (1,5) to (21,5) to (21,0);
\draw (25,0) to (25,3) to (37,3) to (37,0);
\draw (5,0) to (5,2) to (13,2) to (13,0);
\draw (3,0) to (3,4) to (19,4) to (19,0);
\draw (27,0) to (27,1) to (31,1) to (31,0);
\draw (7,0) to (7,1) to (11,1) to (11,0);
\draw (23,0) to (23,4) to (39,4) to (39,-9);
\draw (3,0) to (3,-4) to (25,-4) to (25,0);
\draw (31,0) to (31,-1) to (37,-1) to (37,0);
\draw (1,0) to (1,-5) to (27,-5) to (27,0);
\draw (7,0) to (7,-2) to (21,-2) to (21,0);
\draw (13,0) to (13,-1) to (19,-1) to (19,0);
\draw (5,0) to (5,-3) to (23,-3) to (23,0);
\end{tikzpicture}
\\
\hline
\begin{tikzpicture}[scale =.35]		
\node(7) at (-2,0) {$\cdot$};
\node(1) at (0,0) {$\cdot$};
\node(2) at (2,0) {$\cdot$};
\node(3) at (4,0) {$\cdot$};
\node(4) at (6,0) {$\cdot$};
\node(5) at (8,0) {$\cdot$};
\node(6) at (4,-2) {$\cdot$};

\path[<-](1) edge (2);
\path[->](2) edge (3);
\path[->](3) edge (6);

\path[->](7) edge (1);
\path[<-](3) edge (4);
\path[->](4) edge (5);
\end{tikzpicture}&
$\left(\begin{smallmatrix}
1&2&3&4&5&6&7 \\
2&4& 5&7& 6& 3 &1\\
\end{smallmatrix}\right)$
&\begin{tikzpicture}[rounded corners,scale=.1]
\node at (31,0){$\cdot$};
\node at (7,0){$\cdot$};
\node at (29,0){$\cdot$};
\node at (9,0){$\cdot$};
\node at (11,0){$\cdot$};
\node at (19,0){$\cdot$};
\node at (13,0){$\cdot$};
\draw (1,0) to (1,4) to (21,4) to (21,0);
\draw (23,0) to (23,3) to (37,3) to (37,-9);
\draw (5,0) to (5,2) to (15,2) to (15,0);
\draw (27,0) to (27,1) to (33,1) to (33,0);
\draw (25,0) to (25,2) to (35,2) to (35,0);
\draw (3,0) to (3,3) to (17,3) to (17,0);
\draw (1,0) to (1,-6) to (35,-6) to (35,0);
\draw (15,0) to (15,-2) to (23,-2) to (23,0);
\draw (13,0) to (13,-3) to (25,-3) to (25,0);
\draw (3,0) to (3,-5) to (33,-5) to (33,0);
\draw (17,0) to (17,-1) to (21,-1) to (21,0);
\draw (5,0) to (5,-4) to (27,-4) to (27,0);
\end{tikzpicture}\\

\hline
\begin{tikzpicture}[scale =.35]		
\node(7) at (-2,0) {$\cdot$};
\node(1) at (0,0) {$\cdot$};
\node(2) at (2,0) {$\cdot$};
\node(3) at (4,0) {$\cdot$};
\node(4) at (6,0) {$\cdot$};
\node(5) at (8,0) {$\cdot$};
\node(6) at (4,-2) {$\cdot$};

\path[<-](1) edge (2);
\path[->](2) edge (3);
\path[->](3) edge (6);

\path[<-](7) edge (1);
\path[<-](3) edge (4);
\path[->](4) edge (5);
\end{tikzpicture}&
$\left(\begin{smallmatrix}
1&2&3&4&5&6&7 \\
4&5& 7& 6& 3 &2&1\\
\end{smallmatrix}\right)$
&
\begin{tikzpicture}[rounded corners,scale=.1]
\node at (35,0){$\cdot$};
\node at (33,0){$\cdot$};
\node at (29,0){$\cdot$};
\node at (7,0){$\cdot$};
\node at (11,0){$\cdot$};
\node at (27,0){$\cdot$};
\node at (19,0){$\cdot$};
\draw (5,0) to (5,1) to (9,1) to (9,0);
\draw (15,0) to (15,2) to (23,2) to (23,0);
\draw (1,0) to (1,5) to (37,5) to (37,0);
\draw (13,0) to (13,3) to (25,3) to (25,0);
\draw (3,0) to (3,4) to (31,4) to (31,0);
\draw (17,0) to (17,1) to (21,1) to (21,0);
\draw (9,0) to (9,-1) to (13,-1) to (13,0);
\draw (23,0) to (23,-4) to (37,-4) to (37,0);
\draw (1,0) to (1,-5) to (19,-5) to (19,0);
\draw (25,0) to (25,-2) to (31,-2) to (31,0);
\draw (5,0) to (5,-3) to (15,-3) to (15,0);
\draw (3,0) to (3,-4) to (17,-4) to (17,0);
\end{tikzpicture}
\\
\hline
\begin{tikzpicture}[scale =.35]		
\node(7) at (-2,0) {$\cdot$};
\node(1) at (0,0) {$\cdot$};
\node(2) at (2,0) {$\cdot$};
\node(3) at (4,0) {$\cdot$};
\node(4) at (6,0) {$\cdot$};
\node(5) at (8,0) {$\cdot$};
\node(6) at (4,-2) {$\cdot$};

\path[<-](1) edge (2);
\path[->](2) edge (3);
\path[->](3) edge (6);

\path[->](7) edge (1);
\path[->](3) edge (4);
\path[<-](4) edge (5);
\end{tikzpicture}&
$\left(\begin{smallmatrix}
1&2&3&4&5&6&7 \\
1&2&4& 7& 6&5& 3\\
\end{smallmatrix}\right)$
&
\begin{tikzpicture}[rounded corners,scale=.1]
\node at (3,0){$\cdot$};
\node at (7,0){$\cdot$};
\node at (31,0){$\cdot$};
\node at (13,0){$\cdot$};
\node at (21,0){$\cdot$};
\node at (17,0){$\cdot$};
\node at (15,0){$\cdot$};
\draw (15,0) to (15,1) to (11,1) to (11,0);
\draw (23,0) to (23,4) to (5,4) to (5,0);
\draw (25,0) to (25,5) to (1,5) to (1,0);
\draw (35,-9) to (35,2) to (27,2) to (27,0);
\draw (33,0) to (33,1) to (29,1) to (29,0);
\draw (19,0) to (19,3) to (9,3) to (9,0);
\draw (23,0) to (23,-1) to (19,-1) to (19,0);
\draw (29,0) to (29,-4) to (5,-4) to (5,0);
\draw (27,0) to (27,-3) to (9,-3) to (9,0);
\draw (33,0) to (33,-5) to (1,-5) to (1,0);
\draw (25,0) to (25,-2) to (11,-2) to (11,0);
\end{tikzpicture}
\\

\hline
\begin{tikzpicture}[scale =.35]		
\node(7) at (-2,0) {$\cdot$};
\node(1) at (0,0) {$\cdot$};
\node(2) at (2,0) {$\cdot$};
\node(3) at (4,0) {$\cdot$};
\node(4) at (6,0) {$\cdot$};
\node(5) at (8,0) {$\cdot$};
\node(6) at (4,-2) {$\cdot$};

\path[<-](1) edge (2);
\path[->](2) edge (3);
\path[->](3) edge (6);

\path[<-](7) edge (1);
\path[->](3) edge (4);
\path[<-](4) edge (5);
\end{tikzpicture}&
$\left(\begin{smallmatrix}
1&2&3&4&5&6&7 \\
1&4& 7& 6&5& 3&2\\
\end{smallmatrix}\right)$
&\begin{tikzpicture}[rounded corners,scale=.1]
\node at (3,0){$\cdot$};
\node at (35,0){$\cdot$};
\node at (31,0){$\cdot$};
\node at (11,0){$\cdot$};
\node at (19,0){$\cdot$};
\node at (15,0){$\cdot$};
\node at (13,0){$\cdot$};
\draw (37,0) to (37,2) to (27,2) to (27,0);
\draw (23,0) to (23,5) to (1,5) to (1,0);
\draw (39,-9) to (39,3) to (25,3) to (25,0);
\draw (21,0) to (21,4) to (5,4) to (5,0);
\draw (13,0) to (13,1) to (9,1) to (9,0);
\draw (33,0) to (33,1) to (29,1) to (29,0);
\draw (17,0) to (17,2) to (7,2) to (7,0);
\draw (29,0) to (29,-5) to (1,-5) to (1,0);
\draw (27,0) to (27,-4) to (5,-4) to (5,0);
\draw (25,0) to (25,-3) to (7,-3) to (7,0);
\draw (37,0) to (37,-1) to (33,-1) to (33,0);
\draw (21,0) to (21,-1) to (17,-1) to (17,0);
\draw (23,0) to (23,-2) to (9,-2) to (9,0);
\end{tikzpicture}\\

\hline
\begin{tikzpicture}[scale =.35]		
\node(7) at (-2,0) {$\cdot$};
\node(1) at (0,0) {$\cdot$};
\node(2) at (2,0) {$\cdot$};
\node(3) at (4,0) {$\cdot$};
\node(4) at (6,0) {$\cdot$};
\node(5) at (8,0) {$\cdot$};
\node(6) at (4,-2) {$\cdot$};

\path[<-](1) edge (2);
\path[->](2) edge (3);
\path[->](3) edge (6);

\path[->](7) edge (1);
\path[<-](3) edge (4);
\path[<-](4) edge (5);
\end{tikzpicture}&
$\left(\begin{smallmatrix}
1&2&3&4&5&6&7 \\
1&2&4&5& 7& 6& 3\\
\end{smallmatrix}\right)$
&\begin{tikzpicture}[rounded corners,scale=.1]
\node at (5,0){$\cdot$};
\node at (9,0){$\cdot$};
\node at (29,0){$\cdot$};
\node at (11,0){$\cdot$};
\node at (17,0){$\cdot$};
\node at (25,0){$\cdot$};
\node at (19,0){$\cdot$};
\draw (33,0) to (33,5) to (1,5) to (1,0);
\draw (27,0) to (27,3) to (7,3) to (7,0);
\draw (23,0) to (23,2) to (13,2) to (13,0);
\draw (31,0) to (31,4) to (3,4) to (3,0);
\draw (21,-9) to (21,1) to (15,1) to (15,0);
\draw (19,0) to (19,-5) to (1,-5) to (1,0);
\draw (31,0) to (31,-1) to (27,-1) to (27,0);
\draw (33,0) to (33,-3) to (23,-3) to (23,0);
\draw (15,0) to (15,-4) to (3,-4) to (3,0);
\draw (13,0) to (13,-2) to (7,-2) to (7,0);
\end{tikzpicture}

\\
\hline
\begin{tikzpicture}[scale =.35]		
\node(7) at (-2,0) {$\cdot$};
\node(1) at (0,0) {$\cdot$};
\node(2) at (2,0) {$\cdot$};
\node(3) at (4,0) {$\cdot$};
\node(4) at (6,0) {$\cdot$};
\node(5) at (8,0) {$\cdot$};
\node(6) at (4,-2) {$\cdot$};

\path[<-](1) edge (2);
\path[->](2) edge (3);
\path[->](3) edge (6);

\path[<-](7) edge (1);
\path[<-](3) edge (4);
\path[<-](4) edge (5);
\end{tikzpicture}&
$\left(\begin{smallmatrix}
1&2&3&4&5&6&7 \\
1&4&5& 7& 6& 3&2\\
\end{smallmatrix}\right)$
&
\begin{tikzpicture}[rounded corners,scale=.1]
\node at (4,0){$\cdot$};
\node at (48,0){$\cdot$};
\node at (42,0){$\cdot$};
\node at (18,0){$\cdot$};
\node at (20,0){$\cdot$};
\node at (28,0){$\cdot$};
\node at (26,0){$\cdot$};
\draw (30,0) to (30,1) to (26,1) to (26,0);
\draw (34,0) to (34,7) to (10,7) to (10,0);
\draw (32,0) to (32,6) to (12,6) to (12,0);
\draw (22,0) to (22,2) to (16,2) to (16,0);
\draw (46,0) to (46,3) to (38,3) to (38,0);
\draw (50,0) to (50,5) to (36,5) to (36,0);
\draw (6,0) to (6,1) to (2,1) to (2,0);
\draw (24,0) to (24,4) to (14,4) to (14,0);
\draw (44,0) to (44,1) to (40,1) to (40,0);
\draw (8,0) to (8,3) to (0,3) to (0,-9);
\draw (32,0) to (32,-2) to (22,-2) to (22,0);
\draw (40,0) to (40,-6) to (10,-6) to (10,0);
\draw (46,0) to (46,-8) to (6,-8) to (6,0);
\draw (50,0) to (50,-9) to (2,-9) to (2,0);
\draw (34,0) to (34,-3) to (16,-3) to (16,0);
\draw (44,0) to (44,-7) to (8,-7) to (8,0);
\draw (36,0) to (36,-4) to (14,-4) to (14,0);
\draw (30,0) to (30,-1) to (24,-1) to (24,0);
\draw (38,0) to (38,-5) to (12,-5) to (12,0);
\end{tikzpicture}
\end{tabular}
\caption{Given a quiver and a permutation in $P_Q$, the given non-self-crossing admissible curve~$\g$ is non-decreasing and $\a_\pi(\g)= {1\,\,2\,\,2\,\,3\,\,2\,\,1}\atop {\hspace{1.3cm}1} $.}\label{E_7_table_12232111}
\end{table}

\begin{Prop}\label{prop_e_7}
If $Q \in \cE_7$, then $\D_Q^+ =\bigcup_{\pi\in P_Q}\{\a_\pi(\g)\,|\,\g\in \G_{\pi,nd}\}$.
\end{Prop}
\begin{proof}
Let $Q \in \cE_7$ and $\pi$ be a permutation in $P_Q$ such that $\pi^{-1}(1) =1$ or 7, $\pi^{-1}(2)$ is the maximum or minimum value of $\{1,\ldots, 7\} \setminus \big\{\pi^{-1}(1)\big\}$, and $\pi^{-1}(3)$ is the maximum or minimum value of $\{1,\ldots, 7\} \setminus \big\{\pi^{-1}(1), \pi^{-1}(2)\big\}$. Such permutation exists as $\pi \in P_Q \cap U_n$ satisfies such condition. Note that $(1\,\,2)\pi \in P_Q$ if and only if $\big|\pi^{-1}(1)-\pi^{-1}(2)\big| = 1$ and $(1\,\,3)\pi \in P_Q$ if and only if $\big|\pi^{-1}(1)-\pi^{-1}(3)\big| = 1$. Let $R =\{(1\,\, 2) \pi$, $(1\,\,3) \pi$, $(1\,\,3)(1\,\,2)\pi,\pi\} \cap P_Q$. For all $p \in R$, $c_p\a$~is constant as $s_1$ commutes with $s_2$ and $s_3$.

By Lemma \ref{lem_subquiver}, Remark \ref{rmk_subquiver}, and Propositions~\ref{prop_a_non_dec},~\ref{prop_non_dec_curve_d} and~\ref{prop_e_6}, the positive roots of type~$A$,~$D$, and~$E_6$ are in $\bigcup_{p\in P_Q} \{\a_p(\g)\,|\,\g\in \G_{p,nd}\}$. Also by Lemma~\ref{lem_1_at_leaves}, any root of the form $\a_1 + \a$ where~$\a$ is a positive root of type $D$ or of the form $\a_2 + \a$ where $\a$ is a positive root of type~$E_6$ is in $\bigcup_{p\in P_Q} \{\a_p(\g)\,|\,\g\in \G_{p,nd}\}$. In particular, a root of the form $\a_1 + \a$ where $\a$ is type $D$ is in $\{\a_p(\g)\,|\,\g\in \G_{p,nd}\}$ for $p\in P_Q$ so that $p^{-1}(1) = 1$ or 7. Also a root of the form $\a_2 + \a$ where~$\a$ is type~$E_6$ is in $\{\a_p(\g)\,|\,\g\in \G_{p,nd}\}$ for $p \in P_Q$ so that $p^{-1}(2) = 1$ or 7 and $p|_{p^{-1}(\{1,3,\ldots, 7\})}$ fits the condition of the permutation given in the proof of Proposition~\ref{prop_e_6}. At least one of the permutations in~$R$ satisfies such conditions; thus any root of type $A$, $D$, and $E_6$ and the roots of the form $\a_1 + \a$ where $\a$ is type $D$ and $\a_2+\a$ where~$\a$ is type $E_6$ are in $\bigcup_{p \in R} \{\a_p(\g)\,|\,\g\in \G_{p,nd}\}$.

By Lemma \ref{lem_e_7_root_1233211} and Lemma \ref{lem_e_7_root_1223211}, $\begin{smallmatrix}1&2&2&3&2&1\\ &&&1&& \end{smallmatrix}, \begin{smallmatrix}1&2&3&3&2&1\\ &&&1&& \end{smallmatrix} \in \bigcup_{p\in P_Q}\{\a_p(\g)\,|\,\g\in \G_{p,nd}\}$. In particular, they are in $\bigcup_{p\in R}\{\a_p(\g)\,|\,\g\in \G_{p,nd}\}$. To show that the roots in Lemma~\ref{lem_E_7_6_roots} are in $\bigcup_{p\in P_Q}\{\a_p(\g)\,|\,\g\in \G_{p,nd}\}$, we induct on the height of the roots.

The root with smallest height of the roots in Lemma~\ref{lem_E_7_6_roots} is $\begin{smallmatrix}1&2&2&2&2&1\\ &&&1&& \end{smallmatrix}$. Either $c_\pi\a$ or $c_\pi^{-1}\a$ is smaller than $\begin{smallmatrix}1&2&2&2&1&1\\ &&&1&& \end{smallmatrix}$ which is in $\{\a_\pi(\g)\,|\,\g\in \G_{\pi,nd}\}$ as $\pi^{-1}(1) = 1$ or $7$. Thus $\a \in \{\a_{\pi}(\g)\,|\,\g\in \G_{(1\,\,2)\pi,nd}\}$.

Let $\a$ be a root in Lemma~\ref{lem_E_7_6_roots}. Assume that for any root with height less than $h(\a)$ is in~$\bigcup_{p\in R}\{\a_p(\g)\,|\,\g\in \G_{p,nd}\}$. By Lemma~\ref{lem_E_7_6_roots}, $c_p\a$ or $c_p^{-1}\a$ is smaller than $\a$. Assume without loss of generality, $c_p\a <_D \a$. As $c_p\a$ is same for all $p\in R$, there is $p \in R$ such that $c_p\a \in \{\a_p(\g)\,|\allowbreak \g\in \G_{p,nd}\}$ and by Lemma~\ref{lem_cox_transf}, $\a \in \{\a_p(\g)\,|\,\g\in \G_{p,nd}\}$. Therefore $\D_Q^+ =\bigcup_{\pi\in P_Q}\{\a_\pi(\g)\,|\allowbreak\g\in \G_{\pi,nd}\}$.
\end{proof}

\begin{proof}[Proof of Theorem \ref{thm_main}]
This theorem follows from Propositions~\ref{prop_a_strict_inc}, \ref{prop_a_non_dec}, \ref{prop_non_dec_curve_d},~\ref{prop_e_6}, and~\ref{prop_e_7}.
\end{proof}

\appendix

\section{Other cases}\label{appendix}

\subsection[Type E8]{Type $\boldsymbol{E_8}$} \label{E_8}
Let $Q$ be a quiver of type $E_8$ and $\a$ be a positive root of $\D_Q$. By Lemmas~\ref{lem_subquiver} and~\ref{lem_cox_transf}, Remark~\ref{rmk_subquiver}, and Theorem~\ref{thm_main}, if $c_\pi\a$ is of type $A$, $D$, $E_6$, or $E_7$ and $c_\pi\a <_D \a$, then $\a \in \bigcup_{\pi \in P_Q} \{\a_\pi(\g)\,|\, \g \in \G_{\pi, nd}\}$. We could proceed by induction on the height of the roots to prove that the real Schur roots are associated roots of non-decreasing non-self-crossing admissible curves. However, there are positive roots $\a$ and quivers $Q$ of type $E_8$ such that $c_\pi \a$ and $c_\pi^{-1}\a$ are not comparable with~$\a$ for any $\pi \in P_Q$. For example, let $\a$ be the root $\begin{smallmatrix}
1&2&2&3&3&2&1\\
&&&&1&&
\end{smallmatrix}$ and $Q$ be the quiver below:
\begin{center}
\begin{tikzpicture}[scale =.8]		
\node[shape=circle,draw=black, minimum size = .4cm, inner sep = 4pt, outer sep = 0.5pt](2) at (-4,0) {$2$};
\node[shape=circle,draw=black, minimum size = .4cm, inner sep = 4pt, outer sep = 0.5pt](3) at (-2,0) {$3$};
\node[shape=circle,draw=black, minimum size = .4cm, inner sep = 4pt, outer sep = 0.5pt](4) at (0,0) {$4$};
\node[shape=circle,draw=black, minimum size = .4cm, inner sep = 4pt, outer sep = 0.5pt](5) at (2,0) {$5$};
\node[shape=circle,draw=black, minimum size = .4cm, inner sep = 4pt, outer sep = 0.5pt](8) at (4,0) {$8$};
\node[shape=circle,draw=black, minimum size = .4cm, inner sep = 4pt, outer sep = 0.5pt](6) at (6,0) {$6$};
\node[shape=circle,draw=black, minimum size = .4cm, inner sep = 4pt, outer sep = 0.5pt](1) at (8,0) {$1$};
\node[shape=circle,draw=black, minimum size = .4cm, inner sep = 4pt, outer sep = 0.5pt](7) at (4,-2) {$7$};
\path[->](1) edge (6);
\path[->](2) edge (3);
\path[->](3) edge (4);
\path[<-](4) edge (5);
\path[<-](5) edge (8);
\path[<-](7) edge (8);
\path[->] (8) edge (6);
\end{tikzpicture}
\end{center}

Note that $\pi= \left(\begin{smallmatrix} 1&2&3&4&5&6&7&8\\ 1&2&3&8&7&6&5&4\end{smallmatrix}\right)$ is a permutation in $P_Q$; then $c_\pi \a = \begin{smallmatrix}
1&2&3&3&4&2&1\\
&&&&2&&
\end{smallmatrix}$ and $c_\pi^{-1} \a = \begin{smallmatrix}
1&1&1&2&3&2&1\\
&&&&2&&
\end{smallmatrix}$. Neither are comparable with $\a$. Similarly, $c_p\a$ and $c_p^{-1}\a$ are not comparable with~$\a$ for all other~$p \in P_Q$. As Lemma~\ref{lem_1_at_leaves} does not apply to $\a$, we would need to construct a non-decreasing admissible curve such that the associated root is $\a$. Below is a list of such roots, i.e., roots such that Lemma~\ref{lem_1_at_leaves} does not apply and are not comparable to the resulting root after applying $c_\pi$ or the resulting root after applying $c_\pi^{-1}$.

\begin{tikzpicture}
\node at (-1,0){1};
\node at (-.5,0){2};
\node(1) at (0,0) {$2$};
\node(2) at (.5,0) {$2$};
\node(3) at (1,0) {$3$};
\node(4) at (1.5,0) {$2$};
\node(5) at (2,0) {$1$};
\node(6) at (1,-.5) {$1$};

\node at (3,0){1};
\node at (3.5,0){2};
\node(2) at (4,0) {$2$};
\node(3) at (4.5,0) {$2$};
\node(4) at (5,0) {$3$};
\node(5) at (5.5,0) {$2$};
\node at (6,0) {$1$};
\node(6) at (5,-.5) {$2$};

\node(1) at (7,0) {$1$};
\node(2) at (7.5,0) {$2$};
\node(3) at (8,0) {$2$};
\node(4) at (8.5,0) {3};
\node(5) at (9,0) {3};
\node(4) at (9.5,0) {$2$};
\node(5) at (10,0) {$1$};
\node(6) at (9,-.5) {$1$};

\node(2) at (11,0) {$1$};
\node(3) at (11.5,0) {$2$};
\node(4) at (12,0) {$2$};
\node(5) at (12.5,0) {$3$};
\node(2) at (13,0) {$3$};
\node(3) at (13.5,0) {$2$};
\node(4) at (14,0) {$1$};
\node(6) at (13,-.5) {$2$};

\end{tikzpicture}

\begin{tikzpicture}
\node at (-1,0){1};
\node at (-.5,0){2};
\node(1) at (0,0) {3};
\node(2) at (.5,0) {$3$};
\node(3) at (1,0) {$3$};
\node(4) at (1.5,0) {$2$};
\node(5) at (2,0) {$1$};
\node(6) at (1,-.5) {$1$};

\node at (3,0){1};
\node at (3.5,0){2};
\node(2) at (4,0) {$2$};
\node(3) at (4.5,0) {$3$};
\node(4) at (5,0) {$4$};
\node(5) at (5.5,0) {$2$};
\node at (6,0) {$1$};
\node(6) at (5,-.5) {$2$};

\node(1) at (7,0) {$1$};
\node(2) at (7.5,0) {$2$};
\node(3) at (8,0) {$2$};
\node(4) at (8.5,0) {3};
\node(5) at (9,0) {4};
\node(4) at (9.5,0) {$3$};
\node(5) at (10,0) {$1$};
\node(6) at (9,-.5) {$2$};

\node(2) at (11,0) {$1$};
\node(3) at (11.5,0) {$2$};
\node(4) at (12,0) {$3$};
\node(5) at (12.5,0) {$3$};
\node(2) at (13,0) {$4$};
\node(3) at (13.5,0) {$2$};
\node(4) at (14,0) {$1$};
\node(6) at (13,-.5) {$2$};

\end{tikzpicture}

\begin{tikzpicture}
\node at (-1,0){1};
\node at (-.5,0){2};
\node(1) at (0,0) {3};
\node(2) at (.5,0) {$3$};
\node(3) at (1,0) {$4$};
\node(4) at (1.5,0) {$3$};
\node(5) at (2,0) {$1$};
\node(6) at (1,-.5) {$2$};

\node at (3,0){1};
\node at (3.5,0){2};
\node(2) at (4,0) {$3$};
\node(3) at (4.5,0) {$4$};
\node(4) at (5,0) {$4$};
\node(5) at (5.5,0) {$2$};
\node at (6,0) {$1$};
\node(6) at (5,-.5) {$2$};

\node(1) at (7,0) {$1$};
\node(2) at (7.5,0) {$2$};
\node(3) at (8,0) {$3$};
\node(4) at (8.5,0) {3};
\node(5) at (9,0) {4};
\node(4) at (9.5,0) {$3$};
\node(5) at (10,0) {$2$};
\node(6) at (9,-.5) {$2$};

\node(2) at (11,0) {$1$};
\node(3) at (11.5,0) {$2$};
\node(4) at (12,0) {$3$};
\node(5) at (12.5,0) {$4$};
\node(2) at (13,0) {$4$};
\node(3) at (13.5,0) {$3$};
\node(4) at (14,0) {$1$};
\node(6) at (13,-.5) {$2$};

\end{tikzpicture}

\begin{tikzpicture}
\node at (-1,0){1};
\node at (-.5,0){2};
\node(1) at (0,0) {3};
\node(2) at (.5,0) {$4$};
\node(3) at (1,0) {5};
\node(4) at (1.5,0) {$3$};
\node(5) at (2,0) {$1$};
\node(6) at (1,-.5) {$2$};

\node at (3,0){1};
\node at (3.5,0){2};
\node(2) at (4,0) {$3$};
\node(3) at (4.5,0) {$4$};
\node(4) at (5,0) {5};
\node(5) at (5.5,0) {$3$};
\node at (6,0) {$1$};
\node(6) at (5,-.5) {$3$};

\node(1) at (7,0) {$1$};
\node(2) at (7.5,0) {$2$};
\node(3) at (8,0) {$3$};
\node(4) at (8.5,0) {4};
\node(5) at (9,0) {5};
\node(4) at (9.5,0) {$3$};
\node(5) at (10,0) {$2$};
\node(6) at (9,-.5) {$2$};

\node(2) at (11,0) {$1$};
\node(3) at (11.5,0) {$2$};
\node(4) at (12,0) {$3$};
\node(5) at (12.5,0) {$4$};
\node(2) at (13,0) {$5$};
\node(3) at (13.5,0) {$4$};
\node(4) at (14,0) {$2$};
\node(6) at (13,-.5) {$2$};

\end{tikzpicture}

The Sage code used in Section~\ref{section6.2} to identify a non-decreasing non-self-crossing admissible curve for a positive root of type $E_7$ has limitations to be used to a non-decreasing non-self-crossing admissible curve for a positive root of type $E_8$. Let $\ell$ be the line through every marked points. For example, see the diagram below.

\begin{center}
 \includegraphics[trim =0mm 245mm 70mm 5mm, clip, width = .5\linewidth ]{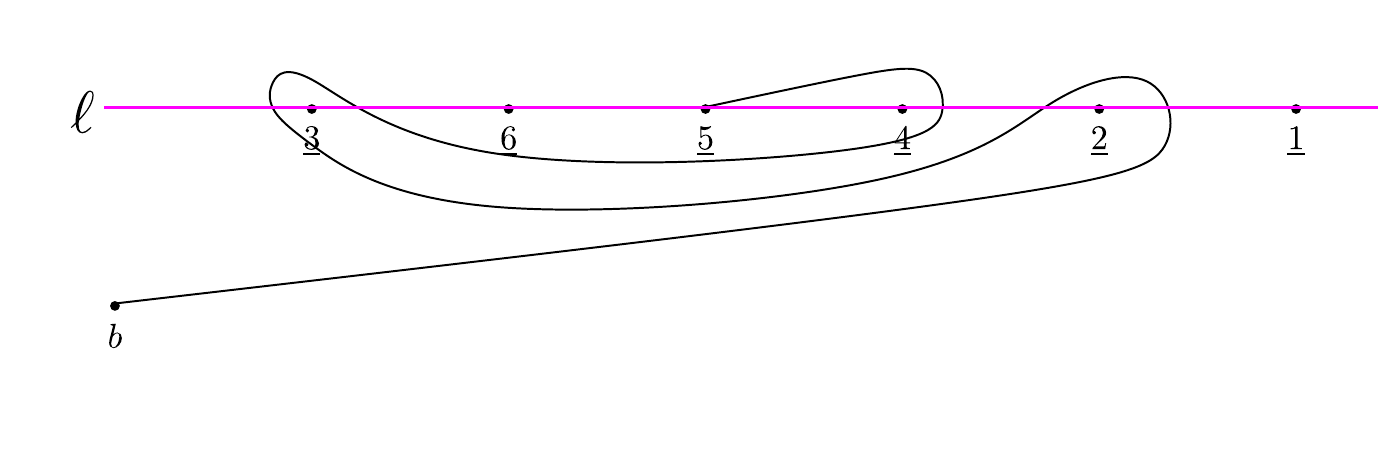}
\end{center}

The curve intersects $\ell$ at 6 points. The Sage code used in Section~\ref{section6.2} identifies all non-self-crossing admissible curves given an intersection number. For a quiver of type $E_8$, this intersection number must be very large compared to a quiver of type $E_7$. This code is too slow to run for quivers of type $E_8$.

\subsection[Affine quivers of type A with unique sink and source]{Affine quivers of type $\boldsymbol{A}$ with unique sink and source} \label{affine_a}
Recall that Conjecture~\ref{conjmain} assumes that $Q$ is any acyclic quiver. In this section we discuss an example of non-finite quiver that satisfies Conjecture~\ref{conjmain}. Let~$Q$ be an acyclic quiver of type $A$ with a unique source and a unique sink. Let $P=s\to p_1\to \cdots \to p_k\to t$ and $R=s\to r_1\to \cdots \to r_\ell \to t$ be the two paths in~$Q$, i.e.,~$Q$ is

\begin{center}
\begin{tikzpicture}
 \node[shape=circle,draw=black, minimum size = .6cm, inner sep = 4pt, outer sep = 0.5pt,](1) at (0,-.2) {$s$};
 \node[shape=circle,draw=black, minimum size = .6cm, inner sep = 2pt, outer sep = 0.5pt,](2) at (1,1) {$p_1$};
 \node[shape=circle,draw=black, minimum size = .6cm, inner sep = 2pt, outer sep = 0.5pt,](3) at (3,1) {$p_2$};
 \node[shape=circle,draw=black, minimum size = .6cm, inner sep = 2pt, outer sep = 0.5pt,](4) at (5,1) {$p_3$};
 \node(5) at (7,1) {$\cdots$};
 \node[shape=circle,draw=black, minimum size = .6cm, inner sep = 2pt, outer sep = 0.5pt,](11) at (9,1) {$p_k$};

 \node[shape=circle,draw=black, minimum size = .6cm, inner sep = 4pt, outer sep = 0.5pt,](6) at (10,-.2) {$t$};

 \node[shape=circle,draw=black, minimum size = .6cm, inner sep = 2pt, outer sep = 0.5pt,](7) at (2,-.7) {$r_1$};
 \node[shape=circle,draw=black, minimum size = .6cm, inner sep = 2pt, outer sep = 0.5pt,](8) at (4,-.7) {$r_2$};
 \node(9) at (6,-.7) {$\cdots$};
 \node[shape=circle,draw=black, minimum size = .6cm, inner sep = 2pt, outer sep = 0.5pt,](10) at (8,-.7) {$r_\ell$};
 \path[->] (1) edge (2);
 \path[->] (1) edge (7);
 \path[->] (2) edge (3);
 \path[->] (3) edge (4);
 \path[->] (4) edge (5);
 \path[->] (5) edge (11);
 \path[->] (11) edge (6);
 \path[->] (7) edge (8);
 \path[->] (8) edge (9);
 \path[->] (9) edge (10);
 \path[->] (10) edge (6);
\end{tikzpicture}
\end{center}

Let $T$ be a triangulation of an annulus $X$ with $k+1$ marked points on the one of the boundaries and $\ell+1$ marked points on the other boundary. Then consider the following triangulation on~$X$.

\begin{center}
\includegraphics[trim =20mm 234mm 142mm 22mm, clip, width = .35\linewidth ]{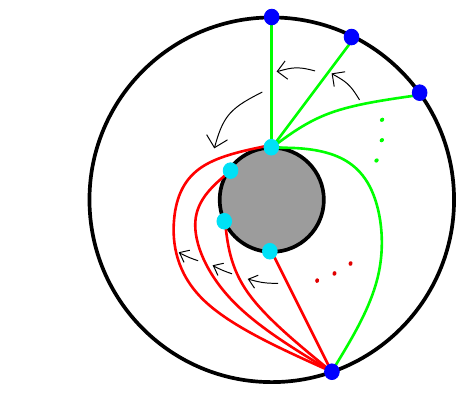}
\end{center}

\begin{Rmk}
Given a triangulation, let $Q_T$ be the quiver arising from $T$ as described in \cite{FominShapThurs}. Then, $Q_T = Q$.
\end{Rmk}

In \cite{FominShapThurs}, it was shown that each cluster variable corresponds to an arc on $A$ up to isotopy and the denominator vector of the cluster variable is given by the intersection of the curve and the arcs in $T$. As $Q$ is an acyclic quiver, the set of positive $c$-vectors coincides with the set of non-initial $d$-vectors.

\begin{Prop}\label{prop_c_vector_from_triangulation}
Let $Q$ be an affine quiver of type $A$ with unique source and sink. Then the $c$-vectors have one of the forms below:
\begin{gather*}
g\left(\a_s+\sum_{i=1}^u \a_{p_i} + \sum_{i=1}^v \a_{q_i}\right) + (g-1)\left(\sum_{i=u+1}^k \a_{p_i}+\sum_{i=v+1}^\ell \a_{q_i} +\a_t \right), \qquad \text{or}\\
 (g-1)\left(\a_s+\sum_{i=1}^u \a_{p_i} +\sum_{i=1}^v \a_{q_i}\right) + g\left(\sum_{i=u+1}^k \a_{p_i}+\sum_{i=v+1}^\ell \a_{q_i} +\a_t\right),
\end{gather*}
where $g$ is any positive integer, $0\leq u\leq k$, and $0\leq v \leq \ell$.
\end{Prop}
\begin{proof}Let $\tau$ be an arc on $X$ with endpoints $z$ and $y$ such that $\tau$ is not in $T$ and it is not boundary parallel. Then $z$ and $y$ are on different boundaries. Let $\text{Int}_T(p)$ be the number of arcs in $T$ such that one of the boundary points is~$p$. Note that if $\tau $ circles around the inner boundary $g-1$ times then each arc in $T$ is intersected at least $g-1$ times. Thus we just need to look at the case where $\tau$ circles around once. There are three possible cases: $\text{Int}_T(z) > 1$ and $\text{Int}_T(y)>1$, $\text{Int}_T(z) =1$ and $\text{Int}_T(y)>1$, and $\text{Int}_T(z)=1$ and $\text{Int}_T(y)=1$.

If $\text{Int}_T(z)>1$ and $\text{Int}_T(y) >1$, then $\tau$ is isopotic to either curves:

\begin{center}
\includegraphics[trim =20mm 118mm 70mm 138mm, clip, width = .8\linewidth ]{affine_a.pdf}
\end{center}

where the red arcs are corresponding to the path $P$ and green correspond to the path $R$.

If $\text{Int}_T(z) =1 $ and $\text{Int}_T(y)>1$, then it is isopotic to either curves:
\begin{center}
\includegraphics[trim =20mm 170mm 70mm 85mm, clip, width = .8\linewidth ]{affine_a.pdf}
\end{center}

If $\text{Int}_T(z) =1 $ and $\text{Int}_T(y)=1$, then it is isopotic to either curves:
\begin{center}
\includegraphics[trim =20mm 63mm 70mm 193mm, clip, width = .8\linewidth ]{affine_a.pdf}
\end{center}
\end{proof}

\begin{Rmk}\label{rmk_subquiver_affine_a}
Let $Q$ be an affine quiver of type $A$. Any connected full subquiver~$Q'$ such that~$Q'$ is not equivalent to~$Q$ is a quiver of type $A$. Given $\pi \in P_Q$, let $\pi' = \varphi(\pi)$. By Proposition~\ref{prop_a_non_dec}, $\{\a_{\pi'}(\g)\,|\,\g \in \G_{\pi',nd}\} = \D^+_{Q'}$. Then by Lemma~\ref{lem_subquiver}, $\D^+_{Q'} \subseteq \{\a_\pi(\g)\,|\,\g \in \G_{\pi,nd}\}$.
\end{Rmk}

\begin{Prop} \label{prop_oriened_path}
Let $Q$ be an affine quiver of type $A$ with a unique source and a unique sink. Let $\a$ be a positive root of the form \[g\left(\a_s+\sum_{i=1}^u \a_{p_i} +\sum_{i=1}^v \a_{q_i}\right) + (g-1)\left(\sum_{i=u+1}^k \a_{p_i}+\sum_{i=v+1}^\ell \a_{q_i} +\a_t\right) \] or \[(g-1)\left(\a_s+\sum_{i=1}^u \a_{p_i} +\sum_{i=1}^v \a_{q_i}\right) + g\left(\sum_{i=u+1}^k \a_{p_i}+\sum_{i=v+1}^\ell \a_{q_i} +\a_t\right), \] where $g \in \Z_{\geq 1}$.
 Then $\a \in \{\a_\pi(\g)\,|\,\g \in \G_{\pi,nd}\}$.
\end{Prop}
\begin{proof}
 We proceed by induction on $g$. If $g =1$, then $\a$ is of a real Schur root of type $A$. Thus $\a \in \{\a_\pi(\g)\,|\,\g \in \G_{\pi,nd}\}$ by Remark~\ref{rmk_subquiver_affine_a}.
 Assume for all $a<g$, the roots of the forms
 \[ a\left(\a_s+\sum_{i=1}^u \a_{p_i} +\sum_{i=1}^v \a_{q_i}\right) + (a-1)\left(\sum_{i=u+1}^k \a_{p_i}+\sum_{i=v+1}^\ell \a_{q_i} +\a_t\right)\] and \[(a-1)\left(\a_s+\sum_{i=1}^u \a_{p_i} +\sum_{i=1}^v \a_{q_i}\right) + a\left(\sum_{i=u+1}^k \a_{p_i}+\sum_{i=v+1}^\ell \a_{q_i} +\a_t\right)\] are in $\{\a_\pi(\g)\,|\,\g \in \G_{\pi,nd}\}$.

Let $\pi = \left(\begin{smallmatrix}1 &2&\cdots & k+1 &k+2&\cdots & k+\ell+1& k+\ell+2\\ s&p_1&\cdots & p_k &q_1&\cdots &q_\ell &t \end{smallmatrix}\right)$. Note that $\pi \in P_Q$. Then
\begin{gather*}
c_\pi^{-1} \left( (g-1)\left(\a_s+\sum_{i=1}^u \a_{p_i} +\sum_{i=1}^v \a_{q_i}\right) + g\left(\sum_{i=u+1}^k \a_{p_i}+\sum_{i=v+1}^\ell \a_{q_i} +\a_t\right)\right) \\
\qquad{} = (g-1)\left(\a_s+\sum_{i=1}^k \a_{p_i} +\sum_{i=1}^\ell \a_{q_i}\right) +(g-2)\a_t\end{gather*}
and
\begin{gather*} c_\pi \left( g\left(\a_s+\sum_{i=1}^u \a_{p_i} +\sum_{i=1}^v \a_{q_i}\right) + (g-1)\left(\sum_{i=u+1}^k \a_{p_i}+\sum_{i=v+1}^\ell \a_{q_i} +\a_t\right)\right) \\
\qquad{} =(g-2)\a_s + (g-1)\left(\sum_{i=1}^k \a_{p_i} +\sum_{i=1}^\ell \a_{q_i}+\a_t\right).
\end{gather*}

As $c_\pi^{-1} \a <_D \a$ and $c_\pi^{-1}\a \in \{\a_\pi(\g)\,|\,\g\in \G_{\pi,nd}\}$ or $c_\pi\a<_D\a$ and $c_\pi\a \in \{\a_\pi(\g)\,|\,\g\in \G_{\pi,nd}\}$, by Lemma~\ref{lem_cox_transf}, the root $\a$ is $\{\a_\pi(\g)\,|\,\g\in \G_{\pi,nd}\}$.
\end{proof}

\begin{Prop}Let $Q$ be an affine quiver of type $A$ with a unique source and a unique sink. Then $\bigcup_{\pi\in P_Q}\{\a_\pi(\g)\,|\,\g\in \G_{\pi,nd}\}$ coincides with the set of real Schur roots.
\end{Prop}
\begin{proof}By Propositions~\ref{prop_c_vector_from_triangulation} and~\ref{prop_oriened_path}, this proposition holds.
\end{proof}

\subsection*{Acknowledgments}
The author would like to thank Kyungyong Lee for guidance and helpful discussions, Son Nyguen for helpful suggestions, and the referees for numerous helpful comments.

\pdfbookmark[1]{References}{ref}
\LastPageEnding

\end{document}